\documentclass[12pt]{article}
\usepackage{graphicx,color}
\usepackage{amsmath,amssymb,amsthm,amsfonts,amscd,amsxtra}
\usepackage{verbatim}
\usepackage{mathrsfs}
\usepackage{esint}
\usepackage{tikz}
\usepackage{enumitem}  
\usepackage[colorlinks=true, citecolor=blue, linkcolor=blue, urlcolor=green]{hyperref}
\newcommand{\ud}{\mathrm{d}}

\newcommand{\supp}{\mathrm{supp}}

\newtheorem{theorem}{Theorem}[section]

\newtheorem{lemma}[theorem]{Lemma}
\newtheorem{proposition}[theorem]{Proposition}

\theoremstyle{definition}
\newtheorem{defn}{Definition}[section]
\newtheorem{remark}{Remark}[section]

\numberwithin{equation}{section}

\DeclareMathSymbol{\C}{\mathalpha}{AMSb}{"43}

\newcommand{\bsub}{\begin{subequations}}
	\newcommand{\esub}{\end{subequations}$\!$}
\begin{document}
\title{Adams-Trudinger-Moser inequalities of Adimurthi-Druet type regulated by the vanishing phenomenon and its extremals}
	\author{
		Abiel Costa Macedo
		\\{\small Instituto de Matem\'{a}tica e Estat\'istica}
		\\{\small Universidade Federal de Goi\'as}\\
		{\small 74001-970 Goi\^ania, GO, Brasil}\\
		{\small abielcosta@ufg.br} 
		\and
		Jos\'{e} Francisco de Oliveira
		\\{\small Departamento de Matem\'{a}tica}
		\\{\small Universidade Federal do Piau\'{i}}\\
		{\small 64049-550 Teresina, PI, Brasil}\\
		{\small jfoliveira@ufpi.edu.br}
		\and 
		F\'abio Sodr\'e Rocha
		\\{\small Instituto de Matem\'{a}tica e Estat\'istica}
		\\{\small Universidade Federal de Goi\'as}\\
		{\small 74001-970 Goi\^ania, GO, Brasil}\\
		{\small fabiosodremat@gmail.com} 
	}

	\date{}
	
	\maketitle
	\begin{abstract}
Let  $W^{m,\frac{n}{m}}(\mathbb{R}^n)$ with $1\le m < n$ be the standard higher order derivative Sobolev space in the critical exponential growth threshold. We  investigate a new Adams-Adimurthi-Druet type inequality on the whole space $\mathbb{R}^n$ which is strongly influenced by the vanishing phenomenon. Specifically, we prove
	\begin{equation}\nonumber
		\sup_{\underset{\|\nabla^{m} u\|_{\frac{n}{m}}^{^{\frac{n}{m}}}+\|u\|_{\frac{n}{m}}^{\frac{n}{m}} \leq 1}{u\in W^{m,\frac{n}{m}}(\mathbb{R}^n)}} \int_{\mathbb{R}^n}\Phi\left(\beta \left(\frac{1+\alpha\|u\|_{\frac{n}{m}}^{\frac{n}{m}}}{1-\gamma\alpha\|u\|_{\frac{n}{m}}^{\frac{n}{m}}}\right)^{\frac{m}{n-m}}|u|^{\frac{n}{n-m}}\right) \mathrm{d}x<+\infty.
	\end{equation}
	where  $0\le \alpha<1$,  $0<\gamma<\frac{1}{\alpha}-1$ for $\alpha>0$,  $\nabla^{m} u$ is the  $m$-th order gradient for $u$, $0\le\beta\le \beta_0$, with $\beta_0$ being the  Adams critical constant, and	$\Phi(t) = \operatorname{e}^{t}-\sum_{j=0}^{j_{m,n}-2}\frac{t^{j}}{j!}$ with $j_{m,n}=\min\{j\in\mathbb{N}\;:\: j\ge n/m\}$. In addition, we prove that the constant $\beta_0$ is sharp. 
 In the subcritical case $\beta<\beta_0$,  the existence  and non-existence of  extremal function  are investigated for  $n=2m$  and attainability is proven for  $n=4$ and $m=2$ in the critical case $\beta=\beta_0$. Our method to analyze the  extremal problem is based on  blow-up analysis, a truncation argument recently introduced by  DelaTorre-Mancini \cite{DelaTorre} and some ideas by Chen-Lu-Zhu \cite{luluzhu20}, who studied the critical Adams inequality in $\mathbb{R}^4$.
	\end{abstract}
	\vskip 0.2truein

	

	%
	%
	\tableofcontents
	\section{Introduction}
	Let $\Omega\subset \mathbb{R}^n (n\geq 2)$ be a smooth bounded  domain and let $W^{m,p}_0(\Omega), p\ge 1$ be the $m$-\textit{th} order derivative classical Sobolev space on $\Omega$. Under the strict condition $n>mp$,  it is well known the optimal continuous  embedding of $W^{m,p}_0(\Omega)$ into Orlicz space $L_{\varphi_{*}}(\Omega)$ defined  by the Young  function $\varphi_{*}(t)=t^{p^*}$ $t\ge 0$, where $p^*=np/(n-mp)$ is  critical  Sobolev exponent. The optimality here means that the embedding is no longer holds if we replace the function $\varphi_*$ with any other  Young function $\psi$ that grows strictly more rapidly than $\varphi_{*}$. This type of scenario leads to a breakdown of compactness, giving rise to interesting questions that have been the subject of investigation by several authors over time, see for instance \cite{BN} and the  subsequent citations thereof. For the limiting case $n=mp$,  the threshold growth is not achieved by any power-type function 
$\varphi_{p}(t)=t^{p}$, $t\ge 0$; instead, it is determined by exponential growth.  The pioneering works  are due to  V.~Yudovich \cite{YUDO61}, S.~Pohozaev \cite{P0H065},  J. Peetre \cite{Peetre}, N. Trudinger \cite{Trudinger67} the proved the admissibility of the exponential growth in the first-order derivative case $m=1$.  The optmality of exponential growth was obtained by Hempel-Morris-Trudinger \cite{HMT} in which they also observe the influence of this growth in the loss of compactness.  Finally, Moser \cite{Moser1970/71} proved an improved version of these results, which is now known as the Trudinger-Moser inequality, and the extension of Moser's result for higher order derivatives $m\ge 2$ is due to Adams \cite{Adams1988}. Adams-Moser-Trudinger inequality and its related extremal problem has a broad range of applications in partial differential equations and geometric analysis, see for instance \cite{Baird1,Baird2}, and there are a lot of extensions and generalizations, among which we point out the works  \cite{AdimurthiDruet2004, Tinta, Yang} for bounded domains and \cite{adachitanaka2000,Cao,DoO,Ruf2005, RufLi2008,RufSani2013,lamlu2012} for  unbounded domains in the Euclidean space.

This work aims to present extensions of the Adams-Trudinger-Moser inequality of the Adimurthi-Druet type \cite{AdimurthiDruet2004} for either bounded or unbounded domains and to investigate their extremal problem. To precisely situate our developments, we will describe below some related advances on this topic.

\subsection{Trudinger-Moser type inequalities: The case first-order deri\-vatives}
 For a smooth  bounded domain $\Omega\subset \mathbb{R}^n (n\geq 2)$ Moser obtained the following sharp estimate
	\begin{equation}\label{PTM1}
		\sup_{u\in  W^{1,n}_0(\Omega),\; \|\nabla u\|_n=1} \int_{\Omega}\operatorname{e}^{\alpha
			|u|^{\frac{n}{n-1}}}\ \ud x 
			\le C_n|\Omega|  \quad \mbox{for} \quad \alpha \leq \alpha_n,
	\end{equation}
	where $\alpha_n:=n \omega_{n-1}^{1/(n-1)}$ and $\omega_{n-1}$ is the area of the surface of the unit $n$-ball in $\mathbb{R}^n$, $|\Omega|$ denotes the $n$-dimensional Lebesgue measure of $\Omega$, and   $\|\cdot\|_n$  is the standard norm in the Lebesgue space $L^n(\Omega)$. Moser also showed that $\alpha_n$ is sharp, i.e., the supremum \eqref{PTM1} is $+\infty$ if $\alpha>\alpha_n$. Nevertheless, P.-L.~Lions~\cite{Lions1984I} (see also \cite{CernyCianchiHencl13}) was able to prove that the exponent $\alpha_n$ can be improved along certain sequences. In fact, if $(u_i)\subset  W^{1,n}_0(\Omega)$ with $\|\nabla u_i\|_n=1$ and $u_i\rightharpoonup u_0$ in $ W^{1,n}_0(\Omega)$, then
	$$
	\sup_{i} \int_{\Omega} \operatorname{e}^{\gamma \alpha_n |u_i|^{\frac{n}{n-1}}}\ \ud x <\infty, \quad \mbox{for any} \quad \gamma<P:= (1-\|\nabla u_0\|_n^n)^{-\frac{1}{n-1}}.
	$$
 Note that, unless $u_0=0$, we have  $P>1$. Motivated by Lions' result, Adimurthi-Druet \cite{AdimurthiDruet2004} for $n=2$ and Yang  \cite{Yang} for $n\ge 3$  proposed an improved version of \eqref{PTM1} that provides additional information even in the case where $u_0=0$. Precisely, 
 \begin{equation}\label{ADTM}
		\sup_{u\in  W^{1,n}_0(\Omega),\; \|\nabla u\|_n=1} \int_{\Omega} \operatorname{e}^{\alpha_n(1+\alpha\|u\|^{n}_n)^{\frac{1}{n-1}}
			|u|^{\frac{n}{n-1}}}\ \ud x 
	\end{equation}
is finite for any  $0\le \alpha<\lambda_1(\Omega)$, and the supremum is infinity for any $\alpha\ge \lambda_1(\Omega)$, where $\lambda_1(\Omega)$ represents the eigenvalue associated with $n$-Laplacian. For extensions of \eqref{ADTM}, still within the context of first-order derivatives and without any claim of completeness, we recommend \cite{JJ2015, LuYang, Tinta,YangS, YangZhu, LiLuZhu2018}  and the references therein.
 
 When $\Omega=\infty$, the inequality \eqref{PTM1} is meaningless. Nevertheless, in works Cao \cite{Cao},  Do \'{O} \cite{DoO}, Panda \cite{Panda}, and  Adachi-Tanaka \cite{adachitanaka2000}, the following Trudinger-Moser type estimate in the scaling invariant form was obtained: For $\alpha\in (0, \alpha_n)$ there exists a constant $ C_{\alpha,n}$ depending only on $\alpha$ and $n$ such that
	\begin{equation}\label{adtaine}
		\sup_{u\in  W^{1,n}(\mathbb{R}^n),\; \|\nabla u\|_n\le 1}\frac{1}{\|u\|^n_{n}}\int_{\mathbb{R}^n} \phi_{n }\left(\alpha |u|^{\frac{n}{n-1}}\right)\ \ud x \le  C_{\alpha,n}
	\end{equation}
 where 
 $$
	\phi_{n}(t) =\operatorname{e}^{t}-\sum_{j=0}^{n-2}\frac{t^{j}}{j!}.
	$$
 In addition, different from the usual Trudinger-Moser inequality \eqref{PTM1} in bounded domains,  \eqref{adtaine} with $\alpha\ge \alpha_n$ is false which excludes $\alpha=\alpha_n$.  For this reason,  the inequality \eqref{adtaine} is currently known
  as subcritical Trudinger-Moser inequality in $\mathbb{R}^n$. To achieve
the critical case, Li-Ruf \cite{RufLi2008} and Ruf \cite{Ruf2005} replace the Dirichlet gradient norm by the full Sobolev norm  $\|u\|_{W^{1, n}\left(\mathbb{R}^n\right)}=\left(\|\nabla u\|_n^n+\|u\|_n^n\right)^{1 / n}$ to obtain the following sharp critical Trudinger-Moser inequality in $\mathbb{R}^n$
\begin{equation}\label{Li-RufTM}
\sup _{u \in W^{1, n}\left(\mathbb{R}^n\right),\;\|u\|_{W^{1, n}\left(\mathbb{R}^n\right)} \leq 1} \int_{\mathbb{R}^N} \phi_n\left(\alpha|u|^{\frac{n}{n-1}}\right) \mathrm{d} x<\infty,\;\;\mbox{for all}\;\; \alpha\le \alpha_n.
\end{equation}
Moreover,  if $\alpha>\alpha_n$ then the supremum in \eqref{Li-RufTM} is infinite. Inequalities of types \eqref{adtaine} and \eqref{Li-RufTM}  have a long history of investigation by various authors. We recommend \cite{Ibra, LLZ2} for a more in-depth discussion, and particularly the recent works \cite{CST,LLZ}, which show that these inequalities are indeed equivalent.

In the context of first-order derivatives, Trudinger-Moser inequalities of the Adimurthi-Druet type on $\mathbb{R}^n$ were developed by Do \'{O}-de Souza \cite{DoOMana1, DoOMana2} (see also \cite{NguyenVanHoang2019}), who established the following version of \eqref{ADTM} for the whole space
\begin{equation}\label{ADTfull}
	\mathrm{MT}(n, \beta, \alpha)=\sup _{u \in W^{1, n}\left(\mathbb{R}^n\right),\;\|u\|_{W^{1, n}\left(\mathbb{R}^n\right)} \leq 1} \int_{\mathbb{R}^n} \phi_n\left(\beta\left(1+\alpha\|u\|_n^n\right)^{\frac{1}{n-1}}|u|^{\frac{n}{n-1}}\right) d x,
\end{equation}
which is finite for any $\beta \leq \alpha_n$ and $0 \leq \alpha < 1$.

In this paper, we provide extensions of inequalities \eqref{ADTM} and \eqref{ADTfull} for higher-order derivatives, which are detailed  in Theorem~\ref{ThmAdiDruetUnbound}  below.

	\subsection{Adams-Trudinger-Moser type inequalities: The case higher-order derivatives}
For any positive integer  $m<n$ and  $u\in C^\infty_0(\Omega)$, we set  
	\begin{equation}\nonumber
		\nabla^m u=\left\{
		\begin{aligned}
			&\Delta^{m/2}u, & \mbox{if} &\;\; m\;\; \mbox{is even}\\
			&\nabla\Delta^{(m-1)/2}u, &\mbox{if}& \;\; m\;\;\mbox{is odd.}
		\end{aligned}\right.
	\end{equation}
For $|\Omega|<\infty$,	Adams 1988 in \cite{Adams1988} proved that
	\begin{equation}\label{adams}
		\sup_{{u\in W_0^{m,\frac{n}{m}}(\Omega),\; \left\|\nabla^m u\right\|_\frac{n}{m}\leq 1}}\int_{\Omega}\operatorname{e}^{\beta |u|^{\frac{n}{n-m}}} \ud x < \infty,\;\; \mbox{if and only if}\;\; \beta \leq \beta_0,
	\end{equation}
	where
	\begin{equation}\label{Adams-best-constant}
		\beta_0=\beta_0(m,n)=
		\begin{cases}
			\frac{n}{\omega_{n-1}}\left[ \frac{\pi^{\frac{n}{2}}2^m\Gamma\left(\frac{m+1}{2}\right)}{\Gamma\left(\frac{n-m+1}{2}\right)}\right]^{\frac{n}{n-m}}, & \mbox{if}\;\; m \mbox{ is odd, } \\
			\frac{n}{\omega_{n-1}}\left[ \frac{\pi^{\frac{n}{2}}2^m\Gamma\left(\frac{m}{2}\right)}{\Gamma\left(\frac{n-m}{2}\right)}\right]^{\frac{n}{n-m}}, &\mbox{if}\;\; m \mbox{ is even, }
		\end{cases}
	\end{equation}
	in which  $\Gamma(x)=\int_0^{1} (-\ln{t})^{x-1} \, \ud t,\, x>0$
	is the gamma Euler function.  Inequality \eqref{adams} is the extension for higher order derivatives of the Moser  inequality \eqref{PTM1} and it is currently known as Adams inequality or  Adams-Moser-Trudinger inequality.	By considering was the more large space given by
	$$
	W^{m,p}_\mathcal{N} (\Omega):=\{ u \in W^{m,p}(\Omega): u_{|_{\partial \Omega}}=\Delta^j u_{|_{\partial \Omega}}=0 \mbox{ in the sense of trace}, 1\leq j< m/2 \},
	$$
	  C.~Tarsi \cite[Theorem 4]{Tarsi2012} was able to extend \eqref{adams} to the following 
	  \begin{equation}\label{TarsiThm}
	  \sup_{{u\in W_\mathcal{N}^{m,\frac{n}{m}}(\Omega),\; \|\nabla^m u\|_{\frac{n}{m}}\leq 1}} \int_{\Omega}\operatorname{e}^{\beta |u|^{\frac{n}{n-m}}} \ \ud x\leq C_{m,n}|\Omega|, \quad \forall \; 0 \leq \beta \leq\beta_0,
	  \end{equation}
		for some constant $C_{m,n}>0$. In addition, $\beta_0$ is also sharp, i.e., the supremum above is $+\infty$ if $\beta>\beta_0$.
		
	Extensions of Adams inequality for the entire space $\mathbb{R}^n$, where first considered by Ozawa \cite{Ozawa1995} who obtained an extension for $W^{m, \frac{n}{m}}\left(\mathbb{R}^n\right)$  by using the restriction $\left\|\Delta^{\frac{m}{2}} u\right\|_{\frac{m}{m}} \leq 1$. However, with the argument in \cite{Ozawa1995}, one cannot obtain the best possible exponent $\beta$ for this type of inequality. Recently Ruf-Sani \cite{RufSani2013} when $m$ is an even  integer number and Lam-Lu   \cite{lamlu2012, LamLuMAA2012}  for $m$ odd integer number have obtained a version of the Adams inequality \eqref{adams} for  any  domain  $\Omega\subset\mathbb{R}^n$ not necessarily bounded. In fact, there hold
	\begin{equation}\label{ruf-sani}
			\sup_{u\in W_0^{m,{\frac{n}{m}}}(\Omega),\; \|(-\Delta+I)^k u\|_{\frac{n}{m}}\leq 1} \int_{\Omega} \Phi({\beta_0
				|u|^{\frac{n}{n-m}}})\ \ud x\le C_{m,n},\;\;\mbox{if}\;\; m=2k
		\end{equation}
		and 
		\begin{equation}\label{lam-lu}
			\sup_{u\in W_0^{m,{\frac{n}{m}}}(\Omega),\; \|\nabla (-\Delta+I)^k u\|_{\frac{n}{m}}^{\frac{n}{m}}+\|(-\Delta+I)^k u\|_{\frac{n}{m}}^{\frac{n}{m}}\leq 1} \int_{\Omega} \Phi({\beta_0
				|u|^{\frac{n}{n-m}}})\ \ud x\le C_{m,n},\;\;\mbox{if}\;\; m=2k+1
		\end{equation}
where 
	\begin{equation}\label{phi}
		\Phi(t)=\operatorname{e}^t - \sum_{j=0}^{j_{m,n}-2} \frac{t^j}{j!}, \qquad j_{m,n}:=\min\{j\in \mathbb{N}\, : \, j\geq \frac{n}{m}\}.
	\end{equation}
	Moreover, the constant $\beta_0$ is sharp for \eqref{ruf-sani} and \eqref{lam-lu}.

	 On Sobolev spaces $W^{\gamma,\frac{n}{\gamma}}\left(\mathbb{R}^n\right)$ of arbitrary positive fractional order $\gamma<n$,  by  using a rearrangement-free argument  Lam-Lu \cite{lamlufree}  established  the following Adams inequality 
	$$
			\sup_{{u\in 	W^{\gamma,p}(\mathbb{R}^{n})},\; \|\tau(I-\Delta^\frac{\gamma}{2}) u \|_{p} \leq 1}\int_{\mathbb{R}^{n}}\Phi(\beta_{0}(n,\gamma) |u(x)|^{\frac{p}{p-1}}) \ud x < \infty,\;\; \beta_0(n, \gamma) =\frac{n}{\omega_{n-1}}\left[\frac{\pi^{\frac{n}{2}} 2^\gamma \Gamma(\frac{\gamma}{2})}{\Gamma\left(\frac{n-\gamma}{2}\right)}\right]^{\frac{p}{p-1}}
		$$
		where $p=\frac{n}{\gamma}$,  $\tau>0$ and $\Phi$ is the same one in \eqref{phi} with $ j_{p}=\min\{j\in \mathbb{N}\, : \, j\geq p\}$ instead of  $j_{m,n}$. Further, if $\beta_0(n, \gamma)$ is replaced by any $\beta>\beta_0(n, \gamma)$, then the supremum is infinite.	Another extension of \eqref{ADTM} for whole space $\mathbb{R}^n$, is the subcritical Adams inequality established by  Lam-Lu \cite{lamlufree} for  $m=2$, Fontana-Morpurgo \cite{FontanaMorpurgo2018} and Lam-Lu-Zhang \cite{LLZ} to all $m>2$ is the following
	\begin{align*}
		\sup_{u\in 	W^{m,\frac{n}{m}}(\mathbb{R}^{n}),\; \|\nabla^{m} u \|_\frac{n}{m}^{\frac{n}{m}} + \|u\|_{\frac{n}{m}}^{\frac{n}{m}}\leq 1 }\int_{\mathbb{R}^{n}}\Phi(\beta |u(x)|^{\frac{n}{n-m}}) \ud x 
		\begin{cases}
			\leq C_{m,n}, ~~ &\text{ if } \beta \leq \beta_{0}(m,n) \\
			= \infty, ~~&\text{ if } \beta > \beta_0(m,n).
		\end{cases} 
	\end{align*}
Finally, we would like to  mention the following sharpened Trudinger-Moser inequality  in $\mathbb{R}^2$ with exact growth condition due to Ibrahim-Masmoudi-Nakanishi \cite{Ibra}: 
there exists a constant $C > 0$ such that
\begin{equation}\label{IbraIne}
\int_{\mathbb{R}^2} \frac{e^{4\pi u^2} - 1}{(1 + |u|)^2} \mathrm{d}x \leq C \|u\|_{L^2(\mathbb{R}^2)}^2, \quad \forall u \in W^{1,2}(\mathbb{R}^2) \text{ with } \|\nabla u\|_{L^2(\mathbb{R}^2)} \leq 1.
\tag{1.4}
\end{equation}
Moreover, this fails if the power $2$ in the denominator is replaced with any $p < 2$.  In \cite{MasmoudiSani2014} Masmound-Sani  extended \eqref{IbraIne}  to  the  corresponding second-order Adams' inequality  in $W^{2,2}(\mathbb{R}^4)$ and then for arbitrary dimensions $n\ge 2$ in  \cite{MasmoudiSani2015}. After,  Lu-Tang in \cite{LuTang} were able to provide an extension to the framework of hyperbolic space. For a more in-depth discussion and improvements of \eqref{IbraIne} we recommend \cite{LamLuChapter,LLZ,LLZ2} and the references therein.

Finally, we would like to emphasize that, despite the large number of extensions of \eqref{adams}, except for recent work in \cite{DelaTorre,LuYangEP} for $n=2m$ and bounded domains, to our knowledge, no other Adimurthi-Druet type extension \eqref{ADTM} has been established for $m \ge 2$, for either bounded or unbounded domains $\Omega \subset \mathbb{R}^n$. In this article, we aim to fill this gap by presenting such extensions, which differ from the usual ones (even for $m=1$), as our proposed inequality is strongly governed by the \textit{vanishing} phenomenon.

\subsection{Maximizers: the extremal problem}
Since the establishment of the Adams-Trudinger-Moser inequality \eqref{PTM1}-\eqref{adams}, a natural question arises:  \textit{does there exist} $u\in W_0^{m,\frac{n}{m}}(\Omega)$ \textit{with} $\left\|\nabla^m u\right\|_\frac{n}{m}=1$  \textit{such that }
\begin{equation}\label{EP}
 \int_{\Omega} e^{\beta_0
			|u|^{\frac{n}{n-m}}}\ \ud x = \sup_{{u\in W_0^{m,\frac{n}{m}}(\Omega),\; \left\|\nabla^m u\right\|_\frac{n}{m}\leq 1}}\int_{\Omega}e^{\beta |u|^{\frac{n}{n-m}}} \ud x \,?
\end{equation}
This is the famous extremal problem for the Adams-Trudinger-Moser inequality under the critical condition $\beta = \beta_0$, which has been studied by various authors over the years. For the first order derivative case $m=1$ it was completely solved through the works \cite{Carleson-Chang,Struwe1988,Flucher1992,Lin1996,DelaTorre}. In fact,    Carleson-Chang  \cite{Carleson-Chang} solved it when $\Omega=B_1(0)\subset\mathbb{R}^n, n\ge 2$ is the unit ball,  M. Struwe \cite{Struwe1988} used blow-up analysis to ensure  extremal functions for a class of nonsymmetric domains in $\mathbb{R}^2$,  Flucher \cite{Flucher1992} applied  conformal rearrangement to derived an isoperimetric inequality which implies the existence of extremal functions to any smooth bounded domain in $\mathbb{R}^2$.  Finally, Lin \cite{Lin1996} ensures extremal function to any smooth bounded domain in $\mathbb{R}^n,n\geq 2$.  In contrast to the case  $m=1$,  extremal problem \eqref{EP} for $m\ge 2$ has been solved only for some particular cases.  We can only  mention Lu-Yang \cite{LuYangEP}, which proved the existence of extremals in the case $m = 2$ with $\Omega \subset \mathbb{R}^4$  and more recently DelaTorre-Mancini \cite{DelaTorre}, where the existence of extremals for the case \( H_0^m(\Omega) \) with $ \Omega \subset \mathbb{R}^{2m}$ is proved. We also draw attention to the work \cite{deOliveiraMacedo2022}, where a sharp estimate for the concentration levels of the Adams-Trudinger-Moser functional is obtained.

The existence of extremal function for the Adimurthi-Druet inequality \eqref{ADTM} it was obtained by Yang in \cite{Yang2006} for $n=2$ and \cite{Yang} for $n\ge 3$.  On the other hand, the unbounded situation $\Omega=\mathbb{R}^n$ for the either subcritical \eqref{adtaine}  or critical \eqref{Li-RufTM} the corresponding  extremal problems were  investigated in \cite{Ishiwata2011, RufLi2008, Ruf2005} and in the recent work \cite{LLZ2}. Further, the extremal for the Adimurthi-Druet type supremum \eqref{ADTfull} was recently studied in \cite{NguyenVanHoang2019}.  It is worth mentioning that, in the unbounded situation,   even for $\alpha<\alpha_n$ attainability is nontrivial due to the loss of compactness,  as observed by Ishiwata \cite{Ishiwata2011} (see also \cite{Lions1984I, Lions1984II}).  To illustrate the situation, let $(u_j)$ with $\|u_j\|_{W^{1,n}(\mathbb{R}^n)}=1$  and $u_j\rightharpoonup u$ weakly in $W^{1,n}(\mathbb{R}^n)$ be a maximizing sequence for the supremum 
$\mathrm{MT}(n, \beta,\alpha)$ in \eqref{ADTfull}. Then the compactness of $(u_j) $ follows if we can exclude both the \textit{concentration}  phenomenon
\begin{itemize}
\item [$(a)$] $u=0$ and $\displaystyle\lim_{j\rightarrow\infty}\int_{B^{c}_{\rho}}(|\nabla u_j|^{n}+|u_j|^n)dx=0$ for any $\rho>0$
\end{itemize}
and the   \textit{vanishing} phenomenon
\begin{itemize}
\item [$(b)$] $\displaystyle\lim_{j\to\infty}\|\nabla u_j\|_{n}=0$ and $\displaystyle\limsup_{j\to\infty}\| u_j\|_{n}>0$.
\end{itemize}
In most works on the existence of extremals a lower bound for the supremum is  established which, after sharp estimates for both  concentrating  and vanishing levels, ensures the compactness of the any maximizing sequence,  see for instance 
\cite{Carleson-Chang, Ishiwata2011, Lin1996,RufLi2008, NguyenVanHoang2019, Ruf2005}.

	\subsection{Mains results: Adams-Adimurthi-Druet type inequalities and extremals}
Our first result is the following  sharp Adams-Adimurthi-Druet type inequalities on $\mathbb{R}^n$.
	\begin{theorem}\label{ThmAdiDruetUnbound}
		Let $n \geq 2$,  $0\leq\alpha <1$,  $0< \gamma <\frac{1}{\alpha}-1$ for $\alpha>0$, and $1\leq m<n$. Then
		\begin{equation}\label{supUnbounded}
			\sup_{\underset{\|\nabla^m u\|_{\frac{n}{m}}^{\frac{n}{m}} + \|u\|_{\frac{n}{m}}^{\frac{n}{m}} \leq 1}{u\in W^{m,\frac{n}{m}}(\mathbb{R}^n)}} \int_{\mathbb{R}^n}\Phi\left(\beta_0 \left(\frac{1+\alpha\|u\|_{\frac{n}{m}}^{\frac{n}{m}}}{1-\gamma\alpha\|u\|_{\frac{n}{m}}^{\frac{n}{m}}}\right)^{\frac{m}{n-m}}|u|^{\frac{n}{n-m}}\right) \mathrm{d}x   < \infty. 
		\end{equation}	
In addition,	 the constant $\beta_0$ is sharp. 
	\end{theorem}
In order to emphasize  the nature of the inequality \eqref{supUnbounded}, let us consider $(u_j)$ be a normalized vanishing type sequence in  $W^{m,\frac{n}{m}}(\mathbb{R}^n)$, that is, $$ \|\nabla^m u_j\|_{\frac{n}{m}}^{\frac{n}{m}} + \|u_j\|_{\frac{n}{m}}^{\frac{n}{m}} =1\;\;\mbox{and}\;\;  \lim_{j\to\infty}  \|\nabla^m u_j\|_{\frac{n}{m}} =0.$$   Hence, along this type of sequence, we observe that the exponent 
$$
\beta \left(\frac{1+\alpha\|u_j\|_{\frac{n}{m}}^{\frac{n}{m}}}{1-\gamma\alpha\|u_j\|_{\frac{n}{m}}^{\frac{n}{m}}}\right)^{\frac{m}{n-m}}
$$
converges to the constant $\beta(1+\alpha)/(1-\gamma\alpha)$ as $j\to \infty$, which can be arbitrarily large  if   $\alpha>0$ is close to $0$ and $\gamma$ is close to the upper limit $1/\alpha -1$.  It proves to be a challenge to the maximizing problem, since that in \eqref{supUnbounded}  the vanishing phenomenon implies in the lack of compactness for the associated functional even in the subcritical case $\beta<\beta_0$. From this reason, we  say that  our inequality  is strongly governed by the vanishing phenomenon.

	Let us denote 
	\begin{equation}\label{AD-notation}
		AD(n, m, \beta, \alpha, \gamma) := \sup_{\underset{\|\nabla^{m} u\|_{\frac{n}{m}}^{^{\frac{n}{m}}}+\|u\|_{\frac{n}{m}}^{\frac{n}{m}} \leq 1}{u\in W^{m,\frac{n}{m}}(\mathbb{R}^n)}} \int_{\mathbb{R}^n}\Phi\left(\beta \left(\frac{1+\alpha\|u\|_{\frac{n}{m}}^{\frac{n}{m}}}{1-\gamma\alpha\|u\|_{\frac{n}{m}}^{\frac{n}{m}}}\right)^{\frac{m}{n-m}}|u|^{\frac{n}{n-m}}\right) \mathrm{d}x.
	\end{equation}
Initially, we consider the existence and non-existence of extremal function  to $AD(n, m, \beta, \alpha,
	 \gamma)$ in the subcritical and critical case.  In this case, we  consider 
	\begin{equation}\label{ConstantGN}
		\mathcal{B}_{GN} := \underset{u \in W^{m,2}_{rad}(\mathbb{R}^{2m})}{\sup}\frac{\|u\|_{4}^{4}}{\|\nabla^{m}u\|_{2}^{2}\|u\|^{2}_{2}},
	\end{equation}
	which is the Gagliardo-Nirenberg constant in $W^{m,2}(\mathbb{R}^{2m})$.  We are able to prove the  following  existence results.
	\begin{theorem}\label{ThmAttain}
		Suppose $0 < \alpha < 1$, $n=2m$ and $0\leq \gamma <\min\{\frac{1}{\alpha}-1, \frac{m(1+\alpha)-1-2\alpha}{m\alpha^2+\alpha m - \alpha^2}\}$. Then   $AD(2m, m, \beta, \gamma, \alpha) $ is attained for any $\beta \in \left(\frac{ 1 +2\alpha  - \gamma\alpha^2}{1 + \alpha(1 - \gamma) - \gamma\alpha^2}\frac{2}{\mathcal{B}_{GN}}, \beta_{0}\right)$.
	\end{theorem}
	The next result shows that the attainability of Theorem~\ref{ThmAttain} is lost for small $\beta$, highlighting that the corresponding extremal problem in \eqref{supUnbounded} remains challenging even in the subcritical case.
	\begin{theorem}\label{ThmNotAttain}
		Let $0 < \alpha < 1$, $0\le\gamma<\frac{1}{\alpha}-1$  and $n = 2m$ for $\beta  \ll \beta_{0}$, then $AD(2m, m, \beta, \alpha, \gamma) $ is not attained.
	\end{theorem}
For the critical range $\beta = \beta_0$, we also establish the following result.
	\begin{theorem}\label{thm-attainn=4}
		Let $0 < \alpha < 1$ and $0\leq \gamma <\min\{\frac{1}{\alpha}-1, \frac{1}{\alpha^2+2\alpha}, \gamma_0\}$ for some $\gamma_0>0$, then there exists $\alpha_0\in (0,1)$ such that $AD(4, 2, \beta_0, \alpha, \gamma) $ is attained for any $0\le \alpha<\alpha_0$.
	\end{theorem}
The techniques employed in this paper, inspired by the work of Adimurthi-Druet \cite{AdimurthiDruet2004}, combine blow-up analysis with sharp functional inequalities, allowing us to capture the delicate balance between concentration and vanishing phenomena.  The technique of blow-up analysis used in this paper follows a similar approach to the one developed by Chen-Lu-Zhu \cite{luluzhu20}, who studied the critical Adams inequality in $\mathbb{R}^4$ and the recent work by DelaTorre-Mancini \cite{DelaTorre}. These results contribute to the growing body of work on the qualitative behavior of solutions to PDEs in the critical regime and open up new directions for further research in the analysis of nonlinear equations on unbounded domains.
	\section{Preliminary results}
	Let us denote
	\begin{equation}
		F_{n, m, \beta, \alpha, \gamma}(u) = \int_{\mathbb{R}^{n}} \Phi\left(\beta \left(\frac{1+\alpha \|u\|_{\frac{n}{m}}^{\frac{n}{m}}}{1-\gamma\alpha\|u\|_{\frac{n}{m}}^{\frac{n}{m}}}\right)|u|^{\frac{n}{n-m}}\right)\mathrm{d}x,\;\;\mbox{for any}\;\; u \in W^{m,2}(\mathbb{R}^{n}).
	\end{equation}
	Set $u^{\#} = \mathcal{F}^{-1}\{(\mathcal{F}(u))^{*}\}$,  where $\mathcal{F}$ is the Fourier transform on $\mathbb{R}^{2m}$ and  $\mathcal{F}^{-1}$ its inverse,   and $u^{*}$ denotes the Schwarz symmetrization of $u$. Using the properties of Fourier rearrangement \cite{Lenzmann} we can write
	\begin{align*}
		\begin{cases}
			\|\nabla^{m} u^{\#}\|_2 \leq \|\nabla^{m} u \|_2; \\
			\|u^{\#}\|_{2} = \|u\|_{2}, ~~ \|u^{\#}\|_{q} \geq \|u\|_{q}, (q>2).
		\end{cases}
	\end{align*}
	Denoting $\mathcal{H} := \{u \in W^{m,2}(\mathbb{R}^{2m}) \;:\;   \|u\|_{2}^{2}+ \|\nabla^{m} u\|_{2}^{2} = 1\}$. It follows that  
	\begin{equation}\label{Radializar}
	\underset{u \in \mathcal{H}}{\sup }F_{2m, m, \beta, \alpha, \gamma}(u) =  \underset{u \in \mathcal{H} \cap W^{m, 2}_{rad}(\mathbb{R}^{2m})}{\sup }F_{2m, m, \beta, \alpha, \gamma}(u).
	\end{equation}
	We also mention the following  Adams type inequality, which is analogous to the inequality proved by Adachi-Tanaka in \cite{adachitanaka2000}, in a scaling invariant form which will be useful for us to proof Theorem \ref{ThmNotAttain}. By Fontana-Morpurgo \cite{FontanaMorpurgo2018} we have the following Adachi-Tanaka type inequality.

	Let $n>m\geq 2$  be integers. Then given $\beta\in (0,\beta_0)$  there exists $C_{\beta,m,n}=C(\beta,m,n)$ depending only on $\beta, \ m$ and $n$ such that
	\begin{equation}\label{adachitanakaine}
		\int_{\mathbb{R}^n} \Phi\left({\beta  \left(\frac{|u|}{\|\nabla^m u\|_{\frac{n}{m}}}\right)^{\frac{n}{n-m}}}\right)\ \ud x \leq C_{\beta,m,n} \frac{\|u\|_{\frac{n}{m}}^{\frac{n}{m}}}{\|\nabla^m u\|_{\frac{n}{m}}^{\frac{n}{m}}}, \quad \forall\ u\in W^{m,{\frac{n}{m}}}(\mathbb{R}^n)\setminus \{0\},
	\end{equation}
	where  $\Phi$ was defined in \eqref{phi}, $\beta_0$ was given  in \eqref{Adams-best-constant}. Moreover, for $\beta\in [\beta_0,\infty)$  inequality \eqref{adachitanakaine} fail, i.e., there exists $(u_i)\subset W^{m,{n/m}}_{rad}(\mathbb{R}^n)$ such that
	\begin{equation}\label{toinfinit}
		\frac{\|\nabla^m u_i\|_{\frac{n}{m}}^{\frac{n}{m}}}{\|u_i\|_{\frac{n}{m}}^{\frac{n}{m}}}\int_{\mathbb{R}^n} \Phi\left({\beta  \left(\frac{|u_i|}{\|\nabla^m u_i\|_{\frac{n}{m}}}\right)^{\frac{n}{n-m}}}\right)\ \ud x\rightarrow \infty.
	\end{equation}
	
Now we state the following Radial Lemma that can be easily extended from \cite{kavian}, Lemma~1.1, Chapter~6, which will be useful in our analysis
	\begin{lemma}\label{radiallemma}
		If $u\in W^{1,\frac{n}{m}}_{rad}(\mathbb{R}^n)$, then
		$$
		|u(x)|\leq \left(\frac{1}{m\sigma_n}\right)^{\frac{m}{n}}\frac{1}{|x|^{(\frac{n-1}{n})m}} \|u\|_{W^{1,\frac{n}{m}}} \quad  \mbox{ a.e } x \in  \ \mathbb{R}^n, 
		$$
		where $\sigma_n$ is the volume of unit ball in $\mathbb{R}^{n}$.
	\end{lemma}
	\begin{remark}
		It is important  to mention that from the proof of \eqref{adachitanakaine} we can see that the constant  $C_{\beta,m,n}$ tend exponentially to infinite when $\beta$ tends to $\beta_0$.
	\end{remark}
	We'll  give an estimate to the constant $\mathcal{B}_{GN}$ in \eqref{GNestimates} for the even case, which will be important in our existence of extremals result. Indeed, let us consider the same function introduced by D. Adams in \cite{Adams1988}.
	Let $\varphi(t)\in C^{\infty}[0,1]$ such that
	\begin{align*}
		&\varphi(0)=\varphi'(0)=\cdots=\varphi^{(m-1)}(0)=0,\\
		&\varphi(1)=\varphi'(1)=1, \qquad \varphi''(1)=\varphi'''(1)=\cdots=\varphi^{(m-1)}(1)=0.
	\end{align*}
	For $0<\varepsilon<\tfrac{1}{2}$, we define
	$$
	H(t)=
	\begin{cases}
		\varepsilon\varphi(\frac{1}{\varepsilon} t), & \mbox{ if } \quad t\leq \varepsilon\\
		t, & \mbox{ if } \quad \varepsilon \leq t \leq 1-\varepsilon\\
		1-\varepsilon \varphi(\frac{1}{\varepsilon}(1-t)), & \mbox{ if } \quad 1-\varepsilon \leq t\leq 1\\
		1, & \mbox{ if } \quad 1\leq t,\\
	\end{cases}
	$$
	and 
	$$
	\psi_\lambda(r)=\left(\ln(\lambda)\right)^{1/2} H\left((\ln(\lambda))^{-{1}}\ln \frac{1}{r}\right).
	$$ 
	For all $\lambda>1$, $\psi_\lambda(|x|)$ is defined on $B_{1}(0)$ and can be extended for whole space $ W_{0,rad}^{m,2}(\mathbb{R}^{2m})$. More than that, $\psi_\lambda(|x|)=\left(\ln\lambda\right)^{1/2}$ for $|x|\leq 1/\lambda$ and, as proved by D. Adams in  \cite{Adams1988}, we have
	$$
	\|\nabla^m \psi_\lambda\|_{2}^{2}=(2m)^{-1}\beta_0 A_{\lambda,\varepsilon},
	$$
	where
	$$
	A_{\lambda,\varepsilon}\leq\left[ 1+ 2\varepsilon\left(\|\Phi'\|_\infty +O\left((\ln\lambda)^{-1}\right)\right)^{2}  \right].
	$$
	Now, computing explicitly $\|\psi_\lambda\|_{2}^{2}$, we obtain
	\begin{align*}
		\|\psi_\lambda\|_{2}^{2} &\leq \int_{B_{1}(0)}|\psi_\lambda(x)|^{2} \mathrm{d}x \\
		&\leq \int_{B_{1}(0)}|(\ln\lambda)^{1/2}|^{2}\mathrm{d}x	\\
		&\leq \int_{B_{1}(0)} \ln\lambda\mathrm{d}x \\
		&= \frac{\omega_{2m-1}}{2m} \ln\lambda
	\end{align*}
	Now, with the aim to  majorate the Gagliardo-Nirenberg constant in $W^{m, 2}(\mathbb{R}^{4}) $, we compute
	\begin{align*}
		\|\psi_\lambda\|_{4}^{4} &\geq\int_{B_{1/\lambda}}|\psi_\lambda(|x|)|^{4}\mathrm{d}x \\
		&=  2\ln_\lambda\frac{\omega_{2m-1}}{2m} \frac{1}{\lambda^{2m}}
	\end{align*}
	where, $\frac{\pi^{m}}{m!}\left(\frac{1}{\lambda}\right)^{m}$ is the volume of the $2m$-ball of radius $1/\lambda$. Moreover this, notice that $\|\psi_\lambda\|_{2}^{2} = \pi^{m}(m!)^{-1}\ln\lambda$
	Therefore
	\begin{align*}
		\mathcal{B}_{GN} &\geq \frac{\|\psi_\lambda\|^{4}_{4}}{\|\nabla^{m}\psi_\lambda\|_{2}^{2}\|\psi_\lambda\|_{2}^{2}} \\
		&=\frac{2\ln\lambda\frac{\omega_{2m-1}}{2m} \frac{1}{\lambda^{2m}}}{\left(\pi^{m}(m!)^{-1}\ln\lambda\right)\left((2m)^{-1}\beta_0 A_{\lambda,\varepsilon}\right)}  
	\end{align*}
	and $\omega_{2m-1} = \frac{2\pi^{m}}{(m-1)!}$, then 
	\begin{align*}
		\mathcal{B}_{GN} &\geq \frac{2m}{\lambda^{2m}\beta_{0}A_{\lambda,\varepsilon}}
	\end{align*}
	if we take limit as $\varepsilon \rightarrow 0$ , we got $\limsup A_{\lambda,\varepsilon} \leq 
	1$ and taking $\lambda \rightarrow 1^+$, we obtain
	\begin{align}\label{GNestimates}
		\mathcal{B}_{GN} \geq \frac{2m}{\beta_{0}}.
	\end{align}
	\section{Adams-Adimurthi-Druet type  inequality for entire space}
	In this section we will prove the Adams inequality of Adimurthi-Druet type  for whole space $\mathbb{R}^n$  stated in Theorem~\ref{ThmAdiDruetUnbound}. 
\subsection{Proof of Theorem \ref{ThmAdiDruetUnbound}}
 We will proceed following a scaling argument  in  \cite{NguyenVanHoang2019}.  For $\tau > 0$ and $u\in W^{m,\frac{n}{m}}(\mathbb{R}^n)$, by setting $u_{\tau}(x)=u(\tau^{\frac{1}{n}}x)$, we have 
 $$\|\nabla^{m}u_{\tau}(x)\|_{\frac{n}{m}}^{\frac{n}{m}} = \|\nabla^{m}u(x)\|_{\frac{n}{m}}^{\frac{n}{m}}\;\; \mbox{and} \;\; \|u_{\tau}(x)\|_{\frac{n}{m}}^{\frac{n}{m}} = \tau^{-1}\|u(x)\|_{\frac{n}{m}}^{\frac{n}{m}}.$$ 
 Hence
 \begin{equation}\label{CompareSup}
	\begin{aligned}
		K_{\tau} &:=  \sup_{\underset{\|\nabla^m u\|_{\frac{n}{m}}^{\frac{n}{m}} + \tau \|u\|_{\frac{n}{m}}^{\frac{n}{m}} \leq 1}{u\in W^{m,\frac{n}{m}}(\mathbb{R}^n)}} \int_{\mathbb{R}^n}\Phi\left(\beta_0 |u|^{\frac{n}{n-m}}\right) \mathrm{d}x \\
		&= \frac{1}{\tau} \sup_{\underset{\|\nabla^{m}u\|_{\frac{n}{m}}^{\frac{n}{m}}+\|u\|_{\frac{n}{m}}^{\frac{n}{m}} \leq 1}{u\in W^{m,\frac{n}{m}}(\mathbb{R}^n)}} \int_{\mathbb{R}^n}\Phi\left(\beta_0 |u|^{\frac{n}{n-m}}\right) \mathrm{d}x \\
		&= \frac{K_1}{\tau} < \infty,
	\end{aligned}	
	\end{equation}
which is finite by a result that can be found in \cite[Theorem~1-(b)]{FontanaLuigi2015SAaM}.  Since $\alpha<1$ and $\gamma<\frac{1}{\alpha}-1$, we can chose $0<\tau = 1 -\alpha<1$, $0<\mu = \tau -\gamma\alpha<1$. Thus, for $u \in W^{m,\frac{n}{m}}(\mathbb{R}^n)$ with $\|\nabla^m u\|_{\frac{n}{m}}^{\frac{n}{m}} + \|u\|_{\frac{n}{m}}^{\frac{n}{m}} \leq 1$, defining
	$$
	v := \frac{u}{(\|\nabla^m u\|_{\frac{n}{m}}^{\frac{n}{m}} + \tau \|u\|_{\frac{n}{m}}^{\frac{n}{m}})^{\frac{m}{n}}} \quad \mbox{and} \quad w := \frac{v}{(\|\nabla^m v\|_{\frac{n}{m}}^{\frac{n}{m}} + \mu \|v\|_{\frac{n}{m}}^{\frac{n}{m}})^{\frac{m}{n}}}
	$$
	 we have $\|\nabla^m v\|_{\frac{n}{m}}^{\frac{n}{m}} + \tau \|v\|_{\frac{n}{m}}^{\frac{n}{m}} = 1$ and $\|\nabla^m w\|_{\frac{n}{m}}^{\frac{n}{m}} + \mu \|w\|_{\frac{n}{m}}^{\frac{n}{m}} = 1$.
	Hence,  \eqref{CompareSup} yields
	\begin{align}\label{EstIntByConst}
		\int_{\mathbb{R}^n}\Phi\left(\beta_0 |w|^{\frac{n}{n-m}}\right) \mathrm{d}x\leq \frac{K_1}{\mu}.
	\end{align}
	Notice that  by the choice of $\tau$ and $\mu$, we get 
\begin{equation*}
		\begin{cases}
		|u|^{\frac{n}{n-m}} \leq \left(1-\alpha\|u\|_{\frac{n}{m}}^{\frac{n}{m}}\right)^{\frac{m}{n-m}}|v|^{\frac{n}{n-m}} \\
		\\
		|v|^{\frac{n}{n-m}} \leq \left(1-\gamma\alpha\|v\|_{\frac{n}{m}}^{\frac{n}{m}}\right)^{\frac{m}{n-m}}|w|^{\frac{n}{n-m}}. 
		\end{cases}
\end{equation*}
	Thus
	\begin{equation}\label{AdiDruetEst}
		\begin{cases}
			\left(\frac{1+\alpha\|u\|_{\frac{n}{m}}^{\frac{n}{m}}}{
				1-\gamma\alpha\|u\|_{\frac{n}{m}}^{\frac{n}{m}}}\right)^{\frac{m}{n-m}}|u|^{\frac{n}{n-m}} &\leq \left(\frac{1-\alpha^2\|u\|_{\frac{n}{m}}^{2\frac{n}{m}}}{
				1-\gamma\alpha\|u\|_{\frac{n}{m}}^{\frac{n}{m}}}\right)^{\frac{m}{n-m}} |v|^{\frac{n}{n-m}} \\
			\\
			\left(\frac{1-\alpha^2\|u\|_{\frac{n}{m}}^{2\frac{n}{m}}}{
				1-\gamma\alpha\|u\|_{\frac{n}{m}}^{\frac{n}{m}}}\right)^{\frac{m}{n-m}} |v|^{\frac{n}{n-m}} &\leq\left(\frac{(1-\alpha^2\|u\|_{\frac{n}{m}}^{2\frac{n}{m}}) (1-\gamma\alpha\|v\|_{\frac{n}{m}}^{\frac{n}{m}})}{
				1-\gamma\alpha\|u\|_{\frac{n}{m}}^{\frac{n}{m}}}\right)^{\frac{m}{n-m}}  |w|^{\frac{n}{n-m}} 
			\\
			& \leq \left(\frac{(1-\alpha^2\|u\|_{\frac{n}{m}}^{2\frac{n}{m}}) (1-{\gamma\alpha}\|u\|_{\frac{n}{m}}^{\frac{n}{m}})}{
				1-\gamma\alpha\|u\|_{\frac{n}{m}}^{\frac{n}{m}}}\right)^{\frac{m}{n-m}}  |w|^{\frac{n}{n-m}}
			\\
			& \leq|w|^{\frac{n}{n-m}},
		\end{cases}
	\end{equation}
	where we have used that
	$$
	\|v\|_{\frac{n}{m}}^{\frac{n}{m}}=\frac{1-\|\nabla^m v\|_{\frac{n}{m}}^{\frac{n}{m}}}{\tau}=\frac{1-\frac{\|\nabla^m u\|_{\frac{n}{m}}^{\frac{n}{m}}}{\|\nabla^m u\|_{\frac{n}{m}}^{\frac{n}{m}} + \tau \|u\|_{\frac{n}{m}}^{\frac{n}{m}}}}{\tau}=\|u\|_{\frac{n}{m}}^{\frac{n}{m}}.
	$$
Thus,  \eqref{AdiDruetEst} with \eqref{EstIntByConst} give us the desired result.

\section{Concentration-compactness-vanishing alternative: subcritical case}
	In this section, unless otherwise mentioned, we assume $ n \geq 2m $, $ 0 \leq \alpha < 1 $, $ 0 < \gamma < \frac{1}{\alpha} - 1 $ for $ \alpha > 0 $, and the subcritical condition $ 0<\beta<\beta_0$. By simplicity, for $(u_j)$ sequence in $ W^{m,\frac{n}{m}}(\mathbb{R}^n)$ with $\|u_j\|_{\frac{n}{m}}\le 1$,  we also will denote 
	\begin{equation}\label{zetaAdiDruet-Def}
		\zeta_{j} =\zeta_{j}(u_j,m,n,\alpha,\gamma):= \left(\frac{1+\alpha\|u_j\|_{\frac{n}{m}}^{\frac{n}{m}}}{1-\gamma\alpha\|u_j\|_{\frac{n}{m}}^{\frac{n}{m}}}\right)^{\frac{m}{n-m}}.
	\end{equation}
\subsection{Concentrating sequences}
	\begin{lemma}
	Let $(u_j) \subset W_{rad}^{m,\frac{n}{m}}(\mathbb{R}^n)$ with $ \|\nabla^{m}u_j\|_{\frac{n}{m}}^{\frac{n}{m}}+\|u_j\|_{\frac{n}{m}}^{\frac{n}{m}}\leq 1$ be  a concentrated sequence at the origin, \textit{i.e.}  
		\begin{equation}\label{hipConcent}
			\lim\limits_{j\rightarrow \infty} \int_{B_R} |\nabla^{m}u_{j}|^{\frac{n}{m}}\ud x = 1 \quad \forall\, R>0.
		\end{equation}	
		Then 
		\begin{equation}\label{resLemma}
			\limsup_{j \rightarrow \infty} \int_{\mathbb{R}^n}\Phi\left(\beta\zeta_{j}|u_j|^{\frac{n}{n-m}}\right) \ud x =0.
		\end{equation}
	\end{lemma}
	\begin{proof}
		First, observe that we have 
		$
		\lim_{j\rightarrow \infty}\|u_{j}\|_{\frac{n}{m}}  \rightarrow 0.
		$
		Thus,
		$\zeta_{j} \rightarrow 1 ~~ \text{ as } j \rightarrow \infty$. By Lemma \ref{radiallemma} we can chose $R$ such that $u_j(x)<1$ for all $x \in\mathbb{R}^n\setminus B_R$ and any $j\in\mathbb{N}$. For $j\geq 1$,  we set
		\begin{align*}
			I_{1_{j}} := \int_{B_{R}}\Phi\left(\beta \zeta_{j}|u_j|^{\frac{n}{n-m}}\right) \ud x 
		\end{align*}
		and
		\begin{align*}
			I_{2_{j}} := \int_{\mathbb{R}^n \setminus B_{R}}\Phi\left(\beta \zeta_{j}|u_j|^{\frac{n}{n-m}}\right) \ud x.
		\end{align*}
		Then,
		\begin{align*}
			I_{2_{j}} =
			\int_{\mathbb{R}^n \setminus B_{R}}\Phi\left(\beta \zeta_{j}|u_j|^{\frac{n}{n-m}}\right) \ud x &\leq \operatorname{e}^{\beta\zeta_j} \int_{\mathbb{R}^{n}\setminus B_{R}} |u_j|^{\frac{n}{n-m}\left(j_{\frac{n}{m} - 1}\right)} \ud x \\
			&\leq \operatorname{e}^{\beta\zeta_j} \int_{\mathbb{R}^{n}\setminus B_{R}} |u_j|^{\frac{n}{m}}\ \ud x ~ \rightarrow  0,   
		\end{align*}
		as $j \rightarrow +\infty$. Now, to estimate $I_{1_{j}}$, we'll proceed as Sani, Ruf in \cite{RufSani2013} and Lam, Lu in \cite{lamlu2012}. Let us define
		\begin{align*}
			g_{i}(|x|) :=
			\left\{
			\begin{array}{ll}
				|x|^{m-2i}, ~~  &m = 2k, k \in \mathbb{N} \\
				|x|^{m-1-2i}, ~~ &m = 2k+1, k \in \mathbb{N}
			\end{array}
			\right.
			~~  \ \ \forall x \in B_{R},
		\end{align*}
		such that $g_{i} \in W_{rad}^{m,\frac{n}{m}}(B_{R})$ and 
		\begin{equation} \nonumber
			\Delta^{j}g_{i}(|x|) = \left\{	
			\begin{array}{ll}
				c_{i}^{j}|x|^{m-2(i+j)},  &\text{ for } j \in \{1,2,..., k-i \}, \ \ \text{and $m$ even} \\
				c_{i}^{j}|x|^{m-1-2(i+j)},  &\text{ for } j \in \{1,2,..., k-i \}, \ \ \text{and $m$ odd} \\
				0  &\text{ for } j \in \{k-i+1,..., k\},
			\end{array}
			\right.,
			\ \ \forall x \in B_{R},
		\end{equation}
		where, for $ j \in \{1,2,...,k-i\}$
		\begin{align*}
			c_{i}^{j} :=  \left\{	
			\begin{array}{ll}
				\prod\limits_{h=1}^{j} [n+2k-2(h+i)][2h-2(i+h-1)], \ \ \text{ when $m$ is even} \\
				\prod\limits_{h=1}^{j} [n+2k-2-2(h+i-1)][2k-2(i+h-1)], \ \ \text{ when $m$ is odd}.
			\end{array}	
			\right. 
		\end{align*}			
		Now, we also define
		\begin{equation}\label{SaniCut}
			v_{j}(|x|) := u_{j}(|x|) - \sum\limits_{i=1}^{k-1} a_{i}g_{i}(|x|) - a_{k}, \ \ \ \ \forall x \in B_{R},
		\end{equation}
		where 
		\begin{align*}
			a_{0} &:= \frac{\Delta u_{j}({R})}{\Delta^{k-i}g_{0}(R)}, \\
			a_{i} &:= \frac{\Delta^{k-i}u_{j}(R) - \sum_{l=1}^{i-1}a_{l}\Delta^{k-i}g_{l}(R)}{\Delta^{k-i}g_{i}(R)},  \ \ \ \ \forall i \in \{1,2,...,k-1\}, \\
			a_{k} &:= u_{j}(R) - \sum\limits_{i=1}^{k-1} a_{i}g_{i}(R).
		\end{align*}
		Notice that by construction $v_{j} \in W_{\mathcal{N}}^{m, \frac{n}{m}}(B_{R}) \cap W_{rad}^{m, \frac{n}{m}}(B_{R})$ and $\Delta^{k}v_{j} = \Delta^{k}u_{j}$ in $B_{R}$ or equivalently $\nabla^{m}v_{j} = \nabla^{m}u_{j}$ in $B_{R}$.
		To simplify notation, we'll write 
		$$
		\Tilde{u}_{j}(|x|) := \sum_{i=1}^{k-1}a_ig_i(|x|) + a_{k}, \  \forall x \in B_{R}
		$$ and $p' = \frac{n}{n-m}$. So, we can notice that $u_{j}(|x|) = v_{j}(|x|) + \Tilde{u}_{j}(|x|)$. For sake of simplicity, let us consider $\Tilde{u}_{j}(|x|) := \Tilde{u}_{j}(R)$.
		We also know that 
		$$
		I_{1_{j}} := \int_{B_{R}} \Phi({{\beta \zeta_{j}
				|u_j|^{\frac{n}{n-m}}}})\ \ud x \leq \int_{B_{R}}(\operatorname{e}^{\beta\zeta_{j}|u_j|^{\frac{n}{n-m}}}-1)\ \ud x. 
		$$
		Then, the Adams Functional along $u_{j}$ can be rewritten as
		\begin{equation}\label{adamsseq1}
			\int_{B_R}(\operatorname{e}^{{\beta \zeta_{j}|u_j|^{p'} }}-1)\ \ud x= \int_{B_R}(\operatorname{e}^{{\beta  \zeta_{j} |v_{j} + \Tilde{u}_{j}(R)|^{p'}}}-1 )\ \ud x.
		\end{equation}
		By construction $v_{j} \in W_{0}^{1,\frac{n}{m}}(B_{R})$  and $\|\Delta^{k} v_j\| = \|\Delta^{k} u_j\|$.
		Now, by the elementary inequality
	\begin{equation}
		(a+b)^{p} \leq (1+\delta)^{p}a^{p}+  \left(1+ \dfrac{1}{\delta}\right)^{p} b^{p},\; \;\;\text{for $p \geq 1$, $a,b>0$ and $\delta > 0$}
		\end{equation}
		and \eqref{adamsseq1}, it follows that
		\begin{align*}
			\int_{B_R}(\operatorname{e}^{\beta\zeta_{j} |v_{j} + \Tilde{u}_{j}(R)|^{p'}} -1 )\ \ud x &\leq \int_{B_R}\Big(\operatorname{e}^{\beta\zeta_{j} \left|(1+\delta)^{p'}(v_{j})^{p'}+\left(1+\frac{1}{\delta}\right)^{p'}(\Tilde{u}_{j}(R))^{p'}\right|} -1\Big) \ \ud x \\
			&\leq \int_{B_R}\Big(\operatorname{e}^{\beta\zeta_{j}\left( \left|(1+\delta)^{p'}(v_{j})^{p'}\right|+\left|\left(1+\frac{1}{\delta}\right)^{p'}(\Tilde{u}_{j}(R))^{p'}\right|\right)}-1\Big) \  \ud x\\
			&\leq \int_{B_R}\operatorname{e}^{\beta\zeta_{j} \left|(1+\delta)^{p'}(v_{j})^{p'}\right|}\operatorname{e}^{\beta\zeta_{j} \left|\left(1+\frac{1}{\delta}\right)^{p'}(\Tilde{u}_{j}(R))^{p'}\right|} \  \ud x \\
			&\leq \int_{B_R}\operatorname{e}^{\beta\zeta_{j} \left|(1+\delta)^{p'}(v_{j})^{p'}\right|} \operatorname{e}^{C(n,m, R)\frac{1}{R}\|u_j\|_{n/2}}  \  \ud x \\
			&\leq \operatorname{e}^{C(n,m, R)\frac{1}{R}\zeta_{j}\|u_j\|_{n/m}}\int_{B_R}\operatorname{e}^{\beta\zeta_{j} \left|(1+\delta)^{p'}(v_{j})^{p'}\right|}  \  \ud x. 		
		\end{align*}
		Then
		\begin{align*}
			\|\Delta^{k} v_j \|_{L^{p'}(B_{R})}^{p'}\leq  \|\Delta^{k} u_j \|_{L^{p'}(B_{\bar{R}})}^{p'}.
		\end{align*}
	Noticing that $ \operatorname{e}^{C(n,m, R)\frac{1}{R}\zeta_{j}\|u_j\|_{n/m}}\rightarrow 1$, and by Vitali's Convergence Theorem and by Tarsi \cite{Tarsi2012} we got the result (\ref{resLemma}) taking limit as $j \rightarrow \infty$.
	\end{proof}
	
\subsection{Compactness and vanish level estimates }
	\begin{lemma}\label{compacidade}
		Let $(u_j) \subset W_{rad}^{m,\frac{n}{m}}(\mathbb{R}^n)$ be  a  sequence satisfying $\|u_j\|_{\frac{n}{m}}^{\frac{n}{m}}+ \|\nabla^{m}u_j\|_{\frac{n}{m}}^{\frac{n}{m}}\leq 1$  and
		$\|u_{j}\|_{\frac{n}{m}}^{\frac{n}{m}} \rightarrow \theta$ for some $\theta \in [0,1)$. Assume, up to a subsequence, $u_j \rightharpoonup u$ in $W_{rad}^{m,\frac{n}{m}}(\mathbb{R}^n)$.
		Then, 
		\begin{enumerate}
		\item [$1)$] if $\frac{n}{m}$ is not integer, 
		\begin{equation}\label{resLemma1}
			\lim_{j\rightarrow +\infty} \int_{\mathbb{R}^{n}}\Phi\Big(\beta\zeta_{j}|u_j|^{\frac{m}{n-m}}\Big) \ud x = \int_{\mathbb{R}^{n}} \Phi\left({{\beta\left(\frac{1+\theta\alpha}{1-\theta\gamma\alpha}\right)^{\frac{m}{n-m}}
					|u|^{\frac{n}{n-m}}}}\right)\ \ud x 
		\end{equation}
		\item [$2)$] if $\frac{n}{m}$ is integer 
		\begin{align}\label{resLemma2}
			\lim_{j\rightarrow +\infty} \int_{\mathbb{R}^{n}}\Phi\Big(\beta\zeta_{j}|u_j|^{\frac{m}{n-m}}\Big) \ud x = &\int_{\mathbb{R}^{n}} \Phi\left({{\beta \left(\frac{1+\theta\alpha}{1-\theta\gamma\alpha}\right)^{\frac{m}{n-m}}
					|u|^{\frac{n}{n-m}}}}\right)\ \ud x \\ &+\frac{ \beta^\frac{n-m}{m}}{\left(\frac{n}{m}-1\right)!}\left(\frac{1+\theta\alpha}{1-\theta\gamma\alpha}\right)\left(\theta-\|u\|_{\frac{n}{m}}^{\frac{n}{m}}\right). \nonumber
		\end{align}
		\end{enumerate} 
	\end{lemma}
	
	\begin{proof} We divide the proof into two cases.
	\paragraph*{ Case $\frac{n}{m} \not\in \mathbb{Z}:$} Note that
		\begin{equation}\label{zetaAdiDruet}
			\zeta_{j} = \left(\frac{1+\alpha\|u_{j}\|_{\frac{n}{m}}^{\frac{n}{m}}}{1-\gamma\alpha\|u_{j}\|_{\frac{n}{m}}^{\frac{n}{m}}}\right)^{\frac{m}{n-m}} \rightarrow  \zeta:=\left(\frac{1+\alpha\theta}{1-\gamma\alpha\theta}\right)^{\frac{m}{n-m}}.
		\end{equation}
		By Lemma~\ref{radiallemma} we can estimate as follows
		\begin{align*}
			\Phi(\beta\zeta_{j}|u_{j}|^{\frac{n}{n-m}}) &\leq 	\operatorname{e}^{\beta\zeta_{j}}|u_{j}|^{(\frac{n}{m})j_{\frac{n}{m}}}\\
			&\leq \operatorname{e}^{\beta \zeta_{j}}\left(\left(\frac{1}{m\sigma_n}\right)^{\frac{m}{n}}\frac{1}{|x|^{(\frac{n-1}{n})m}}\|u_{j}\|_{W^{1, \frac{n}{m}}}\right)^{(\frac{n}{m})j_{\frac{n}{m}}}.
		\end{align*}
		Since that $j_{\frac{n}{m}} \geq \frac{n}{m}$, 
		\begin{align*}
			\Phi(\beta\zeta_j|u_{j}|^{\frac{n}{n-m}}) \leq \operatorname{e}^{\beta\zeta_j}\left(\frac{\left(\frac{1}{m\sigma_n}\right)^{\frac{n}{m}} \|u_{j}\|_{W^{1, \frac{n}{m}}}^{\frac{n^2}{m^2}}}{|x|^{\frac{n(n-1)}{m}}}\right).
		\end{align*}
		Now, by integrating both sides outside the ball centered at origin with radius $R$ and since $\zeta_j \geq 2$ as $j\rightarrow \infty$, we got
		\begin{align*}
			\int_{\mathbb{R}^{n}\backslash B_{R}}\Phi(\beta\zeta_j|u_{j}|^{\frac{n}{n-m}}) \mathrm{d}x\leq C(n, m, \beta)\int_{\mathbb{R}^{n}\backslash B_{R}}\frac{1}{|x|^{\frac{n(n-1)}{m}}}\mathrm{d}x,
		\end{align*}
		using polar coordinates, there exists $\epsilon > 0$ such that, for a sufficiently large $R$ and $n>m$ the integral in right side of the above inequality becomes
		\begin{align*}
			\int_{\mathbb{R}^{n}\backslash B_{R}}\frac{1}{|x|^{\frac{n(n-1)}{m}}}\mathrm{d}x &= \omega_{n-1}\int_{R}^{+\infty}\frac{r^{n-1}}{r^{\frac{n(n-1)}{m}}}\mathrm{d}r \\
			&= \int_{R}^{+\infty} r^{\frac{-n^2+n-m+mn}{m}}\mathrm{d}r \\
			&= \left[\frac{m}{-n^2+nm+1}\right]  \frac{1}{R^{\frac{n^2-nm-1}{m}}} \\
			&< \epsilon 
		\end{align*}
		where $\omega_{n-1}$ denotes the surface area of the $(n-1)$-dimensional unit ball.
		Thus, we just proved that, 
		$$
		\int_{\mathbb{R}^{n}\backslash B_{R}}\Phi(\beta\zeta_{j}|u_{j}|^{\frac{n}{n-m}}) \mathrm{d}x  \rightarrow 0
		$$
		In other words, the sequence $(u_j)$ is tight on whole $\mathbb{R}^n$. Thus, using additionally the fact of the sequence be a uniformly integrable sequence over $\mathbb{R}^n$, applying again the Vitali's Convergence Theorem, we have proven the result. 
		\paragraph*{ Case $\frac{n}{m} \in \mathbb{Z}:$}  Notice that we need to separate in two cases (integer and non-integer), because the Radial Lemma can't be efficently applied by the influence of the first term on the summation in integer case. We observe that 
		\begin{align*}
			\Phi(\beta\zeta_{j}|u_j|^{\frac{n}{n-m}}) &= \exp{(\beta\zeta_{j}|u_j|^{\frac{n}{n-m}})} - \sum_{i=0}^{j_{n,m}-2}\frac{(\beta\zeta_{j}|u_j|^{\frac{n}{n-m}})^{i}}{i!} \\
			&= \frac{(\beta\zeta_{j}|u_j|^{\frac{n}{n-m}})^{j_{n,m}-1}}{(j_{n,m}-1)!}+ \sum_{i=j_{n,m}}^{\infty}\frac{(\beta\zeta_{j}|u_j|^{\frac{n}{n-m}})^{i}}{i!}.
		\end{align*}
		Since $\frac{n}{m} \in \mathbb{Z}$, this implies $j_{n,m} = \frac{n}{m}$. So we obtain
		\begin{equation}\nonumber
			\Phi(\beta\zeta_{j}|u_j|^{\frac{n}{n-m}}) =  \frac{\beta^{\frac{n}{n-m}}\zeta_{j}^{\frac{n}{n-m}}|u_j|^{\frac{n}{m}}}{(\frac{n}{m}-1)!}+ \sum_{i=\frac{n}{m}}^{\infty}\frac{(\beta\zeta_{j}|u_j|^{\frac{n}{n-m}})^{i}}{i!}.
		\end{equation}
		Integrating both sides
		\begin{align*}
			\int_{\mathbb{R}^{n}}\Phi(\beta\zeta_{j}|u_j|^{\frac{n}{n-m}})\mathrm{d}x &=  \int_{\mathbb{R}^{n}}\frac{\beta^{\frac{n}{n-m}}\zeta_{j}^{\frac{n}{n-m}}|u_j|^{\frac{n}{m}}}{(\frac{n}{m}-1)!}\mathrm{d}x+ \int_{\mathbb{R}^{n}}\sum_{i=\frac{n}{m}}^{\infty}\frac{(\beta\zeta_{j}|u_j|^{\frac{n}{n-m}})^{i}}{i!}\mathrm{d}x \\
			&= \frac{\beta^{\frac{n}{n-m}}\zeta_{j}^{\frac{n}{n-m}}}{(\frac{n}{m}-1)!}\int_{\mathbb{R}^{n}}|u_{j}|^{\frac{n}{m}}\mathrm{d}x + \int_{\mathbb{R}^{n}}\sum_{i=\frac{n}{m}}^{\infty}\frac{(\beta\zeta_{j}|u_j|^{\frac{n}{n-m}})^{i}}{i!}\mathrm{d}x \\
			& = \frac{\beta^{\frac{n}{n-m}}\zeta_{j}^{\frac{n}{n-m}}}{(\frac{n}{m}-1)!}\|u_{j}\|_{\frac{n}{m}}^{\frac{n}{m}} + \int_{\mathbb{R}^{n}}\sum_{i=\frac{n}{m}}^{\infty}\frac{(\beta\zeta_{j}|u_j|^{\frac{n}{n-m}})^{i}}{i!}\mathrm{d}x \\
		\end{align*}
		Since $\|u_{j}\|_{\frac{n}{m}}^{\frac{n}{m}} \rightarrow \theta$ and by \eqref{zetaAdiDruet}, taking limit as $j \rightarrow \infty$. 
		\begin{align*}
			& = \frac{\beta^{\frac{n}{n-m}}}{(\frac{n}{m}-1)!}\left(\frac{1+\alpha\theta}{1-\gamma\alpha\theta}\right)\theta + \lim\limits_{j \rightarrow \infty} \int_{\mathbb{R}^{n}}\sum_{i=\frac{n}{m}}^{\infty}\frac{(\beta\zeta_{j}|u_j|^{\frac{n}{n-m}})^{i}}{i!}\mathrm{d}x. 
		\end{align*}
		Using the same argument as in non-integer case, we can prove that the sequence $(u_{j}) \subset W_{rad}^{n, \frac{n}{m}}(\mathbb{R}^{n})$ is uniformly integrable and tight on whole $\mathbb{R}^{n}$, by the Vitali's convergence Theorem, we can change the sign of the limit with the integral on the right side, and one can get
		\begin{align*}
			& = \frac{\beta^{\frac{n}{n-m}}}{(\frac{n}{m}-1)!}\left(\frac{1+\alpha\theta}{1-\alpha\gamma\theta}\right)\theta  + \int_{\mathbb{R}^{n}}\sum_{i=\frac{n}{m}}^{\infty}\frac{(\beta\zeta|u|^{\frac{n}{n-m}})^{i}}{i!}\mathrm{d}x \\
			& = \frac{\beta^{\frac{n}{n-m}}}{(\frac{n}{m}-1)!}\left(\frac{1+\alpha\theta}{1-\alpha\gamma\theta}\right)\theta  + \int_{\mathbb{R}^{n}}\left[\sum_{i=\frac{n}{m}}^{\infty}\frac{(\beta\zeta|u|^{\frac{n}{n-m}})^{i}}{i!}+\frac{\beta^{\frac{n}{n-m}}\zeta^{\frac{n}{n-m}}|u|^{\frac{n}{m}}}{(\frac{n}{m}-1)!} - \frac{\beta^{\frac{n}{n-m}}\zeta^{\frac{n}{n-m}}|u|^{\frac{n}{m}}}{(\frac{n}{m}-1)!}\right]\mathrm{d}x \\
			& = \frac{\beta^{\frac{n}{n-m}}}{(\frac{n}{m}-1)!}\left(\frac{1+\alpha\theta}{1-\alpha\gamma\theta}\right)\theta  + \int_{\mathbb{R}^{n}}\left[\sum_{i=\frac{n}{m}-1}^{\infty}\frac{(\beta\zeta|u|^{\frac{n}{n-m}})^{i}}{i!} - \frac{\beta^{\frac{n}{n-m}}\zeta^{\frac{n}{n-m}}|u|^{\frac{n}{m}}}{(\frac{n}{m}-1)!}\right]\mathrm{d}x \\
			& = \frac{\beta^{\frac{n}{n-m}}}{(\frac{n}{m}-1)!}\left(\frac{1+\alpha\theta}{1-\alpha\gamma\theta}\right)\theta + \int_{\mathbb{R}^{n}}\Phi(\beta\zeta|u_{j}|^{\frac{n}{n-m}})\mathrm{d}x - \int_{\mathbb{R}^{n}}\frac{\beta^{\frac{n}{n-m}}\zeta^{\frac{n}{n-m}}|u|^{\frac{n}{m}}}{(\frac{n}{m}-1)!}\mathrm{d}x	\\
			&= 
			\frac{\beta^{\frac{n}{n-m}}}{(\frac{n}{m}-1)!}\left(\frac{1+\alpha\theta}{1-\alpha\gamma\theta}\right)\theta + \int_{\mathbb{R}^{n}}\Phi(\beta\zeta|u|^{\frac{n}{n-m}})\mathrm{d}x - \frac{\beta^{\frac{n}{n-m}}}{(\frac{n}{m}-1)!}\left(\frac{1+\alpha\theta}{1-\alpha\gamma\theta}\right)\int_{\mathbb{R}^{n}}|u|^{\frac{n}{m}}\mathrm{d}x \\
			&= 		\frac{\beta^{\frac{n}{n-m}}}{(\frac{n}{m}-1)!}\left(\frac{1+\alpha\theta}{1-\alpha\gamma\theta}\right)(\theta-\|u\|_{\frac{n}{m}}^{\frac{n}{m}}) + \int_{\mathbb{R}^{n}}\Phi(\beta\zeta|u|^{\frac{n}{n-m}})\mathrm{d}x 
		\end{align*}
		we got the desired result.
	\end{proof}
	\begin{remark}\label{compacidade-critica}
		The  Lemma ~\ref{compacidade} works for $\beta=\beta_{0}$ if $\alpha, \theta>0$.
	\end{remark}
	\begin{proposition}\label{vanishlevel} Let $0 < \alpha < 1$,  $0 \leq \gamma < \frac{1}{\alpha}-1$ and  $AD(n, m, \beta, \alpha, \gamma)$  as in \eqref{AD-notation}. For $n>m$, then 
		\begin{equation}
			AD(n, m, \beta, \alpha, \gamma)    > \frac{\beta^{\frac{n}{m}-1}}{\left(\frac{n}{m}-1\right)!}\left(\frac{1+\alpha}{1-\gamma\alpha}\right)
		\end{equation}
		for $\beta \in (0, \beta_{0}]$ when $\frac{n}{m} \geq 3$ and for $\beta \in \left(\frac{1 +2\alpha  - \gamma\alpha^2}{ 1 + \alpha(1 - \gamma) - \gamma\alpha^2}\frac{2}{\mathcal{B}_{GN}}, \beta_{0}\right]$ with $0\leq \gamma < \frac{m(1+\alpha)-1-2\alpha}{m\alpha^2+\alpha m - \alpha^2}$ when $\frac{n}{m} = 2$.
	\end{proposition}
	\begin{proof} 	First, we introduce the following operator in $W^{m, n/m}(\mathbb{R}^n)$. Let $t>0$, we define $H_{t}: W^{m, n/m}(\mathbb{R}^n) \rightarrow W^{m, n/m}(\mathbb{R}^n)$ by
		\begin{equation}\label{HT}
		H_{t}(u)(x) := t^{\frac{m}{n}}u(t^{\frac{1}{n}}x).
		\end{equation}
		So, we can compute 
		\begin{align*}
			\|H_{t}(u)\|^{\frac{n}{m}}_{\frac{n}{m}} +\left\|\nabla^m H_{t}(u)\right\|^\frac{n}{m}_{\frac{n}{m}}&= \left\|t^{\frac{m}{n}} u\left(t^\frac{1}{n}x\right)\right\|^\frac{n}{m}_{\frac{n}{m}}+\left\|t^{\frac{m}{n}}\nabla^m\left(u\left(t^\frac{1}{n}x\right)\right)\right\|^\frac{n}{m}_{\frac{n}{m}}\\
			& = \left\| u\left(x\right)\right\|^\frac{n}{m}_{\frac{n}{m}}+t\left\|\nabla^m\left(u\left(x\right)\right)\right\|^\frac{n}{m}_{\frac{n}{m}},
		\end{align*}
		and
		$$
		\frac{\|H_{t}(u)\|_{\frac{n}{m}}^{\frac{n}{m}}}{\|H_{t}(u)\|^{\frac{n}{m}}_{\frac{n}{m}} +\left\|\nabla^m H_{t}(u)\right\|^\frac{n}{m}_{\frac{n}{m}}} = \frac{\|u\|_{\frac{n}{m}}^{\frac{n}{m}}}{\left\| u\left(x\right)\right\|^\frac{n}{m}_{\frac{n}{m}}+t\left\|\nabla^m\left(u\left(x\right)\right)\right\|^\frac{n}{m}_{\frac{n}{m}}}.
		$$
		By simplicity, we denote
		\begin{equation}\label{ETA}
		\eta_{u}(t)=\left\| u\right\|^\frac{n}{m}_{\frac{n}{m}}+t\left\|\nabla^m u\right\|^\frac{n}{m}_{\frac{n}{m}}
		\end{equation}
		 and
		 \begin{equation}\label{RHO}
		\varrho(t):=\varrho(t,\alpha,\gamma)=\frac{1+\alpha t}{1-\gamma\alpha t}, \;\; \text{ for all   }  t\in[0,1].  
		\end{equation}
Note that $\varrho(t)$ is an increasing function with $1\le \varrho(t)\le (1+\alpha)/(1-\gamma\alpha)$ for $t\in [0,1]$. 

Now, since $$\Phi(\beta t) \geq \frac{\beta^{j_{\frac{n}{m}} - 1 }}{\left(j_{\frac{n}{m}} - 1\right)!}t^{j_{\frac{n}{m}} - 1} + \frac{\beta^{j_{\frac{n}{m}} }}{\left(j_{\frac{n}{m}}\right)!}t^{j_{\frac{n}{m}}} \;\;\mbox{and}\;\; j_{\frac{n}{m}}\geq \frac{n}{m} $$
 we can obtain
		\begin{equation}\label{VPhi-pontual}
		\begin{aligned}
			& \Phi\left(\beta \left(\frac{1+\alpha\frac{\|H_{t}(u)\|_{\frac{n}{m}}^{\frac{n}{m}}}{\eta_{u}(t)}}{1-\gamma\alpha\frac{\|H_{t}(u)\|_{\frac{n}{m}}^{\frac{n}{m}}}{\eta_{u}(t)}}\right)^{\frac{m}{n-m}}\frac{|H_{t}(u)|^{\frac{n}{n-m}}}{(\eta_{u}(t))^{\frac{m}{n-m}}}\right) =\Phi\left(\beta\left(\varrho\left(\frac{\|H_{t}(u)\|_{\frac{n}{m}}^{\frac{n}{m}}}{\eta_{u}(t)}\right)\right)^{\frac{m}{n-m}}\frac{|H_{t}(u)|^{\frac{n}{n-m}}}{(\eta_{u}(t))^{\frac{m}{n-m}}}\right)\\
&\geq \frac{\beta^{\frac{n}{m}-1}}{(\frac{n}{m}-1)!}\left[\left(\varrho\left(\frac{\|H_{t}(u)\|_{\frac{n}{m}}^{\frac{n}{m}}}{\eta_{u}(t)}\right)\right)^{\frac{m}{n-m}}\frac{|H_{t}(u)|^{\frac{n}{n-m}}}{(\eta_{u}(t))^{\frac{m}{n-m}}}\right]^{\frac{n-m}{m}} \\
& + \frac{\beta^{\frac{n}{m}}}{(\frac{n}{m})!}\left[\left(\varrho\left(\frac{\|H_{t}(u)\|_{\frac{n}{m}}^{\frac{n}{m}}}{\eta_{u}(t)}\right)\right)^{\frac{m}{n-m}}\frac{|H_{t}(u)|^{\frac{n}{n-m}}}{(\eta_{u}(t))^{\frac{m}{n-m}}}\right]^{\frac{n}{m}}.
		\end{aligned}
		\end{equation}
Since $\|H_t(u)\|_{q}=t^{\frac{qm-n}{qn}}\|u\|_{q}$, we have 
$$
\varrho\left(\frac{\|H_{t}(u)\|_{\frac{n}{m}}^{\frac{n}{m}}}{\eta_{u}(t)}\right)=\varrho\left(\frac{\|u\|_{\frac{n}{m}}^{\frac{n}{m}}}{\eta_{u}(t)}\right).
$$
Hence, by integrating in \eqref{VPhi-pontual} and taking the supremum
		\begin{align*}
			&AD(n,m,\beta, \alpha, \gamma) \geq 
			\frac{\beta^{\frac{n}{m}-1}}{(\frac{n}{m}-1)!}\varrho\left(\frac{\|u\|_{\frac{n}{m}}^{\frac{n}{m}}}{\eta_{u}(t)}\right)\frac{
				\|u\|^{\frac{n}{m}}_{\frac{n}{m}}}{\eta(t)} + \frac{\beta^{\frac{n}{m}}}{(\frac{n}{m})!}\left[\varrho\left(\frac{\|u\|_{\frac{n}{m}}^{\frac{n}{m}}}{\eta_{u}(t)}\right)\right]^{\frac{m}{n-m}}\frac{t^{\frac{m}{n-m}}\|u\|_{\frac{n^2}{(n-m)m}}^{\frac{n^2}{(n-m)m}}}{\eta(t)^{\frac{n}{n-m}}} \\
			&\quad \quad \geq \frac{\beta^{\frac{n}{m}-1}}{(\frac{n}{m}-1)!}\left[\varrho\left(\frac{\|u\|_{\frac{n}{m}}^{\frac{n}{m}}}{\eta_{u}(t)}\right)\frac{\|u\|^{\frac{n}{m}}_{\frac{n}{m}}}{\eta_u(t)} + \frac{\beta}{\frac{n}{m}}\left[\varrho\left(\frac{\|u\|_{\frac{n}{m}}^{\frac{n}{m}}}{\eta_{u}(t)}\right)\right]^{\frac{m}{n-m}}\frac{t^{\frac{m}{n-m}}\|u\|_{\frac{n^2}{(n-m)m}}^{\frac{n^2}{(n-m)m}}}{\eta_u(t)^{\frac{n}{n-m}}}\right].
		\end{align*}
		Let us define
		\begin{align*}
		h(t) :&=\varrho\left(\frac{\|u\|_{\frac{n}{m}}^{\frac{n}{m}}}{\eta_{u}(t)}\right)\frac{\|u\|^{\frac{n}{m}}_{\frac{n}{m}}}{\eta_{u}(t)} + \frac{\beta}{\frac{n}{m}}\left[\varrho\left(\frac{\|u\|_{\frac{n}{m}}^{\frac{n}{m}}}{\eta_{u}(t)}\right)\right]^{\frac{m}{n-m}}\frac{t^{\frac{m}{n-m}}\|u\|_{\frac{n^2}{(n-m)m}}^{\frac{n^2}{(n-m)m}}}{\eta_{u}(t)^{\frac{n}{n-m}}},
		\end{align*}
which we will also denote in terms of its components as
		$$
		h(t) = f(t) +\frac{\beta}{\frac{n}{m}}t^{\frac{m}{n-m}}g(t),
		$$ 
		where
		$$
		g(t) := \left[\varrho\left(\frac{\|u\|_{\frac{n}{m}}^{\frac{n}{m}}}{\eta_{u}(t)}\right)\right]^{\frac{m}{n-m}}\frac{\|u\|_{\frac{n^2}{(n-m)m}}^{\frac{n^2}{(n-m)m}}}{\eta_{u}(t)^{\frac{n}{n-m}}}.
		$$
		Thus, we write 
		\begin{align}\label{Adineqh}
			AD(n,m, \beta, \alpha, \gamma) \ge \frac{\beta^{\frac{n}{m}-1}}{(\frac{n}{m}-1)!}h(t).
		\end{align}
		Note that 
		\begin{equation}\label{eta0rho0}
		\eta_{u}(0)=\|u\|^{\frac{n}{m}}_{\frac{n}{m}},\;\; \varrho(1)=\frac{1+\alpha}{1-\gamma\alpha} \; \text{ and }\; h(0)=\varrho(1).
		\end{equation}
 Thus, by taking into account \eqref{Adineqh}, for $\frac{n}{m}\ge 3$ we need to prove that $h(t)$ is a increasing function for $t$ near to $0$. Indeed, we have
$$
		h'(t) = f'(t)+\frac{\beta}{\frac{n}{m}}\left(\frac{m}{n-m}\right)t^{\frac{2m-n}{n-m}}g(t)+\frac{\beta}{\frac{n}{m}}t^{\frac{m}{n-m}}g'(t).
$$
	A direct calculation shows that both $f^{\prime}$ and $g^{\prime}$ are bounded for 
$t$ near $0$. Thus, since $g(0)>0$  and $t^{\frac{2m-n}{n-m}}\to+\infty$ as $t\to 0$,   we obtain
 	$h'(t)>0$ for $t$ near $0$ which gives the desired result.

Now, suppose $\frac{n}{m}=2$.  In this case, we have 
\begin{equation}\nonumber
\begin{aligned}
&\varrho\left(\frac{\|u\|_{2}^{2}}{\eta_{u}(t)}\right)=\frac{1+\frac{\alpha\|u\|^{2}_{2}}{\eta_u(t)}}{1-\frac{\gamma\alpha\|u\|^{2}_{2}}{\eta_{u}(t)}}\\
&f(t)=\varrho\left(\frac{\|u\|_{2}^{2}}{\eta_{u}(t)}\right)\frac{\|u\|_{2}^{2}}{\eta_{u}(t)}\\
&g(t)=\varrho\left(\frac{\|u\|_{2}^{2}}{\eta_{u}(t)}\right)\frac{\|u\|_{4}^{4}}{\eta_{u}(t)^{2}}.
\end{aligned}
\end{equation}
By noticing that 
\begin{equation}
\begin{aligned}
\varrho_0:=\frac{d}{dt}\left[\varrho\left(\frac{\|u\|_{2}^{2}}{\eta_{u}(t)}\right)\right]\Bigg|_{t=0}=-\frac{\alpha+\gamma\alpha}{(1-\gamma\alpha)^2}\frac{\|\nabla^{m}u\|_{2}^{2}}{\|u\|^{2}_{2}}
\end{aligned}
\end{equation}
and taking into account \eqref{eta0rho0}, we obtain
\begin{equation}\label{0-identidades}
\begin{aligned}
&f(0)=\frac{1+\alpha}{1-\gamma\alpha}\\
&f^{\prime}(0)=\varrho_0-\frac{1+\alpha}{1-\gamma\alpha}\frac{\|\nabla^{m}u\|_{2}^{2}}{\|u\|^{2}_{2}}=-\frac{1+2\alpha-\gamma\alpha^2}{(1-\gamma\alpha)^2}\frac{\|\nabla^{m}u\|_{2}^{2}}{\|u\|^{2}_{2}}\\
&g(0)=\frac{1+\alpha}{1-\gamma\alpha}\frac{\|u\|_{4}^{4}}{\|u\|^{4}_{2}}=\frac{1-\gamma\alpha+\alpha-\gamma\alpha^2}{(1-\gamma\alpha)^{2}}\frac{\|u\|_{4}^{4}}{\|u\|^{4}_{2}}\\
&g^{\prime}(0)=\varrho_0\frac{\|u\|^{4}_4}{\|u\|^{4}_2}-\frac{1+\alpha}{1-\gamma\alpha}\frac{\|u\|_{4}^{4}}{\|u\|^{4}_2}\frac{\|\nabla^{m}u\|^{2}_{2}}{\|u\|^{2}_{2}}=-\frac{1+2\alpha-\gamma\alpha^2}{(1-\gamma\alpha)^2}\frac{\|u\|_{4}^{4}}{\|u\|^{4}_2}\frac{\|\nabla^{m}u\|^{2}_{2}}{\|u\|^{2}_{2}}.
\end{aligned}
\end{equation}
If $n=2m$, we have $h(t)=f(t)+\frac{\beta}{2}tg(t)$ and, thus  $ h'(0) = f'(0)+\frac{\beta}{2}g(0)$ and $h(0)=f(0)$. So, the Taylor expansion  and \eqref{0-identidades} yield 
\begin{align*}
			h(t) &= f(0) + (f'(0)+ \frac{\beta}{2}g(0))t + O(t^2) \\ 
			&= \frac{1+\alpha}{1-\gamma\alpha} +O(t^2) \\
			&+\frac{1}{(1-\gamma\alpha)^2}\frac{\|u\|_{4}^{4}}{\|u\|^{4}_{2}}\left[-(1+2\alpha-\gamma\alpha^{2})\frac{\|\nabla^{m}u\|_{2}^{2}\|u\|^{2}_{2}}{\|u\|^{4}_{4}}+\frac{\beta}{2}(1-\gamma\alpha+\alpha-\gamma\alpha^2)\right]t\\
			&=\frac{1+\alpha}{1-\gamma\alpha}+O(t^2) \\
			& +\frac{(1 + \alpha)(1-\gamma\alpha)}{2(1-\gamma\alpha)^2}\frac{\|u\|_{4}^{4} }{\|u\|_{2}^{4}}\left(\beta -\frac{\left\|\nabla^m u\right\|^2_{2}\|u\|_{2}^{2} }{\|u\|_{4}^{4}}\frac{ 2  (1 +2\alpha  - \gamma\alpha^2 )}{ 1 + \alpha - \gamma\alpha - \gamma\alpha^2}\right)t.
		\end{align*}
Therefore, by estimate \eqref{Adineqh} and expanding $h$ in MacLaurin series, we obtain
		\begin{align*}
			& AD(2m,m, \beta, \alpha, \gamma) \ge \beta\left(\frac{1+\alpha}{1-\gamma\alpha}\right) \\
			& \quad +\frac{\beta^2}{2}\frac{1+\alpha}{1-\gamma\alpha}\frac{\|u\|_{4}^{4} }{\|u\|_{2}^{4}}\left(1-\frac{2}{\beta}\frac{\left\|\nabla^m u\right\|^2_{2}\|u\|_{2}^{2} }{\|u\|_{4}^{4}}\frac{1 +2\alpha  - \gamma\alpha^2 }{ 1 + \alpha - \gamma\alpha - \gamma\alpha^2}\right)t+ O(t^2).
		\end{align*}
By the arbitrariness of $u$,
	taking into account  \eqref{ConstantGN}  and \eqref{GNestimates},   we conclude the result.
	\end{proof}	
	\section{Existence and non-existence of extremals for the subcritical Adams functional of Adimurthi-Druet type in even dimension}
In this section we concerned to prove both Theorem~\ref{ThmAttain} and Theorem~\ref{ThmNotAttain}. In fact, we will study the attainability   and the non-attainability of $AD(2m,m, \beta, \alpha, \gamma)$ in the subcritical case $\beta<\beta_0$. Let us recall the Adams functional
	\begin{equation}\label{AdamsFunctional}
		F_{2m, m, \beta, \alpha, \gamma}(u) = \int_{\mathbb{R}^{2m}}\left(\operatorname{e}^{\beta\varrho(\|u\|_{2}^{2})u^2}-1\right)\mathrm{d}x.
	\end{equation}
where $\varrho(\|u\|_{2}^{2})=\varrho(\|u\|_{2}^{2}, \alpha,\gamma)$ is given by \eqref{RHO}.
	\subsection{Proof of Theorem \ref{ThmAttain}: Attainability in \texorpdfstring{\(n=2m\)}{n=2m}  case}
	Taking into account \eqref{Radializar}, we can choose  $(u_{j} ) \in W^{m,2}_{rad}(\mathbb{R}^{2m})$ with $\|u_{j}\|_{W^{m,2}(\mathbb{R}^{2m})}=1$ radially symmetric maximizing sequence for $AD(2m,m, \beta, \alpha, \gamma)$, that is, 
	\begin{equation}\label{Maxn=2m}
	\|\nabla^{m} u_j\|_{2}^{2}+\|u_j\|_{2}^{2} =1, \quad \lim\limits_{j \rightarrow \infty} F_{2m, m, \beta, \alpha, \gamma}(u_j) = \underset{\underset{\|\nabla^{m} u\|_{2}^{2}+\|u\|_{2}^{2} \leq 1}{u\in W^{m,2}(\mathbb{R}^{2m})}}{\sup }F_{2m, m, \beta, \alpha, \gamma}(u).
	\end{equation}
Up to a subsequence, we can assume that $u_j \rightharpoonup u_0$ weakly in $W_{rad}^{m,2}(\mathbb{R}^{2m})$, 	 $\|u_{j}\|_{2}^{2} \rightarrow \theta_0$ and $\|\nabla^{m} u_{j}\|_{2}^{2}\rightarrow \theta_1$ for some $\theta_0, \theta_1 \in [0,1]$ such that $\theta_0 + \theta_1 := \overline{\theta} \leq 1$. From Lemma~\ref{compacidade}-\eqref{resLemma2}, we derive 
\begin{align}\label{resLemma2n=2m}
			& \lim_{j\rightarrow +\infty}  \int_{\mathbb{R}^{2m}}\left(\operatorname{e}^{\beta\varrho(\|u_j\|_{2}^{2})u^2_j}-1\right)\mathrm{d}x =\int_{\mathbb{R}^{2m}}\left(\operatorname{e}^{\beta\varrho(\theta_0)u^2_0}-1\right)\mathrm{d}x+\beta\varrho(\theta_0)\left(\theta_0-\|u_0\|_{2}^{2}\right).
		\end{align}
If $u_{0}\equiv 0$, combing \eqref{Maxn=2m} with \eqref{resLemma2n=2m}, 
\begin{align*}		
		AD(2m,m,\beta, \alpha,\gamma) &= \beta\varrho(\theta_0)\theta_0\leq \beta\left(\frac{1+\alpha}{1- \gamma\alpha}\right),
	\end{align*}
which is impossible due to Proposition \ref{vanishlevel}. Then, we can suppose that $u_0\not=0$.  By the lower semicontinuity of the norm $\|u_0\|^{2}_2\le \liminf_{j\to\infty}\|u_j\|^{2}_{2}=\theta_0$.  So, let us define $\tau\ge 1$ and $v$  given by
\begin{equation}\label{magia=1}
\tau^{2m}=\frac{\theta_0}{\|u_0\|_{2}^2}\; \text{ and }\; v(x)=u_0\big(\frac{x}{\tau}\big).
\end{equation}	
It follows that 
	\begin{align}\label{magic}
	 \|v\|_{2}^{2} = \tau^{2m}\|u_{0}\|_{2}^{2} = \theta_{0}\;\text{ and }\; \|\nabla^{m}  v\|_{2}^{2} = \|\nabla^{m}u_{0}\|_{2}^{2}.
	\end{align}
Then, by  lower semicontinuity again
	$$
	\|\nabla^{m}  v\|_{2}^{2} +	\|v\|_{2}^{2}=\|\nabla^{m}  u_0\|_{2}^{2} +	\tau^{2m}\|u_0\|_{2}^{2}  \leq  \liminf_{j\rightarrow +\infty}\|\nabla^{m} u_{j}\|_{2}^{2}\ +\theta_0= \theta_1+\theta_0\le 1.
	$$
	Hence, from \eqref{magic}
	\begin{align*}
		 AD(2m, m, \beta, \alpha, \gamma) & \geq \int_{\mathbb{R}^{2m}}\Phi\left(\beta\varrho(\|v\|_{2}^{2})v^2\right)\mathrm{d} x = \tau^{2m}\int_{\mathbb{R}^{2m}}\left(\operatorname{e}^{\beta\varrho(\theta_0)u^2_0}-1\right)\mathrm{d}x  \\
	& =\int_{\mathbb{R}^{2m}}\left(\operatorname{e}^{\beta\varrho(\theta_0)u^2_0}-1\right)\mathrm{d}x + (\tau^{2m} -1)\beta\varrho(\theta_0)\|u_{0}\|^{2}_{2} \\
	&+(\tau^{2m}-1)\int_{\mathbb{R}^{2m}} \left[\operatorname{e}^{\beta\varrho(\theta_0)u^2_0}-1-\beta\varrho(\theta_0)u^2_0\right] \mathrm{d}x \\
	& =\int_{\mathbb{R}^{2m}}\left(\operatorname{e}^{\beta\varrho(\theta_0)u^2_0}-1\right)\mathrm{d}x + \beta\varrho(\theta_0)(\theta_0-\|u_{0}\|^{2}_{2}) \\
	&+(\tau^{2m}-1)\int_{\mathbb{R}^{2m}} \left[\operatorname{e}^{\beta\varrho(\theta_0)u^2_0}-1-\beta\varrho(\theta_0)u^2_0\right] \mathrm{d}x.
	\end{align*}
By \eqref{Maxn=2m} and \eqref{resLemma2n=2m}, we obtain
	\begin{align*}
		 AD(2m, m, \beta, \alpha, \gamma) & \geq 	AD(2m,m, \beta, \alpha, \gamma)
		+(\tau^{2m}-1)\int_{\mathbb{R}^{2m}} \left[\operatorname{e}^{\beta\varrho(\theta_0)u^2_0}-1-\beta\varrho(\theta_0)u^2_0\right] \mathrm{d}x.
	\end{align*}
	Since $u_{0} \neq 0$ we need to have $\tau^{2m} = 1$ or equivalently $\theta_{0} = \|u_{0}\|_{2}^{2}$. Therefore, by \eqref{Maxn=2m} and  \eqref{resLemma2n=2m} again we have 
	$$
	AD(2m, m, \beta, \alpha, \gamma) = \int_{\mathbb{R}^{2m}}\Big(\operatorname{e}^{\beta\varrho(\|u_0\|_{2}^{2})u^2_0}-1\Big)\mathrm{d}x.
	$$
	Since $\|\nabla^{m}u_{0}\|_{2}^{2}+\|u_{0}\|_{2}^{2}\le  \liminf_{j\to\infty}(\|\nabla^{m}u_{j}\|_{2}^{2}+\|u_{j}\|_{2}^{2})\le1$, we get that $u_{0}$ maximizes $AD(2m,m, \beta, \alpha, \gamma)$.
	\subsection{Proof of Theorem \ref{ThmNotAttain}: Non-attainability in \texorpdfstring{$\frac{n}{m}=2$}{frac} case}
	This section is devoted to prove that $AD(2m, m, \beta, \alpha, \gamma)$ is not attained for $\beta$ sufficiently small. We will proceed analogously  Ishiwata and Nguyen in \cite{NguyenVanHoang2019}, \cite{Ishiwata2011}. Let $\beta < (4\pi)^{m}m!\left(\frac{1-\alpha}{2+2\alpha}\right)$. Firstly, by \eqref{adachitanakaine} with $n=2m$ we derive
	\begin{equation}
		\int_{\mathbb{R}^{2m}} \left[\exp\left({\beta \frac{|u|^2}{\|\nabla^m 	u\|^{2}_{2}}}\right)\ -1\right] \ud x \leq C_{\beta,m,2m} \frac{\|u\|_{2}^{2}}{\|\nabla^m u\|_{2}^{2}}, \quad \forall\ u\in W^{m,{2}}_{rad}(\mathbb{R}^{2m})\setminus \{0\}.
	\end{equation}
	From this we can observe that for any fixed $\tilde{\beta} < (4\pi)^{m}m!$,
	\begin{align}\label{ineqnorms}
		\frac{\tilde{\beta}^{j}}{j!}\frac{\|u\|_{2j}^{2j}}{\|\nabla ^{m}u\|_{2}^{2j}} \leq C_{2m,m, \tilde{\beta}} \frac{\|u\|^{2}_{2}}{\|\nabla^{m}u\|_{2}^{2}}, ~~ \text{for any } j\ge 1.	
	\end{align}
Let $\mathcal{H}= \{u \in W^{m,2}(\mathbb{R}^{2m}) \;:\;   \|u\|_{2}^{2}+ \|\nabla^{m} u\|_{2}^{2} = 1\}$ and let $v\in\mathcal{H}$. By using the same notation in \eqref{HT} and  \eqref{ETA},  for $t>0$  we define the family of functions
	$$
	v_{t} :=H_{t}(v)(x)= t^{\frac{1}{2}}v(t^{\frac{1}{2m}}x) ~~ \text{ and } w_{t} = \frac{v_t}{\|v_{t}\|_{2}^{2} +  \|\nabla^{m} v_{t}\|_{2}^{2}}=\frac{v_t}{\eta_{v}(t)}
	$$
where we have used that $\|v_t\|_{2}^{2} = \|v\|_{2}^{2}$ and $\|\nabla^{m} v_t\|_{2}^{2} = t\|\nabla^{m} v\|_{2}^{2}$.
If $v\in  W^{m,2}_{rad}(\mathbb{R}^{2m})$ is a maximizer for $AD(2m, m, \beta, \alpha, \gamma)$, then  $v\in\mathcal{H}\cap W^{m,2}_{rad}(\mathbb{R}^{2m})$ and each $w_t$ is a curve in $W^{m,2}_{rad}(\mathbb{R}^{2m})\cap\mathcal{H}$ such that $w_1 = v$, and
	\begin{equation}\label{nullderivative}
		\frac{d}{dt}F_{2m, m, \beta, \alpha, \gamma}(w_t) \bigg|_{t=1} = 0.
	\end{equation}
The idea is to show that the identity \eqref{nullderivative} does not occur for any $v \in W^{m, 2}_{rad}(\mathbb{R}^{2m})\cap\mathcal{H}$,  if  $\beta$ is sufficiently small. This  leads to the non-existence of a maximizer  $v$. 
	Using the relations $\|v_t\|_{2k}^{2k} = t^{k-1}\|v\|_{2k}^{2k}$ and $\|\nabla^{m} v_t\|_{2}^{2} = t\|\nabla^{m} v\|_{2}^{2}$, we obtain
	$$
	F_{2m, m, \beta, \alpha, \gamma}(w_t) = \sum_{k=1}^{\infty}\frac{\beta^k}{k!}\|v\|_{2k}^{2k}\left[\varrho\Big(\frac{\|v\|_{2}^{2}}{\eta_v(t)}\Big)\right]^{k}\frac{t^{k-1}}{\eta^{k}_{v}(t)}
	$$
	where $\varrho$ is given by \eqref{RHO}. Now, we note that 
	\begin{equation}\label{key-Id}
\begin{aligned}
&\eta_v(1)=\|\nabla^{m}v\|^{2}_{2}+\|v\|^{2}_{2}=1\\
&\varrho\Big(\frac{\|v\|_{2}^{2}}{\eta_v(1)}\Big)=\varrho(\|v\|^{2}_2)=\frac{1+\alpha\|v\|^{2}_{2}}{1-\gamma\alpha\|v\|^{2}_{2}}\\
&\varrho_1:=\frac{d}{dt}\left[\varrho\Big(\frac{\|v\|_{2}^{2}}{\eta_v(t)}\Big)\right]\Bigg|_{t=1}= -\frac{\|v\|^{2}_{2}\|\nabla^m v\|^{2}_{2}}{(1-\gamma\alpha\|v\|^{2}_{2})^2}\left[\alpha(1-\gamma\alpha\|v\|^{2}_{2})+\gamma\alpha(1+\alpha\|v\|^{2}_{2})\right]\\
& \nu_k:=\frac{d}{dt}\left[\frac{t^{k-1}}{\eta^{k}_{v}(t)}\right]\Bigg|_{t=1}=(k-1)-k\|\nabla^{m}v\|^{2}_{2}.
\end{aligned}
\end{equation}
Hence,  the identities in \eqref{key-Id} yield
\begin{equation}\nonumber
\begin{aligned}
\frac{d}{dt}\left[F_{2m, m, \beta, \alpha, \gamma}(w_t)\right] \big|_{t=1}&=\sum_{k=1}^{\infty}\frac{\beta^k}{k!}\|v\|_{2k}^{2k}\left[\varrho(\|v\|^{2}_2)\right]^{k}\left[\frac{k\varrho_1}{\varrho(\|v\|^{2}_2)}+\nu_k\right]\\
&\le\sum_{k=1}^{\infty}\frac{\beta^k}{k!}\|v\|_{2k}^{2k}\left[\varrho(\|v\|^{2}_2)\right]^{k}\nu_k,
\end{aligned}
\end{equation}
where we have used  that $k\varrho_1/\varrho(\|v\|^{2}_2)<0$, which is also a consequence of \eqref{key-Id}. Now, since $\nu_1=-\|\nabla^{m}v\|^{2}_{2}$ and $\nu_k\le k$ we can derive
\begin{align*}
&\sum_{k=1}^{\infty}\frac{\beta^k}{k!}\|v\|_{2k}^{2k}\left[\varrho(\|v\|^{2}_2)\right]^{k}\nu_k=-\beta\|v\|^{2}_{2}\|\nabla^{m}v\|^{2}_{2}\varrho(\|v\|^{2}_2)+\sum_{k=2}^{\infty}\frac{\beta^k}{k!}\|v\|_{2k}^{2k}\left[\varrho(\|v\|^{2}_2)\right]^{k}\nu_k\\
&\le -\beta\|v\|^{2}_{2}\|\nabla^{m}v\|^{2}_{2}\varrho(\|v\|^{2}_2)+\sum_{k=2}^{\infty}\frac{\beta^k}{k!}\|v\|_{2k}^{2k}\left[\varrho(\|v\|^{2}_2)\right]^{k}k\\
&=\beta\|v\|^{2}_{2}\|\nabla^{m}v\|^{2}_{2}\varrho(\|v\|^{2}_2)\left[\sum_{k=2}^{\infty}\frac{\beta^{k-1}}{(k-1)!}\left[\varrho(\|v\|^{2}_2)\right]^{k-1}\frac{\|v\|_{2k}^{2k}}{\|v\|^{2}_{2}\|\nabla^{m}v\|^{2}_{2}}-1\right].
\end{align*}
By using $\varrho(\|v\|^{2}_2)\le (1+\alpha)/(1-\gamma\alpha)$ we obtain
	$$
	\frac{d}{dt}\left[F_{2m, m, \beta, \alpha, \gamma}(w_t)\right] \big|_{t=1}\le \beta\|v\|^{2}_{2}\|\nabla^{m}v\|^{2}_{2}\varrho(\|v\|^{2}_2)\left[\sum_{k=2}^{\infty}\frac{\beta^{k-1}}{(k-1)!}\left(\frac{1+\alpha}{1-\gamma\alpha}\right)^{k-1}\frac{\|v\|_{2k}^{2k}}{\|v\|^{2}_{2}\|\nabla^{m}v\|^{2}_{2}}-1\right].
	$$
From inequality \eqref{ineqnorms} and since  $\|\nabla^{m}v\|_{2}^{2}\le 1$,   for some fixed $\tilde{\beta} < (4\pi)^m m!$ such that $\beta/\tilde{\beta} < \frac{1-\alpha}{2+2\alpha}$ we obtain
	\begin{align*}
	\frac{\beta^{k-1}}{(k-1)!}\frac{\|v\|_{2k}^{2k}}{\|v\|^{2}_{2}\|\nabla^{m}v\|^{2}_{2}}&=\frac{k}{\tilde{\beta}}\left(\frac{\beta}{\tilde{\beta}}\right)^{k-1}\left[\frac{\tilde{\beta}^{k}}{k!}\frac{\|v\|_{2k}^{2k}}{\|\nabla ^{m}v\|_{2}^{2k}}\right]\frac{\|\nabla ^{m}v\|_{2}^{2k-2}}{\|v\|^{2}_{2}}\\
	& \leq \frac{k}{\tilde{\beta}}\left(\frac{\beta}{\tilde{\beta}}\right)^{k-1}\|\nabla ^{m}v\|_{2}^{2k-4}C_{2m,m, \tilde{\beta}}	\\
	& \leq \frac{k}{\tilde{\beta}}\left(\frac{\beta}{\tilde{\beta}}\right)^{k-1}C_{2m,m, \tilde{\beta}},	
	\end{align*}
	for all $k\ge 2$. Consequently, we derive
	\begin{align*}
		\frac{d}{dt}\left[F_{2m, m, \beta, \alpha, \gamma}(w_t)\right] \big|_{t=1} &\leq\beta\|v\|^{2}_{2}\|\nabla^{m}v\|^{2}_{2}\varrho(\|v\|^{2}_2)\left[\sum_{k=2}^{\infty}\frac{k}{\tilde{\beta}}\left(\left(\frac{1+\alpha}{1-\gamma\alpha}\right)\frac{\beta}{\tilde{\beta}}\right)^{k-1}C_{2m,m, \tilde{\beta}}-1\right] \\
		&\leq \beta\|v\|^{2}_{2}\|\nabla^{m}v\|^{2}_{2}\varrho(\|v\|^{2}_2)\left[\beta\left(\frac{1+\alpha}{1-\gamma\alpha}\right)\frac{1}{\tilde{\beta}^2}\sum_{k=2}^{\infty}k\left(\frac{1}{2}\right)^{k-2}C_{2m,m, \tilde{\beta}}-1\right].
	\end{align*}
	Thus, for $\beta< \min\left\{\left(\left(\frac{1+\alpha}{1-\gamma\alpha}\right)\frac{1}{\tilde{\beta}^2}\sum_{k=2}^{\infty}k\left(\frac{1}{2}\right)^{k-2}C_{2m,m, \tilde{\beta}}\right)^{-1}, \tilde{\beta}/2\right\}$, we have 
	$$
	\frac{d}{dt}\left[F_{2m, m, \beta, \alpha, \gamma}(w_t)\right] \big|_{t=1}<0, \;\text{ for any }\; v \in W^{m, 2}_{rad}(\mathbb{R}^{2m})\cap\mathcal{H}.
	$$
	From \eqref{nullderivative}, we conclude that $AD(2m, m, \beta, \alpha, \gamma)$ does not admit maximizers if $ \beta$ is chosen as above.

	\section{Critical Case}
In this section, we apply blow-up analysis together with  a truncation argument by DelaTorre-Mancini \cite{DelaTorre} and some ideas by Chen-Lu-Zhu \cite{luluzhu20}, to obtain the attainability stated in Theorem~\ref{thm-attainn=4} under the critical regime $\beta=32\pi^{2}$. Let $(u_j)$ be a sequence in $W^{2,2}(\mathbb{R}^4)$ formed by maximizers for the subcritical supremum of $AD(4,2, \beta_j, \alpha, \gamma)$ with $\beta_j=32\pi^{2} -1/j$, ensured by Theorem~\ref{ThmAttain}, i.e., $\|u_j\|_{W^{2,2}(\mathbb{R}^4)}=1$ and  

\begin{equation}\label{maxi-sub}  
	AD(4,2, \beta_j, \alpha, \gamma)=\int_{\mathbb{R}^4}\left(\operatorname{e}^{\beta_{j}\varrho(\|u_j\|^{2}_2)u_{j}^{2}}-1\right)\mathrm{d}x=\sup_{\underset{\|\Delta u \|_{2}^{2} + \|u\|_{2}^{2} \leq 1}{u\in W^{2,2}(\mathbb{R}^4)}} \int_{\mathbb{R}^4}\left(\operatorname{e}^{\beta_{j}\varrho(\|u\|^{2}_2)u^{2}}-1\right)\mathrm{d}x,  
\end{equation}  
with $\varrho$ as in \eqref{RHO} and $\frac{ 1 +2\alpha  - \gamma\alpha^2}{1 + \alpha(1 - \gamma) - \gamma\alpha^2}\frac{2}{\mathcal{B}_{GN}}<\beta_j< 32\pi^{2}$. From \eqref{Radializar}, we can assume that each $u_j$ is a positive radially symmetric function and also suppose that $u_{j} \rightharpoonup u$ in $W^{2,2}_{rad}(\mathbb{R}^4)$.  
\begin{lemma}\label{ADlimit=AD}
We have
\begin{equation}\nonumber
AD(4,2, 32\pi^2,  \alpha, \gamma)=\lim_{j\to\infty}AD(4,2, \beta_j,  \alpha, \gamma)=\lim_{j\to\infty}\int_{\mathbb{R}^4}\left(\operatorname{e}^{\beta_{j}\varrho(\|u_j\|^{2}_2)u_{j}^{2}}-1\right)\mathrm{d}x.
\end{equation}
\end{lemma}
\begin{proof}
We first note that 
\begin{equation}\label{AD>}
\limsup_{j\to\infty}AD(4,2, \beta_j,  \alpha, \gamma)\le AD(4,2, 32\pi^2,  \alpha, \gamma).
\end{equation}
The  reverse inequality is a consequence of the  Fatou Lemma. Indeed, for any $u\in W^{2,2}(\mathbb{R}^4)$ with $\|u\|_{W^{2,2}(\mathbb{R}^4)}\le 1$ we have 
$$
\liminf_j\int_{\mathbb{R}^4}\left(\operatorname{e}^{\beta_{j}\varrho(\|u\|^{2}_2)u^{2}}-1\right)\mathrm{d}x\ge \int_{\mathbb{R}^4}\left(\operatorname{e}^{32\pi^2\varrho(\|u\|^{2}_2)u^{2}}-1\right)\mathrm{d}x
$$
which implies that 
$$
\liminf_{j\to\infty}AD(4,2, \beta_j,  \alpha, \gamma)\ge \int_{\mathbb{R}^4}\left(\operatorname{e}^{32\pi^2\varrho(\|u\|^{2}_2)u^{2}}-1\right)\mathrm{d}x.
$$
Taking the supremum over  $u\in W^{2,2}(\mathbb{R}^4)$ with $\|u\|_{W^{2,2}(\mathbb{R}^4)}\le 1$ we obtain 
\begin{equation}\label{AD<}
\liminf_{j\to\infty}AD(4,2, \beta_j,  \alpha, \gamma)\ge AD(4,2, 32\pi^2,  \alpha, \gamma).
\end{equation}
From \eqref{AD>} and \eqref{AD<} we conclude the result.
\end{proof}

A straightforward computation shows that the
Euler-Lagrange equation of $u_j$
is given by
	 \begin{equation}\label{EulerLagrange}
		\begin{cases}
			\Delta ^2 u_{j} + u_{j}= \frac{u_{j}}{\lambda_{j}}\tilde{\zeta}_{j}\operatorname{e}^{\beta_{j}\tilde{\zeta}_{j} u_{j}^2} +\mu_{j}u_j ~~ \mbox{ in } \mathbb{R}^{4}, \\
			\|\Delta u_{j}\|^{2}_{2} +  \|u_{j}\|^{2}_{2} = 1\\
			\beta_{j} = 32\pi^2-\frac{1}{j},\\
			\tilde{\zeta}_{j} =\varrho(\|u_j\|^{2}_2)=\frac{1+\alpha\|u_{j}\|^{2}_{2}}{1-\gamma\alpha\|u_{j}\|^{2}_{2}}, \\
			\mu_{j} = \frac{\alpha(1+\gamma)}{(1-\gamma\alpha\|u_{j}\|_{2}^{2})^2},\\
			\lambda_{j} =\frac{\tilde{\zeta}_{j} }{1-\mu_j\|u_j\|^{2}_{2}} \int_{\mathbb{R}^4}u^{2}_{j}\operatorname{e}^{\beta_{j}\tilde{\zeta}_{j}u^2_{j}} \mathrm{d}x.
		\end{cases}
	\end{equation}
By using Lemma~\ref{RegularityGazzola2} together with \eqref{EulerLagrange}, we can see that $u_j\in C^{\infty}(\mathbb{R}^4)$.	 From Lemma~\ref{radiallemma}, we can always take a point $x_j\in \mathbb{R}^4$ such that
\begin{equation}\label{cj-point}
c_{j} = u_{j}(x_{j}) = \max_{\mathbb{R}^{4}}|u_{j}| .
\end{equation}
We divide our argument into two cases: 
	\begin{enumerate}
		\item[(a)]  $\displaystyle\sup_{j}c_{j} < \infty$,
		\item[(b)]$c_{j} \rightarrow +\infty$, as $j\to\infty$.
	\end{enumerate}
\begin{defn}\label{NVSdef} Let $(u_j)$ be a sequence in $W^{2,2}(\mathbb{R}^4)$  such that $u_{j} \rightharpoonup u$   in $W^{2,2}(\mathbb{R}^4)$. We  say that $(u_j)$ is a normalized vanishing sequence [(NVS) in short] if   $\|u_j\|_{W^{2,2}(\mathbb{R}^4)}=1$, $u=0$ and  
$$
\lim_{\rho\to\infty}\lim_{j\to\infty}\int_{B_{\rho}}(\operatorname{e}^{\beta_j\varrho(\|u_j\|^{2}_{2})u^2_j}-1)\mathrm{d}x=0.
$$
\end{defn}
Firstly, we deal with the bounded case. 
	\begin{lemma}
    If $\sup_{j} c_j < +\infty$, then only one of the following alternatives is satisfied:
    \begin{enumerate}
        \item[$(i)$] $u \neq 0$ and $AD(4,2, 32\pi^2, \alpha, \gamma)$ is achieved by some function in $W^{2, 2}_{rad}(\mathbb{R}^{4})$;
        \item[$(ii)$] $(u_j)$ is a NVS  and  $AD(4,2, 32\pi^2, \alpha, \gamma)\leq 32\pi^2\varrho(1)$.
    \end{enumerate}
\end{lemma}
	\begin{proof}
		If $\sup_{j}c_{j} < +\infty$, then by standard elliptic estimates we conclude that $u_j \rightarrow u$ in  $C^{3}_{loc}(\mathbb{R}^{4})$, see Lemma~\ref{RegularityGazzola2}. Then, for any $\rho>1$, from Lemma~\ref{radiallemma}
		\begin{equation}\label{exp-tail-tail}
		\begin{aligned}
			\operatorname{e}^{\beta_j\varrho(\|u_j\|^{2}_{2})u^2_j}-1-\beta_j\varrho(\|u_j\|^{2}_{2})u^2_j&=\sum_{i=2}^{\infty}\frac{(\beta_j\varrho(\|u_j\|^{2}_{2})u^2_j)^{i}}{i!} \le \sum_{i=2}^{\infty}\frac{(\varrho(1)\beta_j)^{i}}{i!}u^{2i}_j\\
			& \le \sum_{i=2}^{\infty}\frac{(\varrho(1)\beta_j)^{i}}{i!}\left(\frac{C}{\rho^3}\right)^{i-1}u^{2}_j \le \frac{u^{2}_j}{C\rho}\sum_{i=2}^{\infty}\frac{(32\pi^2C\varrho(1))^{i}}{i!}\\
			&=C^{\prime}\frac{u^2_j}{\rho},
			\end{aligned}
		\end{equation}
on $\mathbb{R}^4\setminus B_{\rho}$,  where $C^{\prime}$ is  independent of $j$ and $\rho$. Thus, by \eqref{exp-tail-tail}
\begin{equation}\label{exp-partI}
\lim_{\rho\to\infty}\lim_{j\to\infty}\int_{\mathbb{R}^4\setminus B_{\rho}}\left[\operatorname{e}^{\beta_j\varrho(\|u_j\|^{2}_{2})u^2_j}-1-\beta_j\varrho(\|u_j\|^{2}_{2})u^2_j\right]\mathrm{d}x=0.
\end{equation}
Up to a subsequence,  we can assume that $\|u_j\|^{2}_{2}\to \theta$ with $\theta\in [0,1]$, as $j\to\infty$. Thus, the convergence in $C^{3}_{loc}(\mathbb{R}^{4})$ yield
\begin{equation}\label{exp-partII}
\lim_{\rho\to\infty}\lim_{j\to\infty}\int_{B_{\rho}}\left[\operatorname{e}^{\beta_j\varrho(\|u_j\|^{2}_{2})u^2_j}-1-\beta_j\varrho(\|u_j\|^{2}_{2})u^2_j\right]\mathrm{d}x=\int_{\mathbb{R}^4}\left[\operatorname{e}^{32\pi^2\varrho(\theta)u^2}-1-32\pi^2\varrho(\theta)u^2\right]\mathrm{d}x.
\end{equation}
Since 
$$
\lim_{j\to\infty}\int_{\mathbb{R}^4}\beta_j\varrho(\|u_j\|^{2}_{2})u^2_j\mathrm{d}x=32\pi^2\varrho(\theta)\theta
$$
it follows from \eqref{maxi-sub}, \eqref{exp-partI} and \eqref{exp-partII} that
\begin{equation}\label{ADlimit}
\begin{aligned}
\lim_{j\to\infty}AD(4,2, \beta_j,  \alpha, \gamma)&=\lim_{j\to\infty}\int_{\mathbb{R}^4}\left(\operatorname{e}^{\beta_{j}\varrho(\|u_j\|^{2}_2)u_{j}^{2}}-1\right)\mathrm{d}x\\
&=\int_{\mathbb{R}^4}\left(\operatorname{e}^{32\pi^2\varrho(\theta)u^2}-1\right)\mathrm{d}x+32\pi^2\varrho(\theta)(\theta-\|u\|^{2}_2).
\end{aligned}
\end{equation}
 From \eqref{ADlimit} and Lemma~\ref{ADlimit=AD}, we can write 
\begin{equation}\label{AD=}
\begin{aligned}
AD(4,2, 32\pi^2,  \alpha, \gamma)=\int_{\mathbb{R}^4}\left(\operatorname{e}^{32\pi^2\varrho(\theta)u^2}-1\right)\mathrm{d}x+32\pi^2\varrho(\theta)(\theta-\|u\|^{2}_2).
\end{aligned}
\end{equation}
If $u \neq 0$, as in \eqref{magia=1}  we define  $\tau\ge 1$ and $v$ given by
\begin{equation}\label{magia=2}
\tau^{4}=\frac{\theta}{\|u\|_{2}^2}\; \text{ and }\; v(x)=u\big(\frac{x}{\tau}\big).
\end{equation}	
Then, by \eqref{AD=}
	\begin{align*}
		 AD(4, 2, 32\pi^2, \alpha, \gamma) & \geq \int_{\mathbb{R}^{4}}\Phi\left(32\pi^2\varrho(\|v\|_{2}^{2})v^2\right)\mathrm{d} x = \tau^{4}\int_{\mathbb{R}^{4}}\left(\operatorname{e}^{32\pi^2\varrho(\theta)u^2}-1\right)\mathrm{d}x  \\
	& =\int_{\mathbb{R}^{4}}\left(\operatorname{e}^{32\pi^2\varrho(\theta)u^2}-1\right)\mathrm{d}x + (\tau^{4} -1)32\pi^2\varrho(\theta)\|u\|^{2}_{2} \\
	&+(\tau^{4}-1)\int_{\mathbb{R}^{4}} \left[\operatorname{e}^{32\pi^2\varrho(\theta)u^2}-1-32\pi^2\varrho(\theta)u^2\right] \mathrm{d}x \\
	& =AD(4,2, 32\pi^2,  \alpha, \gamma)+(\tau^{4}-1)\int_{\mathbb{R}^{4}} \left[\operatorname{e}^{32\pi^2\varrho(\theta)u^2}-1-32\pi^2\varrho(\theta)u^2\right] \mathrm{d}x .
	\end{align*}
This forces $\tau=1$.  Consequently, we conclude that  $u$ is a maximizer for $	AD(4,2,32\pi^2, \alpha, \gamma)$.

Now, suppose $u=0$.  From the convergence $u_j\to 0$ in $C^3_{loc}(\mathbb{R}^4)$ we get 
$$
\lim_{j\to\infty}\int_{B_{\rho}}(\operatorname{e}^{\beta_j\varrho(\|u_j\|^{2}_{2})u^2_j}-1)\mathrm{d}x=\int_{B_{\rho}}(\operatorname{e}^{32\pi^2\varrho(\theta)u^2}-1)\mathrm{d}x=0,\;\;\mbox{for any}\;\; \rho>0.
$$
Thus,  $(u_j)$ is a NVS according with the Definition~\ref{NVSdef}. Finally, directly from \eqref{AD=} with $u=0$, we obtain
		$$
		AD(4,2, 32\pi^2,  \alpha, \gamma)= 32\pi^2\varrho(\theta)\theta\le 32\pi^2\varrho(1).
		$$	
	\end{proof}\\	
Next, we perform the blow-up analysis to deal with the case when $c_{j}\to+\infty$  as  $j\to \infty$.
	\subsection{Blow-up Analysis}
Throughout this section we assume that $c_{j}\to+\infty$  as  $j\to \infty$. Our analysis  is based on the works \cite{luluzhu20}, \cite{NguyenVanHoang2019}. 

Firstly, we note that,  with the notation in \eqref{EulerLagrange},  the   condition $1-\mu_j\|u_j\|^{2}_{2}=0$ with $0<\|u_j\|^{2}_{2}<1$, which is undesirable in the definition of $\lambda_j$,  occurs if and only if
\begin{align*}
	1=\sqrt{\mu_j}\|u_j\|_{2} = \frac{\sqrt{\alpha(1+\gamma)}}{1-\gamma\alpha\|u_{j}\|_{2}^{2}}\|u_j\|_{2}
\end{align*}
or equivalently $\gamma\alpha\|u_j\|^{2}_2+ \sqrt{\alpha(1+\gamma)}\|u_j\|_2-1=0$ with $0<\|u_j\|_{2}<1$. Hence, the undesirable condition occurs if and only if
\begin{align}\label{raiz=norm}
\|u_j\|_2=\frac{-\sqrt{\alpha(1+\gamma)}+\sqrt{\alpha(\gamma+1)+4\gamma\alpha}}{2\gamma\alpha}.
\end{align}
 Although there exist values of  $ j$ , $\alpha $, and  $\gamma$ such that \eqref{raiz=norm} holds, this is not possible if $ j$ is sufficiently large. In fact, we will see below that  $\|u_j\|_2 \to 0$ as $j \to \infty$.
	\begin{lemma}\label{maxseqconc}  Let  $\delta_0$ be  the Dirac Measure supported at $0$. Then,  $|\Delta u_{j}|^{2}\mathrm{d}x \stackrel{\ast}{\rightharpoonup} \delta_{0}$ weakly in the sense of measure.  Further,  $u_j \to 0$ in $L^2(\mathbb{R}^2)$,   $ \tilde{\zeta}_{j} \rightarrow 1$ and $\mu_{j} \rightarrow \alpha(1+\gamma)$,  as  $j\to \infty$.
	\end{lemma}
	\begin{proof}
	By contradiction,  suppose that $|\Delta u_{j}|^{2}\mathrm{d}x \not\stackrel{\ast}{\rightharpoonup}  \delta_{0}$.  Then, there are $R>0$ and $\mu < 1$ such that 
		\begin{equation}\label{non-concentrated}
		\lim\limits_{j \rightarrow \infty}\int_{B_R}|\Delta u_{j}|^{2}\mathrm{d}x=\mu<1.
	\end{equation}
	Set 
		$\hat{u}_{j, R}(r)= u_{j}(r)-u_{j}(R)$ for  $x\in B_R$ with $r=|x|$. Then 
	$\hat{u}_{j, R} \in W_{\mathcal{N}}^{2,2}(B_{R})$, and 
	\begin{equation}\label{RR4}
	\int_{B_{R}}|\Delta \hat{u}_{j, R}|^2\mathrm{d}x = \int_{B_{R}}|\Delta u_{j}|^{2}\mathrm{d}x.
	\end{equation}
	By Lemma \ref{radiallemma}, for any $\delta>0$ we derive 
		\begin{align*}
			u_{j}^{2} &\leq (1+\delta) \hat{u}^{2}_{j, R}+c_{\delta}u_{j}^{2}(R)
			\leq (1+\delta) \hat{u}^{2}_{j, R}+c_{\delta}\frac{C}{R^{3}},
		\end{align*}
	with $c_{\delta}=(1-(1+\delta)^{-1})^{-1}$.  Now, set $\hat{v}_{j,R} := \frac{\hat{u}_{j, R}}{\|\Delta \hat{u}_{j,R}\|_{2}}$. We have  
		\begin{align*}
			\exp(\beta_{j}\tilde{\zeta}_{j}u_{j}^{2}) 
			&\leq \exp\Big(\beta_{j}\tilde{\zeta}_{j}(1+\delta) \hat{u}^2_{j, R}+c_{\delta}\beta_{j}\tilde{\zeta}_{j}\frac{C}{R^{3}}\Big) \\
			&\le  \operatorname{e}^{c_{\delta}32\pi^2\varrho(1)\frac{C}{R^{3}}}  \exp\big(\beta_{j}\tilde{\zeta}_{j}(1+\delta) \hat{u}^2_{j, R}\big)\\
	& = C(R, \delta) \exp\big(\beta_{j}\tilde{\zeta}_{j}(1+\delta) \hat{u}^2_{j, R}\big)\\
	&= C(R,\delta )\exp\left(\beta_{j}(1+\delta) \tilde{\zeta}_{j}\|\Delta \hat{u}_{j, R}\|_{2}^{2}\hat{v}_{j, R}^{2}\right),
		\end{align*}
		on the ball $B_R$. Recalling $\|\Delta u_j\|^{2}_{2}+\|u_j\|^{2}_2=1$ and $\|\Delta\hat{u}_{j,R}\|_{2}=\|\Delta u_j\|_{L^2(B_R)}\le \|\Delta u_j\|_{2}$, from \eqref{non-concentrated}  and \eqref{RR4} we  can write
		\begin{align*}
			\tilde{\zeta}_{j}\|\Delta \hat{u}_{j, R}\|_{2}^{2} &= \frac{(1+\alpha\|u_{j}\|_{2}^{2})\|\Delta \hat{u}_{j, R}\|_{2}^{2}}{(1-\gamma\alpha\|u_{j}\|_{2}^{2})} \\ 
			&= \frac{\big(1+\alpha-\alpha\|\Delta u_{j}\|_{2}^{2}\big)\|\Delta \hat{u}_{j, R}\|_{2}^{2}}{(1-\gamma\alpha\|u_{j}\|_{2}^{2})}\\
			&\leq \frac{\big(1+\alpha-\alpha\|\Delta \hat{u}_{j, R}\|_{2}^{2}\big)\|\Delta \hat{u}_{j, R}\|_{2}^{2}}{(1-\gamma\alpha+\gamma\alpha\|\Delta \hat{u}_{j,R}\|_{2}^{2})}.
		\end{align*}
	Letting $j\to \infty$, we get
		\begin{equation}\label{cotaquase1}
		\lim\limits_{j \rightarrow \infty} \left(\beta_{j}(1+\delta) \tilde{\zeta}_{j}\|\Delta \hat{u}_{j, R}\|_{2}^{2}\right) \leq 32\pi^2(1+\delta)\frac{\mu+\alpha(1-\mu)\mu}{1-\gamma\alpha(1-\mu)}.
		\end{equation}
Since the positive numbers $\gamma,\alpha$ and $1-\mu $ are less than $1$, we obtain $1-\gamma\alpha(1-\mu)>0$, thus
\begin{align}\label{cadeia}
\frac{\mu+\alpha(1-\mu)\mu}{1-\gamma\alpha(1-\mu)}<1 & \; \text{ iff }\; \; \alpha(1-\mu)\mu<(1-\mu)-\gamma\alpha(1-\mu)\nonumber \\
& \; \text{ iff } \; \alpha\mu<1-\gamma\alpha \nonumber\\
& \; \text{ iff } \; \mu+\gamma<\frac{1}{\alpha}.
\end{align}
We are assuming the condition $0<\gamma<\frac{1}{\alpha}-1$. Thus,  for any $0<\mu<1$ we obtain $\mu+\gamma<(\mu-1)+\frac{1}{\alpha}<\frac{1}{\alpha}$. So, \eqref{cadeia}  holds and from \eqref{cotaquase1}, for $\delta>0$ small enough it follows that 
$$
		\lim\limits_{j \rightarrow \infty} \left(\beta_{j}(1+\delta) \tilde{\zeta}_{j}\|\Delta \hat{u}_{j, R}\|_{2}^{2}\right)<32\pi^2.
$$
Therefore, we can apply the Adams-Trudinger-Moser type inequality \eqref{TarsiThm} by  C.~Tarsi \cite[Theorem 4]{Tarsi2012} to conclude that 	
$\exp\left(\beta_{j}(1+\delta) \tilde{\zeta}_{j}\|\Delta \hat{u}_{j, R}\|_{2}^{2}\hat{v}_{j, R}^{2}\right)$ is bounded in $L^{p}(B_{R})$ for some $p>1$. Consequently, 		$\exp(\beta_{j}\tilde{\zeta}_{j}u_{j}^{2})$ is also bounded in $L^{p}(B_{R})$. Since, $(u_{j})$ is bounded in $L^{q}(B_{R})$ for any $q<\infty$, by Hölder inequality together with Lemma \ref{lambdainf} below, it holds that
		$$
		\frac{u_{j}}{\lambda_{j}}\exp\left(\beta_{j}\tilde{\zeta}_{j} u_{j}^2\right)
		$$
		is bounded in $L^{s}(B_{R})$ for any $s>1$. Hence,  by apply  standard elliptic estimates we have that $(u_{j})$ is uniformly bounded in $B_{R/2}$ which  contradicts the hypothesis that $c_{j} \rightarrow \infty$. Thus, we need to have $|\Delta u_{j}|^{2}\mathrm{d}x \stackrel{\ast}{\rightharpoonup} \delta_{0}$. As direct consequence, since $\|\Delta u_{j}\|_{2}^{2}+\|u_{j}\|_{2}^{2}=1$, we have $\|u_{j}\|_{2}^{2} \rightarrow 0$ as $j \rightarrow \infty$. Thus, $u_{j} \rightarrow 0$ and $u \equiv 0$ in $L^{2}(\mathbb{R}^{4})$, which also implies that $\tilde{\zeta}_{j} \rightarrow 1$, $\mu_{j} \rightarrow \alpha(1+\gamma)$ as $j \rightarrow \infty$.
	\end{proof}
	\\
	Next, we  show that $(\lambda_{j})$ is bounded away from zero.
	\begin{lemma}\label{lambdainf}
		There holds $\inf_{j \rightarrow \infty }\lambda_{j} > 0$.
	\end{lemma}
	\begin{proof}
		By contradiction, suppose that $\inf_{j \rightarrow \infty }\lambda_{j}=0$.  By applying the elementary inequality $(e^t-1) \leq te^t$ for $t \geq 0$  and noticing that $ (1-\mu_{j}\|u_j\|^{2}_{2})^{-1}\ge 1$, we obtain 
		\begin{align*}
			0&= \liminf\limits_{j\rightarrow \infty} \beta_j\lambda_{j} \\
			&= \liminf\limits_{j\rightarrow \infty}\frac{\varrho(\|u_j\|^{2}_2) }{1-\mu_j\|u_j\|^{2}_{2}} \int_{\mathbb{R}^4}\beta_ju^{2}_{j}\operatorname{e}^{\beta_j\varrho(\|u_j\|^{2}_2)u^2_j}\mathrm{d}x\\
			& \geq \liminf\limits_{j \rightarrow \infty}\int_{\mathbb{R}^4}\beta_j\varrho(\|u_j\|^{2}_2)u^{2}_j\operatorname{e}^{\beta_j\varrho(\|u_j\|^{2}_2)u^2_j}\mathrm{d}x \\
			& \ge  \liminf\limits_{j \rightarrow \infty}\int_{\mathbb{R}^4}\Big(\operatorname{e}^{\beta_j\varrho(\|u_j\|^{2}_2)u^2_j}-1\Big)\mathrm{d}x \\
			& = AD(4,2,32\pi^2, \alpha, \gamma)
		\end{align*}
	which is the desired contradiction.
	\end{proof}
	\subsection{Assymptotic behavior}
	Here, we shall  proceed exactly the same way as was done in \cite{luluzhu20}.
	Since $(u_j)$ is a bounded sequence in $W^{2,2}_{rad}(\mathbb{R}^4)$, we have
	\begin{align}\label{maximizingsequenceprop}
		\begin{cases}
			u_j \rightharpoonup u \text{ in } W^{2,2}_{rad}(\mathbb{R}^4); \\
			u_j \rightarrow u \text{ in }  L^{p}(\mathbb{R}^4), ~~\forall p > 2; \\
			\beta_j \nearrow 32\pi^2.
		\end{cases}
	\end{align}
In the case $c_{j} \rightarrow \infty$,  from Lemma~\ref{radiallemma} we can assume that the sequence $(x_j)\subset \mathbb{R}^4$ in \eqref{cj-point} satisfies
	$$
	x_{j} \rightarrow 0,\;\;\mbox{as}\;\; j\to \infty.
	$$ 
	With aim to study the assymptotic behavior of $(u_{j})$ near to the blow-up point, let us define
	\begin{equation}\label{rj-definition}
	r^{4}_{j} := \frac{\lambda_{j}}{c^{2}_{j}\operatorname{e}^{\beta_j\tilde{\zeta}_{j} c_{j}^{2}}}.
	\end{equation}
	\begin{lemma}\label{assymp-rj}
	For any $\xi<32\pi^2$, we have $\displaystyle\limsup_{j\to\infty}r^{4}_{j}c^{2}_{j}\operatorname{e}^{\xi\tilde{\zeta}_{j} c_{j}^{2}}=0$.	In particular, $\lim\limits_{j\rightarrow +\infty}r^{4}_{j} = 0$
	\end{lemma}
	\begin{proof}
		Let $\xi < 32\pi^2$, then
		\begin{align*}				
			r^{4}_{j}c^{2}_{j}\operatorname{e}^{\xi\tilde{\zeta}_{j} c_{j}^{2}} &=\operatorname{e}^{(\xi -\beta_j )\tilde{\zeta}_{j}c_{j}^{2}}\frac{\tilde{\zeta}_{j} }{1-\mu_j\|u_j\|^{2}_{2}} \int_{\mathbb{R}^4}u^{2}_{j}\operatorname{e}^{\beta_{j}\tilde{\zeta}_{j}u^2_{j}}\mathrm{d}x \\
			& \leq \frac{\tilde{\zeta}_{j} }{1-\mu_j\|u_j\|^{2}_{2}}\int_{\mathbb{R}^4}u^{2}_{j}\operatorname{e}^{\beta_{j}\tilde{\zeta}_{j} u^2_{j}}\operatorname{e}^{(\xi -\beta_j )\tilde{\zeta}_{j} u_{j}^{2}}\mathrm{d}x \\
			& \leq \frac{\tilde{\zeta}_{j} }{1-\mu_j\|u_j\|^{2}_{2}}\int_{\mathbb{R}^4}u^{2}_{j}\operatorname{e}^{\xi\tilde{\zeta}_{j} u^2_{j}}  \mathrm{d}x.
		\end{align*}
From Lemma~\ref{maxseqconc} we have  $\frac{\tilde{\zeta}_{j} }{1-\mu_j\|u_j\|^{2}_{2}}\rightarrow 1$, as $j\rightarrow+\infty$. In particular,   $\frac{\tilde{\zeta}_{j} }{1-\mu_j\|u_j\|^{2}_{2}}$ is bounded and we also have
		\begin{align*}
			r^{4}_{j}c^{2}_{j}\operatorname{e}^{\xi\tilde{\zeta}_{j} c_{j}^{2}} &\leq C\Big(\int_{\mathbb{R}^4}u^{2}_{j}(\operatorname{e}^{\xi\tilde{\zeta}_{j} u^2_{j}} -1)\mathrm{d}x +\int_{\mathbb{R}^4}u^{2}_{j}\mathrm{d}x\Big)\\
			& \leq C\left(\left(\int_{\mathbb{R}^{4}}|u_j|^p \mathrm{d}x\right)^{\frac{2}{p}}\left(\int_{\mathbb{R}^{4}}(\operatorname{e}^{\xi \tilde{\zeta}_{j} u_j^2}-1)^{\frac{p}{p-2}} \mathrm{d}x\right)^{\frac{p-2}{p}}+\int_{\mathbb{R}^4}u^{2}_{j}\mathrm{d}x\right)\\
			& \leq C( p)\left(\int_{\mathbb{R}^{4}}|u_j|^p  \mathrm{d}x\right)^{\frac{2}{p}}\left(\int_{\mathbb{R}^{4}}\left(\operatorname{e}^{\frac{\xi \tilde{\zeta}_{j}  p}{p-2}u_j^2}-1\right)^{\frac{p}{p-2}} \mathrm{d}x\right)^{\frac{p-2}{p}}+C(p)\int_{\mathbb{R}^4}u^{2}_{j}\mathrm{d}x,
		\end{align*}
		for some constants $C$ and $C(p)$, which are independent of $j$. Taking into account the  Adams-Adimurthi-Druet inequality in \eqref{supUnbounded},  Lemma~\ref{maxseqconc} and \eqref{maximizingsequenceprop} we obtain the result.
	\end{proof}
	
Now, we need to define some auxiliary sequences to understand the assymptotic behavior of $u_j$ near to the blow-up point. Namely, 
	\begin{align}\label{3functions}
		\begin{cases}
			w_{j}(x) = \dfrac{u_j(x_j+r_j x)}{c_{j}} \\
			z_{j}(x) = c_{j}(u_{j}(x_j+r_j x)-c_{j})\\
			v_{j}(x)=u_{j}(x_j+r_jx)-c_j
		\end{cases}
	\end{align}
	where the all sequences are defined on the sequence of set $\Omega_j = \{x \in \mathbb{R}^4\;:\;  x_j+r_j x \in B_{1}\}$.
	\begin{lemma}\label{lemmaseqwk}
		$w_{j}(x) \rightarrow 1$ in $C^{3}_{loc}(\mathbb{R}^4)$.
	\end{lemma}
	\begin{proof}
		By the Euler-Lagrange equation \eqref{EulerLagrange}, the definition of $r_{j}$ and the fact of $w_j \leq 1$, we know that for any $R > 0$ and $x \in B_{R}(0)$, $w_{j}(x)$ satisfies
		\begin{align*}
			|\Delta^{2}(w_j(x))| &= \bigg|\frac{r_{j}^{4}}{c_{j}}(\Delta^2u_{j})(x_j +r_{j}x)\bigg| \\
			&= \bigg|\frac{r_{j}^{4}}{c_{j}}\left(\lambda_{j}^{-1}\tilde{\zeta}_{j}u_{j}(x_j +r_{j}x)\operatorname{e}^{\beta_j\tilde{\zeta}_{j}u_{j}^2(x_j +r_{j}x)}+(\mu_j-1)u_{j}(x_j +r_{j}x)\right) \bigg| \\
			&\leq \left|r_{j}^{4}\left(\lambda_{j}^{-1}\tilde{\zeta}_{j}w_{j}(x)\operatorname{e}^{\beta_{j}\tilde{\zeta}_{j}c^2_{j}}+(\mu_j-1)w_{j}(x)\right) \right| \\
			&=\left|\frac{\tilde{\zeta}_{j}w_{j}(x)}{c_{j}^{2}}+(\mu_j-1)\frac{w_{j}(x)\lambda_{j}}{c_{j}^{2}\operatorname{e}^{\beta_{j}\tilde{\zeta}_{j}c_{j}^{2}}} \right| \\
			&=\left|\frac{w_{j}(x)}{c_{j}^{2}}\left(\tilde{\zeta}_{j}+(\mu_j-1)\frac{\lambda_{j}}{\operatorname{e}^{\beta_{j}\tilde{\zeta}_{j}c_{j}^{2}}} \right)\right|\\
			&\le\frac{1}{c_{j}^{2}} \left|\left(\tilde{\zeta}_{j}+(\mu_j-1)\frac{\lambda_{j}}{\operatorname{e}^{\beta_{j}\tilde{\zeta}_{j}c_{j}^{2}}} \right)\right|   \rightarrow 0,  
		\end{align*}	
	where we have used that  $\mu_j \rightarrow \alpha(1+\gamma)$, $\tilde{\zeta}_{j}\to 1$ and 
		$$
		\frac{\lambda_{j}}{\operatorname{e}^{\beta_{j}\tilde{\zeta}_{j}c_{j}^{2}}} \leq \frac{\tilde{\zeta}_{j} }{1-\mu_j\|u_j\|^{2}_{2}} \int_{\mathbb{R}^4}u^{2}_{j}\operatorname{e}^{\beta_{j}\tilde{\zeta}_{j}(u^2_{j}-c^{2}_{j})}\mathrm{d}x \leq M\|u_{j}\|_{2}^{2} \leq C .
		$$
 Furthermore, $w_{j}(x)$ is bounded in $L^{1}_{loc}(\mathbb{R}^{4})$. By, the standard regularity theory, we got that for any $R>0$ and $0 < \kappa < 1$, the sequence of $\|w_j(x)\|_{C^{3,\kappa}(B_{R}(0))}$ is uniformly bounded for every $j$. Finally, by the Ascoli-Arzela Theorem, there exists a function $w \in C^{3}(\mathbb{R}^{4})$ such that the sequence $w_{j}(x)$ converges to $w$ in  $C^{3}(\mathbb{R}^{4})$ having the property of $\Delta^2 w(x) = 0$ for all $x \in \mathbb{R}^{4}$. Now, since $w_{j}(0)=1$,  by the Liouville's Theorem for harmonic functions we obtain $w$ is constant and equal to $1$ in $\mathbb{R}^{4}$.
	\end{proof}
	
	Next, we will summarize some results that can be found in \cite{MR2667016, Martinazzi2009, Pizzetti} and will be used throughout this work.
	\begin{lemma}[Pizzetti  \cite{Pizzetti}]\label{PizzetiFormula}
		Let $u \in C^{2m}(B_{R}(x_{0}))$ with $B_{R}(x_{0}) \subset \mathbb{R}^{n}$ and $m, n \in \mathbb{Z}_{+}$. Then  there exists constants $c_{i} = c_{i}(n)$ such that
		\begin{equation}\label{PizzetiEquation}
			\int_{B(x_{0})}u(x)\mathrm{d}x = \sum_{i=0}^{m-1}c_{i}R^{n+2i}\Delta^{i}u(x_{0}) + c_{m}R^{n+2m}\Delta^{m}u(\xi),
		\end{equation}
		for some $\xi \in B_{R}(x_{0}) $.
	\end{lemma}
	
	\begin{lemma}[Martinazzi \cite{Martinazzi2009}] \label{polynomial}
		Let $u$ satisfying the biharmonic equation $(-\Delta)^{2}u=0$ with $u \leqslant (1+|x|^{l})$ for some $l \geq 0$. Then $u$ is a polynomial degree at most $\max\{l,2\}$.
	\end{lemma}

	\begin{lemma}[Gazzola  \cite{MR2667016}]\label{RegularityGazzola1}
		Let $\Omega \subset \mathbb{R}^{n}$ be a bounded open set with smooth boundary, and take $k, m \in \mathbb{N}$, $k \geq 2m$ and $\kappa\in (0,1)$. If $u\in H^{m}(\Omega)$ is a weak solutions of the problem
		\begin{align*}
			\begin{cases}
				(-\Delta)^{m}u = f ~~ &\text{in } \Omega\\
				\partial_{v}^{i} u = h_{i}, ~~&\text{on } \partial\Omega, \;\;0 \leq i \leq m-1
			\end{cases}
		\end{align*} 
		with $f \in C^{k-2m,\kappa}(\Omega)$ and $h_i \in  C^{k-i,\kappa}(\partial \Omega)$ then $u \in  C^{k,\kappa}(\Omega)$, and there exists a constant $C = C(\Omega, k,\kappa)$ such that
		$$
		\|u\|_{C^{k,\kappa}(V)} \leq C\left(\|f\|_{C^{k-2m,\kappa}(\Omega)}+\sum_{i=0}^{m-1}\|h_{i}\|_{C^{k-i,\kappa}(\partial \Omega)}\right).
		$$
		Similarly, if $f \in C^{k-2m,\kappa}(\Omega)$ and $u $ is a weak solution of $(-\Delta)^{m}u=f$ in $\Omega$, then $u \in C_{loc}^{k,\kappa}(\Omega)$ and for any open set $V \subset\subset \Omega$, then there exists a constant $C = C(k,p, V ,\Omega)$ such that 
		$$
		\|u\|_{C^{k,\kappa}} \leq C(\|f\|_{C^{k-2m,\kappa}(\Omega)}+\|u\|_{L^{1}(\Omega)}).
		$$
	\end{lemma}	
	\begin{lemma}[Gazzola  \cite{MR2667016}]\label{RegularityGazzola2}
		Let $\Omega \subset \mathbb{R}^{n}$ be a bounded open set with smooth boundary  and take $ m,k \in \mathbb{N}$, $k \geq 2m$ and $p>1$. If $f\in W^{k-2m, p}(\Omega)$ and $u \in H^{m}(\Omega)$ is a weak solution of  of $(-\Delta)^{m}u =f$ in $\Omega$, then $u \in  W_{loc}^{k, p}(\Omega)$, and for any open set $V \subset\subset \Omega $, then there exists a constant $C=C(k,p, V, \Omega)$ such that 
		$$
		\|u\|_{W^{k,p}(V)} \leq C(\|f\|_{W^{k-2m,p}(\Omega)}+\|u\|_{L^{1}(\Omega)}).
		$$
		Similarly, if $f \in C^{k-2m,\kappa}(\Omega)$ and $u $ is a weak solution of $(-\Delta)^{m}u=f$ in $\Omega$, then $u \in C_{loc}^{k,\kappa}(\Omega)$ and for any open set $V \subset \subset \Omega$, then there exists a constant $C = C(k,p, V ,\Omega)$ such that 
		$$
		\|u\|_{C^{k,\kappa}(V)} \leq C(\|f\|_{C^{k-2m,\kappa}(\Omega)}+\|u\|_{L^{1}(\Omega)}).
		$$	
	\end{lemma}
Now, we are in a position to prove the following:
	\begin{lemma}\label{lemmaSeqvk} It holds
		$v_{j}(x) = u_{j}(x_j+r_jx)-c_j\rightarrow 0$ in $C^{3}_{loc}(\mathbb{R}^4)$.
		Therefore, 
		$$
		|\nabla^{i}u_{j}(x)| = o\left(\frac{1}{r_{j}^{i}}\right) \text{ in } B_{Rr_j}, ~~ i=1,2,3,
		$$
		for any $R>0$.
	\end{lemma}
	\begin{proof}
We can notice that $v_{j}$ satisfies the equation
		\begin{align*}
			(-\Delta)^{2}v_{j} &=\frac{ r_{j}^4}{\lambda_{j}}\tilde{\zeta_{j}}u_{j}(x_j+r_jx)\operatorname{e}^{\beta_{j}\tilde{\zeta_{j}} u_{j}^2(x_j+r_jx)}  +r_{j}^4 (\mu_j-1)u_{j}(x_j+r_jx) \\
			&=\frac{\tilde{\zeta_{j}}u_{j}(x_j+r_jx)}{c_{j}^2}\operatorname{e}^{\beta_{j}\tilde{\zeta_{j}}(u^{2}_{j}(x_{j}+r_{j}x)-c^{2}_{j})} +r_{j}^4 (\mu_j-1)u_{j}(x_j+r_jx).
		\end{align*}
		By setting $\Delta v_{j} = g_{j}$ and then $\Delta g_{j} = f_j$, where
		$$
		f_{j} =\frac{\tilde{\zeta_{j}}u_{j}(x_j+r_jx)}{c_{j}^2}\operatorname{e}^{\beta_{j}\tilde{\zeta_{j}}(u^{2}_{j}(x_{j}+r_{j}x)-c^{2}_{j})} +r_{j}^4 (\mu_j-1)u_{j}(x_j+r_jx).
		$$
		Since $(u_j)$ is bounded in $H^{2}(\mathbb{R}^4)$, it is clear that $\int_{\mathbb{R}^{4}}|g_{j}|^2\mathrm{d}x = \int_{\mathbb{R}^{4}}|\Delta u_{j}|^2\mathrm{d}x < c$. By the fact of $(f_{j})$ is bounded in $L_{loc}^{p}(\mathbb{R}^4)$ for any $p \geq 1$, by Lemma \ref{RegularityGazzola2} joint with Morrey's inequality, we got that for some $0 < \kappa < 1$,
		\begin{equation}\label{gkEstimate}
			\|g_{j}\|_{C^{1,\kappa}(B_{R})} \leq c,
		\end{equation}
		for any $R>0$. Therefore by Pizzetti's formula \eqref{PizzetiEquation}, we obtain
		$$
		\int_{B_{R}}v_{j}(x)\mathrm{d}x = c_{0}R^{8}\Delta^2v_{j}(t)+c_{1}R^{6}\Delta v_{j}(0)+c_{2}R^{4}v_{j}(0),
		$$	
		for some $t\in B_R$, where $B_{R}$ is a ball centered at origin with radius $R$.

		Now, we note that  $v_{j} \leq 0$, $v_{j}(0)=0$ and by \eqref{gkEstimate}, one can conclude that $v_j(x)$ is bounded in $L_{loc}^{1}(\mathbb{R}^4)$. Thus, by Lemma \ref{RegularityGazzola2} again, there exists a $v \in C^{3}(\mathbb{R}^{4})$ to which the sequennce $v_{j}(x)$ converges in $C^3(\mathbb{R}^4)$, satisfying  $(-\Delta)^2v=0$. Finally, by Lemma \ref{polynomial} and knowing that $v \leq 0$, we guarantee that $v$ is a polynomial of degree at most 6. Therefore, by 
		$$
		\int_{\mathbb{R}^{4}}|\Delta v|^{2}\mathrm{d}x \leq \lim\limits_{j \rightarrow \infty} \int_{\mathbb{R}^{4}}|\Delta v_{j}|^{2}\mathrm{d}x  \leq C, 
		$$
		then $v$ must be a constant. By the above estimate together with the fact of $v(0)=0$, we conclude the result.
	\end{proof}\\
	In the next lemma we will establish a gradient estimates on $B_{R_{r_j}}$, and it will be of utmost importance for our aim in studying and determining the limit behavior of $z_{j}(x)$. The proof is basically the same procedure did in \cite{luluzhu20}.
	Let us state first a theorem which can be found in \cite{Martinazzi2009}, which involves uniform estimates  for $\nabla^{2m-l}u$,  with $ 1\leq l \leq 2m-1$ of a solution  $\Delta^{m}u = f$ in the Lorentz space $L^{(\frac{n}{n-l}, \frac{1}{\alpha})}(\Omega)$,  $0\le \alpha \leq 1$.
	\begin{theorem}[Martinazzi \cite{Martinazzi2009}]\label{Thm10Martinazzi} Let $\Omega\subset\mathbb{R}^n$ be  a smooth bounded domain and let  $u$ be a  solution of $\Delta^{m}u = f \in L(\ln L)^{\alpha}$with the Dirichlet boundary condition,  for some $0 \le \alpha \leq 1$ and $n \geq 2m$. Then $\nabla^{2m-l}u  \in L^{(\frac{n}{n-l}, \frac{1}{\alpha})}(\Omega),  1 \leq l \leq 2m-1 $  and 
		\begin{equation}\label{graduInterpolation}
			\|\nabla^{2m-l} u \|_{L^{(\frac{n}{n-l}, \frac{1}{\alpha})}} \leq C\|f\|_{L(\ln L)^{\alpha}},
		\end{equation}
		where $L (\ln L)^\alpha (\Omega)$ is the Zygmund space
		\[
L (\ln L)^\alpha (\Omega) := \left\{ f \in L^1 (\Omega) : \| f \|_{L (\ln L)^\alpha} := \int_{\Omega} |f| \ln^\alpha (2 + |f|) \mathrm{d}x < \infty \right\}.
\]
	\end{theorem}
	\begin{lemma}\label{GradienteEstimates} It holds that for any $R>0$,
		$$
		c_{j}\int_{B_{Rr_j}} |\Delta u_{j}|\mathrm{d}x \leq c(Rr_j)^2.
		$$
	Furthermore,
		\begin{equation}\label{LaplacianZEst}
			\int_{B_{R}} |\Delta z_{j}|\mathrm{d}x = \frac{c_{j}}{r^2_{j}}\int_{B_{R_{r_j}}} |\Delta u_{j}|\mathrm{d}x \leq cR^{2},
		\end{equation}
		where $z_j$ is given by \eqref{3functions}.
	\end{lemma}
	\begin{proof}
		For any $R_0 > 0$, let $u_{j}^{R_{0}}$ be the  biharmonic functions that solve the problem
		$$
		\begin{cases}
			\Delta ^2 u^{R_{0}}_{j} = 0, ~~ &\text{ in }\overline{B}_{R_{0}}(x_{j}) \\
			\partial^{i}_{\nu}u^{R_{0}}_{j} = \partial^{i}_{\nu}u_{j}, & \text{ on }\partial \overline{B}_{R_{0}}(x_{j}); ~~ i = 0,1.
		\end{cases}
		$$
		By the radial lemma and the elliptic estimates (Lemma \ref{RegularityGazzola2}), we obtain
		\begin{equation}\label{BiharmonicEstimate}
			\|u_{j}^{R_{0}}\|_{C^{4}(B_{R_{0}})} < \frac{c}{R_{0}^{\tau}}, \text{ for some }\tau >0.
		\end{equation}
		Note that $u_{j} - u^{R_{0}}_{j}$ satisfies the equation
		\begin{align*}
			\begin{cases}
				\Delta^{2}(u_{j}-u^{R_{0}}_{j}) = \lambda_{j}^{-1}u_{j}\tilde{\zeta_{j}}\mathrm{e}^{\beta_{j}\tilde{\zeta_{j}}u^{2}_{j}} +(\mu_j-1) u_{j}, & \text{ in } \overline{B}_{R_{0}}(0)\\
				\partial_{v}^{i}(u_{j}-u^{R_{0}})=0, &\text{ on }\partial\overline{B}_{R_{0}}(0), i=0,1.
			\end{cases}			
		\end{align*}
		Set  $f_{j} := \lambda_{j}^{-1}u_{j}\tilde{\zeta_{j}}\exp(\beta_{j}\tilde{\zeta_{j}}u^{2}_{j}) +(\mu_j-1) u_{j}$. Then  $(f_{j})$ is bounded in $L(\ln L)^{\alpha}(B_{R_{0}})$. So, as consequence of the definitions above and joining the result in Theorem \ref{Thm10Martinazzi}, we got 
		\begin{align}\label{LorentzEstimates}
			\|\nabla^{i}(u_{j}-u^{R_{0}}_{j})\|_{L^{(\frac{4}{i},2)}} \leq C, ~~ i =1,2,3,
		\end{align}
		where  $\|\cdot\|_{L^{\left(\frac{4}{i},2\right)}}$ is the Lorentz norm. 
	
	We have the  estimate
		$$
		|\Delta^2((u_{j}-u^{R_{0}}_{j})^2)| \leq |2(u_j - u_{j}^{R_{0}})\Delta^{2}(u_j - u_{j}^{R_{0}})| + C\sum_{i=1}^{3}|\nabla^{i}(u_j - u_{j}^{R_{0}})||\nabla^{4-i}(u_j - u_{j}^{R_{0}})|.
		$$
		From  \eqref{LorentzEstimates}, the  Hölder type inequality and by O'Neil in \cite{OneilConvolution1963}, the term $\sum_{i=1}^{3}|\nabla^{i}(u_j - u_{j}^{R_{0}})||\nabla^{4-i}(u_j - u_{j}^{R_{0}})|$ is bounded in $L^{1}(B_{R_{0}})$. Now, we must prove that the term $|2(u_j - u_{j}^{R_{0}})\Delta^{2}(u_j - u_{j}^{R_{0}})|$ is also bounded in $L^{1}(B_{R_{0}})$. Indeed, we separate into two integrals as follows
		$$
		\int_{B_{R_{0}}} |2(u_j - u_{j}^{R_{0}})\Delta^{2}(u_j - u_{j}^{R_{0}})| \mathrm{d}x \leq 2\left(	\int_{B_{R_{0}}} |u_j\Delta^{2}u_j| \mathrm{d}x + \int_{B_{R_{0}}} |u_j^{R_{0}}\Delta^{2}u_j| \mathrm{d}x\right)= 2(I_{1} + I_{2}). 
		$$
		Firstly,  by the Euler-Lagrange equation \eqref{EulerLagrange} and applying integration by parts, we derive
		\begin{align*}
			I_{1}=\int_{B_{R_{0}}} |u_j\Delta^{2}u_j| \mathrm{d}x &\leq \Big|\int_{\mathbb{R}^{4}} \Big(\tilde{\zeta}_{j}\frac{u^2_{j}}{\lambda_{j}}\operatorname{e}^{\beta_{j}\tilde{\zeta}_{j}u^{2}_{j}} + \mu_j u_{j}^{2}\Big)\mathrm{d}x\Big|+ \int_{\mathbb{R}^{4}}|u_j|^{2}\mathrm{d}x  \\
			&= \Big|\int_{\mathbb{R}^{4}}u_{j}(\Delta^{2} u_{j}+u_{j}) \mathrm{d}x\Big| +\int_{\mathbb{R}^{4}}|u_{j}|^2 \mathrm{d}x \\
			&\le  \int_{\mathbb{R}^{4}}|\Delta u_{j}|^2  \mathrm{d}x +  2\int_{\mathbb{R}^{4}}|u_{j}|^2  \mathrm{d}x \leq c.
		\end{align*}	
		On the other hand,  we have
		\begin{align}\label{EstimatesLaplacianUk}
		    I_{2} = \int_{B_{R_{0}}  }|u^{R_{0}}_{j}\Delta^{2}u_{j}|\mathrm{d}x \leq c \int_{B_{R_{0}}} |u_{j}\Delta^2u_j|\mathrm{d}x + c \int_{B_{R_{0}} \cap \{|u_j| \leq 1\}} |\Delta^2u_j|\mathrm{d}x \leq c(R_{0}).
		\end{align}
		Hence, $\int_{B_{R_{0}}}|\Delta^2(u_{j}-u^{R_{0}}_{j})^2|\mathrm{d}x \leq c$. Now, we want to prove that for any $R > 0$, 
		\begin{equation}\label{SquareEst}
			\int_{B_{Rr_{j}}}\Delta((u_{j}-u^{R_{0}}_{j})^2)\mathrm{d}x \leq c(Rr_{j})^2.
		\end{equation}
        Indeed, we will proceed as \cite{Martinazzi2009}. Firstly, we claim that
         \begin{align}  \label{EstimatesLaplacianuk2}
            \|\Delta u_{j}^2\|_{L^{1}(B_{Rr_{j}})}   \leq C.
         \end{align}
        Notice that 
        \begin{align*}
            |\Delta u_{j}^2| \leq 2 |u_j\Delta u_{j}| + 2|\nabla u_{j}|^2.
        \end{align*}
   Firstly,  by the Hölder inequality and we get $$\int_{B_{Rr_{j}}}|u_j\Delta u_{j}|\mathrm{d}x \leq \|u_j\|_{L^{2}(B_{Rr_{j}})}\|\Delta u_{j}\|_{L^{2}(B_{Rr_{j}})} \leq c_{1}.$$ For the second term $2|\nabla u_{j}|^2$ it suffices to see that by Lemma \ref{lemmaSeqvk}, $|\nabla u_{j}| = o(r_{j}^{-1})$ for any $R$ and $j$ sufficienty large, and the claim is proved. Now, from  \eqref{EstimatesLaplacianUk} and  \eqref{EstimatesLaplacianuk2},  we have that \eqref{SquareEst}  holds. By \eqref{SquareEst}, \eqref{BiharmonicEstimate}, and  Lemma \ref{lemmaSeqvk}, 
		\begin{align}\label{intlaplacianuksquare}
			\int_{B_{Rr_{j}}}|\Delta(u_{j})^2|\mathrm{d}x \leq c \int_{B_{Rr_{j}}}(\Delta(u_{j}-u_{j}^{R_{0}})^2)\mathrm{d}x + o(r_{j}^{2}).
		\end{align}
		In other hand,
		\begin{align}\label{laplacianuksquare}
			c_{j}|\Delta u_{j}| \leq cu_{j}|\Delta u_{j}|\leq c(\Delta(u_{j}^2)+|\nabla u_{j}|^2) \leq c\Delta(u_{j}^2) + o\left(\frac{1}{r_{j}^2}\right).
		\end{align}
		Finally, by \eqref{intlaplacianuksquare} and \eqref{laplacianuksquare}, we conclude
		$$
		c_{j}\int_{B_{Rr_{j}}}|\Delta u_{j}|\mathrm{d}x \leq c(Rr_{j})^2.
		$$
		Thus, for any $R>0$, 
		$$
		\int_{B_{R}}|\Delta z_{j}|\mathrm{d}x = \frac{c_{j}}{r^2_{j}}\int_{B_{Rr_j}} |\Delta u_{j}|\mathrm{d}x \leq cR^2.
		$$
	\end{proof}
	
	Let us analyze the limit behavior of $z_{j}(x)$ in \eqref{3functions}.
	\begin{lemma}\label{doz}
		It holds  $z_{j}(x) \rightarrow z$ in $C^{3}_{loc}(\mathbb{R}^4)$ with $z$ satisfying the equation $(-\Delta)^2z = \exp(64\pi^2z)$. Moreover, 
		\begin{equation}\label{zform}
			z(x) =- \frac{1}{16\pi ^2}\ln\left(1+\frac{\pi}{\sqrt{6}}|x|^2\right)
		\end{equation}
		and
		$$
		\int_{\mathbb{R}^{4}}\operatorname{e}^{64\pi^2z(x)}\mathrm{d}x = 1.
		$$
	\end{lemma}
	\begin{proof}
		By the Euler-Lagrange equation \eqref{EulerLagrange}, we can notice that $z$ satisfies
		\begin{align*}
			(-\Delta)^2z_{j}+c_j r_{j}^4 (1-\mu_j)u_{j}(x_j+r_jx) &=  \frac{c_jr_{j}^4}{\lambda_{j}}\tilde{\zeta_{j}}u_{j}(x_j+r_jx) \operatorname{e}^{\beta_{j}\tilde{\zeta_{j}} u_{j}^2(x_j+r_jx)}  \\
			& = \frac{\tilde{\zeta_{j}}u_{j}(x_j+r_jx)}{c_{j}}\operatorname{e}^{\beta_{j}\tilde{\zeta_{j}} (u_{j}^{2}(x_j+r_jx) -c^2_{j} )} \\
			&= \frac{\tilde{\zeta_{j}}u_{j}(x_j+r_jx)}{c_{j}}\operatorname{e}^{\beta_{j}\tilde{\zeta_{j}} c_{j}(u_{j}(x_j+r_jx) -c_{j} )\big(\frac{u_{j}(x_j+r_jx)}{c_j} +1\big)}.
		\end{align*}
		By Lemma \ref{GradienteEstimates}, we know that $\int_{B_{R}}|\Delta z_{j}|\mathrm{d}x \leq cR^2$. Hence, by 
		 elliptic estimates in \cite{MR2667016}, we obtain $\|\Delta z_j \|_{C^{1, \alpha}_{loc}} \leq c$. By Lemma \ref{lemmaSeqvk}, there exists $z \in C^{3}(\mathbb{R}^4)$ such that $z_{j} \rightarrow z$ in $C_{loc}^{3}(\mathbb{R}^4)$ with $z$ satisfying the equation
		$$(-\Delta)^2z = \operatorname{e}^{64\pi^2 z}.$$ 
		By Fatou's Lemma, we have
		$$
		\liminf\limits_{j \rightarrow \infty}\int_{\mathbb{R}^4} \operatorname{e}^{64\pi^2 z}\mathrm{d}x \leq \lambda_{j}^{-1}\int_{\mathbb{R}^{4}} \tilde{\zeta}_{j} u_{j}^2\operatorname{e}^{\beta_j \tilde{\zeta}_{j} u^{2}_{j}}\mathrm{d}x \leq\lambda_{j}^{-1}
		\frac{\tilde{\zeta}_{j} }{1-\mu_j\|u_j\|^{2}_{2}} \int_{\mathbb{R}^4}u^{2}_{j}\operatorname{e}^{\beta_{j}\tilde{\zeta}_{j}u^2_{j}}\mathrm{d}x\leq 1.
		$$	
		Now, let us suppose by contradiction that $z(x)$ is not of the form in \eqref{zform}. Then, by \cite{Lin1998}, there exist a negative number $N$ such that $\lim\limits_{|x| \rightarrow +\infty}(-\Delta)z(x) = N$. Consequently, we would have that
		$$
		\lim\limits_{j \rightarrow +\infty} \int_{B_R} |\Delta z_{j}(x)|\mathrm{d}x = |N| \mathrm{vol}(B_{1})R^4+o(R^4)
		$$
		as  $R \rightarrow \infty$, which contradicts \eqref{LaplacianZEst}. Thus, we have \eqref{zform}. By computations as done in \cite{luluzhu20}, ones can see that
		$$
		\int_{\mathbb{R}^{4}}\operatorname{e}^{64\pi^2 z(x)}\mathrm{d}x = \int_{\mathbb{R}^{4}}\left(\frac{1}{1+\frac{\pi}{\sqrt{6}}|x|^2}\right)^4 = 1.
		$$
	\end{proof}	
	\subsection{Bi-harmonic Truncations}
In this subsection, we will  use the bi-harmonic truncation proposed in \cite{luluzhu20} which are  inspired by \cite{DelaTorre}, which in turn are generalizations of the truncation argument introduced in \cite{AdimurthiDruet2004}. Basically, for any $A > 1$, we will introduce a new function $u_{j}^{A}$ valued close to $c_{j}/A$ in a small ball centered at $x_{j}$ and coinciding with $u_{j}$ outside this ball. The main objective of this section is to study the properties of $u_{j}^{A}$. 
	\begin{lemma}[DelaTorre, Lemma 4.20 \cite{DelaTorre}]\label{ThmTruncation}
		For any $A>1$ and $j \in \mathbb{N}$, ${\Omega \subset \mathbb{R}^4}$ smooth bounded domain, there exists a radius $0 < \rho_{j}^{A}< \mathrm{dist}(x_{j}, \partial\Omega)$ and a constant $C_A$ depending on $A$, such that
		\begin{enumerate}
			\item $u_{j} \geq \dfrac{c_{j}}{A}$ in $B_{\rho_{j}^{A}}(x_{j})$;
			\item $\big|u_{j}-\dfrac{c_{j}}{A}\big| \leq \dfrac{C_{A}}{c_{j}}$ on $\partial B_{\rho_{j}^{A}}(x_{j})$;
			\item $|\nabla^{l}u_{j}| \leq \dfrac{C_{A}}{c_{j}(\rho_{j}^{A})^{l}}$ on $\partial B_{\rho_{j}^{A}}(x_{j})$ for any $1 \leq l \leq 3$;
			\item $\displaystyle\lim_{j\to\infty}\rho^{A}_j=0$ and,  if $r_j$ is defined as in \eqref{rj-definition}, then $\lim\limits_{j \rightarrow \infty}\frac{\rho_{j}^{A}}{r_{j}} =\infty$.
					\end{enumerate}
	\end{lemma}
	Let $\rho^{A}_{j}>0$ and $v^{{A}}_{j} \in C^{4}(\overline{B}_{\rho_{j}^{A}}(x_{j}))$ be the unique solution of the problem
	\begin{align*}
		\begin{cases}
			\Delta^2 v^{A}_{j} = 0, ~~ &\text{ in }B_{\rho_{j}^{A}}(x_{j}) \\
			\partial_{\nu}^{i}v^{A}_{j} = \partial_{\nu}^{i}u_{j}, ~~ &\text{ on }\partial B_{\rho_{j}^{A}}(x_{j}), ~~ i = 0,1.
		\end{cases}
	\end{align*}
	Let us consider the function
	\begin{align*} u_{j}^{A} = 
		\begin{cases}
			v^{A}_{j},  &\text{ in }B_{\rho_{j}^{A}}(x_{j}) \\
			u_{j}, ~~ &\text{ in } \mathbb{R}^{4} \backslash B_{\rho_{j}^{A}}(x_{j}).
		\end{cases}
	\end{align*}
	\begin{lemma}\label{uAA}
	For $A>1$, we have
	$$
	u^{A}_j=\frac{c_j}{A}+O\Big(\frac{1}{c_j}\Big)
	$$
	uniformly on $\overline{{B}}_{\rho^{A}_j}(x_j)$.
	\end{lemma}
	\begin{proof}
	Set $\tilde{v}_j(x)=v^{A}_{j}(x_j+\rho^{A}_jx)-\frac{c_j}{A}$ for $x\in B_1$. By elliptic estimates \cite[Proposition~A.2]{DelaTorre}, we have 
	\begin{equation}\nonumber
	\begin{aligned}
	\|v^{A}_j-\frac{c_j}{A}\|_{L^{\infty}(B_{\rho^{A}_j}(x_j))}=\|\tilde{v}_j\|_{L^{\infty}(B_1)}&\le C\big[\|\tilde{v}_j\|_{L^{\infty}(\partial B_1)}+\|\nabla \tilde{v}_j\|_{L^{\infty}(\partial B_1)}\big]\\
	&=C\big[\|{v}^{A}_j-\frac{c_j}{A}\|_{L^{\infty}(\partial B_{\rho^{A}_j}(x_j))}+\rho^{A}_j\|\nabla {v}^{A}_j\|_{L^{\infty}(\partial B_{\rho^{A}_j}(x_j))}\big]\\
	&=C\big[\|u_j-\frac{c_j}{A}\|_{L^{\infty}(\partial B_{\rho^{A}_j}(x_j))}+\rho^{A}_j\|\nabla {u}_j\|_{L^{\infty}(\partial B_{\rho^{A}_j}(x_j))}\big].
	\end{aligned}
	\end{equation}
	This together with Lemma~\ref{ThmTruncation} yields the result.
	\end{proof}
	\begin{lemma}\label{LemmaEstimateAinverse}
		For any $A > 1$, there holds
		$$
		\limsup_{j \rightarrow + \infty}\int_{\mathbb{R}^{4}}\left(|\Delta u^{A}_{j}|^2 + |u^{A}_{j}|^2\right)\mathrm{d}x \leq \frac{1}{A}.
		$$
	\end{lemma}
	\begin{proof}
	By combining \eqref{EulerLagrange} with Lemma~\ref{ThmTruncation}, we obtain $(-\Delta)^2u_j\ge 0$ in $B_{\rho^{A}_j}(x_j) $ for $j$ large enough. So, the maximum principle yields $u_j\ge u^{A}_j$ in $B_{\rho^{A}_j}(x_j)$ and,  from Lemma~\ref{uAA}, $u^{A}_j\ge 0$ on  $B_{\rho^{A}_j}(x_j)$ for $j$ large enough.  In addition, since $u^{A}_j\equiv u_j$ in $\mathbb{R}^n\setminus B_{\rho^{A}_j}(x_j)$ and in view of Lemma~\ref{uAA},  by using  $u_{j}-u_{j}^{A}$ as test function  in  \eqref{EulerLagrange} and recalling that $\rho^{A}_j/r_j\to\infty$,   for any $R > 0$ and $j$ sufficiently large, we obtain
		\begin{align*}
			&\int_{B_{\rho_{j}^{A}}(x_{j})}\big[\Delta u_{j}\Delta(u_{j}-u_{j}^{A})+u_{j}(u_{j}-u_{j}^{A})\big]\mathrm{d}x = 
			\int_{B_{\rho_{j}^{A}}(x_{j})}\frac{u_{j}\tilde{\zeta}_j}{\lambda_{j}}\big[\operatorname{e}^{\beta_{j}\tilde{\zeta}_{j}u_{j}^{2}}+\mu_j u_j\big] (u_{j}-u_{j}^{A})\mathrm{d}x\\
			& \geq \frac{\tilde{\zeta}_j}{\lambda_j}\int_{B_{Rr_j}(x_{j})}u_j \operatorname{e}^{\beta_{j}\tilde{\zeta}_{j}u_{j}^{2}}(u_{j}-u_{j}^{A})\mathrm{d}x \\
			& = \frac{\tilde{\zeta}_j}{\lambda_j}r^{4}_j\int_{B_{R}(0)}\Big(c_j+\frac{z_j}{c_j}\Big) \operatorname{e}^{\beta_{j}\tilde{\zeta}_{j}(c^2_j+2z_j+\frac{z^2_j}{c^2_j})}\Big(c_j+\frac{z_j}{c_j}-\frac{c_j}{A}+O(c^{-1}_j)\Big)\mathrm{d}x \\
			& = \tilde{\zeta}_j\int_{B_{R}(0)}\Big(1+\frac{z_j}{c^2_j}\Big) \operatorname{e}^{\beta_{j}\tilde{\zeta}_{j}(2z_j+\frac{z^2_j}{c^2_j})}\Big(1-\frac{1}{A}+\frac{z_j}{c^2_j}+O(c^{-2}_j)\Big)\mathrm{d}x.
		\end{align*}
From Lemma~\ref{lemmaSeqvk} and  Lemma~\ref{lemmaseqwk}, we have $z_j/c_j=v_j\to 0$  and $z_j/c^2_j=w_j-1\to 0$ in  $C^{3}_{loc}(\mathbb{R}^2)$. So, by  applying Lemma~\ref{doz}, we can write 
 \begin{align*}
			\int_{B_{\rho_{j}^{A}}(x_{j})}\big[\Delta u_{j}\Delta(u_{j}-u_{j}^{A})+u_{j}(u_{j}-u_{j}^{A})\big]\mathrm{d}x
			\ge  \left(1-\frac{1}{A}\right)\int_{B_{R}}\operatorname{e}^{64\pi^2z}\mathrm{d}x + o_j(1),
		\end{align*}
and  letting $R \rightarrow +\infty$, we obtain
		\begin{equation}\label{EstimateLittleo}
			\int_{B_{\rho_{j}^{A}}(x_{j})}\big[\Delta u_{j}\Delta(u_{j}-u_{j}^{A})+u_{j}(u_{j}-u_{j}^{A})\big]\mathrm{d}x \geq 1 - \frac{1}{A} + o_{j}(1).
		\end{equation}			 
	Recalling $\|\Delta u_{j}\|^{2}_{2} +  \|u_{j}\|^{2}_{2} = 1$, we have 
		\begin{align*}
			 \int_{\mathbb{R}^{4}}\left(|\Delta u^{A}_{j}|^2 + |u^{A}_{j}|^2\right)\mathrm{d}x & = \int_{B_{\rho_{j}^{A}}(x_j)}\big(|\Delta u^{A}_{j}|^2 +|u^{A}_{j}|^2\big) \mathrm{d} x +  \int_{\mathbb{R}^{4} \backslash B_{\rho_{j}^{A}}(x_j)}\big(|\Delta u_{j}|^2+|u_{j}|^2\big)\mathrm{d}x \\
			& = \int_{B_{\rho_{j}^{A}}(x_j)}\big(|\Delta u^{A}_{j}|^2 +|u^{A}_{j}|^2\big) \mathrm{d} x+1- \int_{B_{\rho_{j}^{A}}(x_j)}\big(|\Delta u_{j}|^2+|u_{j}|^2\big)\mathrm{d}x  \\
			& = 1 - \int_{B_{\rho_{j}^{A}}(x_j)}\big[\Delta u_{j}\Delta(u_j - u_j^A)  +u_{j}(u_j - u_j^A)\big]\mathrm{d}x\\
			&+\int_{B_{\rho_{j}^{A}}(x_j)}\Delta u^{A}_j\Delta\big(u^{A}_j-u_j) \mathrm{d} x +  \int_{B_{\rho_{j}^{A}}(x_j)} u_j^A(u^{A}_j-u_j) \mathrm{d}x.
		\end{align*}
	From  \eqref{EstimateLittleo} and recalling $u^{A}_j(u^{A}_j-u_j)\le 0$ in $B_{\rho_{j}^{A}}(x_j)$, we derive
		\begin{align*}
			\int_{\mathbb{R}^{4}}\left(|\Delta u^{A}_{j}|^2 + |u^{A}_{j}|^2\right)\mathrm{d}x 
			& \leq \dfrac{1}{A} + \int_{B_{\rho_{j}^{A}}(x_j)}\Delta u^{A}_{j}\Delta(u^{A}_j - u_j) \mathrm{d}x + o_j(1)  \\
			& = \dfrac{1}{A} + o_{j}(1),
		\end{align*}
	where we used integration by parts and $\Delta^2 u^{A}_j=0$ in $B_{\rho_{j}^{A}}(x_j)$ to obtain the last identity.
	\end{proof}
	\begin{lemma}\label{LemmaDoubleLimit} We have
		$$AD(4,2, 32\pi^2,  \alpha, \gamma)=\lim\limits_{j \rightarrow \infty} \int_{\mathbb{R}^{4}}(\operatorname{e}^{\beta_{j}\tilde{\zeta}_{j}u_{j}^{2}} - 1)\mathrm{d}x = \lim\limits_{\hat{R}\rightarrow \infty}\lim\limits_{j \rightarrow \infty} \int_{B_{\hat{R} r_{j}}}(\operatorname{e}^{\beta_{j}\tilde{\zeta}_{j}u_{j}^{2}} - 1)\mathrm{d}x = \lim\limits_{j \rightarrow \infty} \frac{\lambda_{j}}{c^{2}_{j}}$$
	 and consequently $$\frac{\lambda_{j}}{c_{j}}\rightarrow \infty\;\;\mbox{and}\;\; \sup_{j}\frac{c_{j}^{2}}{\lambda_{j}} < \infty$$
	\end{lemma}
	\begin{proof}
	The first identity has already been proved in Lemma~\ref{ADlimit=AD}. Now,  let us write
		\begin{equation}\label{SpliT}
		\begin{aligned}
		\int_{\mathbb{R}^{4}}(\operatorname{e}^{\beta_{j}\tilde{\zeta}_{j}u_{j}^{2}} - 1)\mathrm{d}x &=  \int_{B_{\rho_{j}^{A}}(x_{j})}(\operatorname{e}^{\beta_{j}\tilde{\zeta}_{j}u_{j}^{2}} - 1) \mathrm{d}x+\int_{\mathbb{R}^{4} \backslash B_{\rho_{j}^{A}}(x_{j})}(\operatorname{e}^{\beta_{j}\tilde{\zeta}_{j}u_{j}^{2}} - 1)\mathrm{d}x\\
			& \leq \int_{B_{\rho_{j}^{A}}(x_{j})}(\operatorname{e}^{\beta_{j}\tilde{\zeta}_{j}u_{j}^{2}} - 1)\mathrm{d}x + \int_{\mathbb{R}^{4}}(\operatorname{e}^{\beta_{j}\tilde{\zeta}_{j}(u^{A}_{j})^{2}} - 1)\mathrm{d}x.
		\end{aligned}
		\end{equation}
		By Radial Lemma~\ref{radiallemma}, for some radius $\hat{R}$ such that $u_{j} \leq 1$ on $\mathbb{R}^4 \backslash B_{\hat{R}}$, we get
		\begin{align}\label{foraBR}
			 \int_{\mathbb{R}^{4} \backslash B_{\hat{R}}}(\operatorname{e}^{\beta_{j}\tilde{\zeta}_{j}u_{j}^{2}} - 1)\mathrm{d}x  \leq  C\int_{\mathbb{R}^{4}}u_{j}^{2}\mathrm{d}x\to 0,\;\;\mbox{as}\;\; j\to \infty.
		\end{align}
		Note that $\beta_{j}\tilde{\zeta_{j}} \rightarrow 32\pi^2$, which together with  Lemma \ref{LemmaEstimateAinverse} and by  Tarsi's Adams inequality \eqref{TarsiThm}, implies
		$$
		\sup_{j\rightarrow \infty} \int_{B_{\hat{R}}}(\operatorname{e}^{\beta_{j}\tilde{\zeta_{j}}q'(u_{j}^{A}-u_{j}(\hat{R}))^{2}}-1)\mathrm{d}x < \infty
		$$
		for any $q' < A^2$ and $j$ sufficiently large.  In addition, we have
		$$
		q(u_{j}^{A})^2 \leq q' (u_{j}^{A}- u_{j}(\hat{R}))^2+c(q,q'), \text{ for } q<q',
		$$
		and so 
		$$
		\limsup_{j \rightarrow \infty} \int_{B_{\hat{R}}}(\operatorname{e}^{\beta_{j}\tilde{\zeta_{j}}q(u_{j}^{A})^{2}}-1)\mathrm{d}x < \infty
		$$
		for any $q<A^2$. Now, recalling $A>1$, then $\exp(\beta_{j} \tilde{\zeta_{j}} (u_{j}^{A})^2)$ is uniformly integrable and $u_j \rightarrow u=0$ a.e. Then, the Vitali's Convergence Theorem provides
		\begin{equation}\label{Vital}
		\lim_{j\rightarrow \infty} \int_{B_{\hat{R}}}(\operatorname{e}^{\beta_{j} \tilde{\zeta_{j}} (u_{j}^{A})^2}-1)\mathrm{d}x = 0.
		\end{equation}
		Hence,  from \eqref{SpliT}, \eqref{foraBR} and recalling that  Lemma~\ref{ThmTruncation}-(4) yields  $\rho^{A}_j\to 0$, we obtain 
		\begin{equation}\label{Similar}
		\begin{aligned}
			\int_{\mathbb{R}^{4}}(\operatorname{e}^{\beta_{j}\tilde{\zeta}_{j}u_{j}^{2}} - 1)\mathrm{d}x  & = \int_{B_{\rho_{j}^{A}}(x_{j})}(\operatorname{e}^{\beta_{j}\tilde{\zeta}_{j}u_{j}^{2}} - 1) \mathrm{d}x+o_j(1)\\
			& = \int_{B_{\rho_{j}^{A}}(x_{j})}\operatorname{e}^{\beta_{j}\tilde{\zeta}_{j}u_{j}^{2}} \mathrm{d}x+o_j(1)\\
		\end{aligned}
		\end{equation}
		and from Lemma~\ref{ThmTruncation}-(1) 
		\begin{equation}\label{Eqleftside}
		\begin{aligned}
		\lim_{j\to\infty}\int_{\mathbb{R}^{4}}(\operatorname{e}^{\beta_{j}\tilde{\zeta}_{j}u_{j}^{2}} - 1)\mathrm{d}x & =\lim_{j\to\infty}\int_{B_{\rho_{j}^{A}}(x_{j})}\operatorname{e}^{\beta_{j}\tilde{\zeta}_{j}u_{j}^{2}} \mathrm{d}x\\
			& \le \lim_{j\to\infty}\frac{A^2}{c_{j}^2} \int_{B_{\rho_{j}^{A}}(x_{j})}u^2_j\operatorname{e}^{\beta_{j}\tilde{\zeta}_{j}u_{j}^{2}} \mathrm{d}x\\
			&\leq  \lim_{j\rightarrow \infty} \frac{A^2}{c_{j}^2} \int_{\mathbb{R}^4}u_{j}^{2}\operatorname{e}^{\beta_{j}\tilde{\zeta}_{j} u_{j}^2}\mathrm{d}x \\ 
			&= A^2 \lim_{j\rightarrow \infty} \frac{\lambda_{j}}{c_{j}^2}\frac{1-\mu_j\|u_j\|^2_2}{\tilde{\zeta_{j}}}\\
			&= A^2 \lim_{j\rightarrow \infty} \frac{\lambda_{j}}{c_{j}^{2}}.
		\end{aligned}
		\end{equation}
By taking $A \rightarrow 1^+$ in \eqref{Eqleftside}, we obtain 
\begin{equation}\nonumber
		\begin{aligned}
		\lim_{j\to\infty}\int_{\mathbb{R}^{4}}(\operatorname{e}^{\beta_{j}\tilde{\zeta}_{j}u_{j}^{2}} - 1)\mathrm{d}x \le  \lim_{j\rightarrow \infty} \frac{\lambda_{j}}{c_{j}^{2}}.
		\end{aligned}
		\end{equation}
To reverse inequality, we note that 
\begin{equation}\nonumber
\begin{aligned}
\lambda_{j} & =\frac{\tilde{\zeta}_{j} }{1-\mu_j\|u_j\|^{2}_{2}}\Big[ \int_{\mathbb{R}^4}u^{2}_{j}\Big(\operatorname{e}^{\beta_{j}\tilde{\zeta}_{j}u^2_{j}}-1\Big)\mathrm{d}x+\int_{\mathbb{R}^2}u^2_j\mathrm{d}x\Big]\\
& \le \frac{\tilde{\zeta}_{j}}{1-\mu_j\|u_j\|^{2}_{2}}\Big[ \int_{\mathbb{R}^4}c^2_j\Big(\operatorname{e}^{\beta_{j}\tilde{\zeta}_{j}u^2_{j}}-1\Big)\mathrm{d}x+o_j(1)\Big] 
\end{aligned}
\end{equation}
and it follows that 
\begin{equation}\nonumber
\begin{aligned}
\lim_{j\to \infty}\frac{\lambda_{j}}{c^2_j} 
& \le \lim_{j\to \infty}\int_{\mathbb{R}^4}\Big(\operatorname{e}^{\beta_{j}\tilde{\zeta}_{j}u^2_{j}}-1\Big)\mathrm{d}x.
\end{aligned}
\end{equation}
Now, note that 
\begin{equation}
\label{foramedidaB}
\begin{aligned}
\lim_{j\to\infty} \int_{B_{\hat{R}r_{j}}(x_j)}(\operatorname{e}^{\beta_{j}\tilde{\zeta}_{j}u_{j}^{2}} - 1)\mathrm{d}x & = \lim_{j\to\infty}\Big[\int_{B_{\hat{R}r_{j}}(x_j)}\operatorname{e}^{\beta_{j}\tilde{\zeta}_{j}u_{j}^{2}}\mathrm{d}x+|B_{\hat{R}r_{j}}(x_j)|\Big]\\
&=\lim_{j\to\infty}\int_{B_{\hat{R}r_{j}}(x_j)}\operatorname{e}^{\beta_{j}\tilde{\zeta}_{j}u_{j}^{2}}\mathrm{d}x
\end{aligned}
\end{equation}
for any $\hat{R}>0$. By  using the definition of $r_j$  and $u_{j}^{2}(x_{j}+r_{j}x)-c_j^2=z_j(w_j+1)$ on $B_{\hat{R}r_{j}}(x_j)$, we can write
\begin{equation}
 \label{Eqrightside}
\begin{aligned}
	 \int_{B_{\hat{R}r_{j}}(x_j)}\operatorname{e}^{\beta_{j}\tilde{\zeta}_{j}u_{j}^{2}} \mathrm{d}x    & = \frac{\lambda_j}{c_{j}^{2}}\int_{B_{\hat{R} }(0)}\operatorname{e}^{\beta_{j}\tilde{\zeta}_{j}[u_{j}^{2}(x_{j}+r_{j}x)-c_j^2]}\mathrm{d}x \
	 &=  \frac{\lambda_j}{c_{j}^{2}}\int_{B_{\hat{R} }(0)}\operatorname{e}^{\beta_{j}\tilde{\zeta}_{j}z_j(w_j+1)}\mathrm{d}x.		
		\end{aligned}
		\end{equation}
Taking into account \eqref{foramedidaB}, \eqref{Eqrightside}, Lemma~\ref{lemmaseqwk} 
and Lemma~\ref{doz}, we get 
\begin{equation}\nonumber
\begin{aligned}
\lim_{j\to\infty} \int_{B_{\hat{R}r_{j}}(x_j)}(\operatorname{e}^{\beta_{j}\tilde{\zeta}_{j}u_{j}^{2}} - 1)\mathrm{d}x & =\lim\limits_{j \rightarrow \infty} \frac{\lambda_j}{c_{j}^{2}}\Big(\int_{B_{\hat{R}}(0)}\operatorname{e}^{64\pi^2 z}\mathrm{d}x \Big).
\end{aligned}
\end{equation}
Then, from   Lemma~\ref{doz} again
\begin{align*}
			\lim\limits_{\hat{R} \rightarrow \infty} \lim\limits_{j \rightarrow \infty}  \int_{B_{\hat{R}r_{j}}(x_j)}(\operatorname{e}^{\beta_{j}\tilde{\zeta}_{j}u_{j}^{2}} - 1)\mathrm{d}x  &= \lim\limits_{j \rightarrow \infty} \frac{\lambda_j}{c_{j}^{2}}.	
		\end{align*}
	\end{proof}
	
Let us define
	\begin{align*}
		&\xi_{R,j}=  \frac{\lambda_j}{\int_{B_{R}(x_{j})}|u_j|\operatorname{e}^{\beta_{j}\tilde{\zeta}_{j} u_{j}^2}\mathrm{d}x}, \\
		&\tau = \lim\limits_{j \rightarrow \infty} \frac{\xi_{j}}{c_{j}}, \;\;\mbox{with}\;\; \xi_{j}=\xi_{\rho_{j}^{A},j},\\
		&\varphi =  \lim\limits_{R \rightarrow \infty}\lim\limits_{j \rightarrow \infty} \frac{\int_{B_{R}(x_{j})}u_{j}\operatorname{e}^{\beta_{j}\tilde{\zeta}_{j} u_{j}^2} \mathrm{d}x}{\int_{B_{R}(x_{j})}|u_{j}|\operatorname{e}^{\beta_{j}\tilde{\zeta}_{j} u_{j}^2}\mathrm{d}x}.
	\end{align*}
	\begin{lemma}\label{LemmaPhi} We have
		$\varphi = 1$.
	\end{lemma}
	\begin{proof}
		For all $A>1$ and $R>0$, it holds
		\begin{align*}
			\int_{B_{R}(x_{j})}u_{j}\operatorname{e}^{\beta_{j}\tilde{\zeta}_{j} u_{j}^2} \mathrm{d}x&= \int_{B_{\rho_{j}^{A}}(x_j)}u_{j}\operatorname{e}^{\beta_{j}\tilde{\zeta}_{j} u_{j}^2} \mathrm{d}x +  \int_{B_{R}(x_{j}) \backslash B_{\rho_{j}^{A}}(x_j)}u^{A}_{j}\operatorname{e}^{\beta_{j}\tilde{\zeta}_{j} {(u^{A}_{j}})^2} \mathrm{d}x.
		\end{align*}
From Lemma~\ref{LemmaEstimateAinverse}, we can conclude that $\exp(\beta_{j}\tilde{\zeta_{j}}(u^{A}_{j})^2)$  is bounded in $L^{p}(B_{R}(x_{j}) \backslash B_{\rho_{j}^{A}}(x_j))$ for some $p>1$.  Thus,
		\begin{equation}\label{rabovai0}
			  \int_{B_{R}(x_{j}) \backslash B_{\rho_{j}^{A}}(x_j)}u^{A}_{j}\operatorname{e}^{\beta_{j}\tilde{\zeta}_{j} {(u^{A}_{j}})^2} \mathrm{d}x=o_j(1).
		\end{equation}
	which implies 
		\begin{equation}\label{Est1EqualityInBalls}
			\int_{B_{R}(x_{j})}u_{j}\operatorname{e}^{\beta_{j}\tilde{\zeta}_{j} u_{j}^2} \mathrm{d}x=\int_{B_{\rho_{j}^{A}}(x_j)}|u_{j}|\operatorname{e}^{\beta_{j}\tilde{\zeta}_{j} u_{j}^2} \mathrm{d}x + o_{j}(1),
		\end{equation}
		where we have used that $u_j$ is positive  in $B_{\rho_{j}^{A}}(x_{j})$. Analogously, 
		\begin{equation}\label{Est2EqualityInBallsModulus}
			\int_{B_{R}(x_{j})}|u_{j}|\operatorname{e}^{\beta_{j}\tilde{\zeta}_{j} u_{j}^2} \mathrm{d}x =\int_{B_{\rho_{j}^{A}}(x_j)}|u_{j}|\operatorname{e}^{\beta_{j}\tilde{\zeta}_{j} u_{j}^2}\mathrm{d}x  + o_{j}(1).
		\end{equation}
Note that 
 \begin{equation}\nonumber
 \begin{aligned}
  \int_{B_{\rho_{j}^{A}}(x_j)}\operatorname{e}^{\beta_{j}\tilde{\zeta}_{j} u_{j}^2} \mathrm{d}x&=\int_{B_{\rho_{j}^{A}}(x_j)}\big(\operatorname{e}^{\beta_{j}\tilde{\zeta}_{j} u_{j}^2} -1\big)+1\big) \mathrm{d}x\\
  &=\int_{B_{\rho_{j}^{A}}(x_j)}\big(\operatorname{e}^{\beta_{j}\tilde{\zeta}_{j} u_{j}^2} -1\big)\mathrm{d}x +o_j(1)
 \end{aligned}
 \end{equation}
Hence,  from Lemma~\ref{ThmTruncation}-(4) and Lemma~\ref{LemmaDoubleLimit}, we have 
\begin{equation}\label{bolina>D}
 \begin{aligned}
 \lim_{j\to\infty} \int_{B_{\rho_{j}^{A}}(x_j)}\operatorname{e}^{\beta_{j}\tilde{\zeta}_{j} u_{j}^2} \mathrm{d}x
  &= \lim_{j\to\infty} \int_{B_{\rho_{j}^{A}}(x_j)}\big(\operatorname{e}^{\beta_{j}\tilde{\zeta}_{j} u_{j}^2} -1\big)\mathrm{d}x\\
  & \ge \lim\limits_{\hat{R}\rightarrow \infty}\lim\limits_{j \rightarrow \infty} \int_{B_{\hat{R} r_{j}}}(\operatorname{e}^{\beta_{j}\tilde{\zeta}_{j}u_{j}^{2}} - 1)\mathrm{d}x\\
  &=AD(4,2, 32\pi^2,  \alpha, \gamma)>0.
 \end{aligned}
 \end{equation}
On the other hand, from  Lemma~\ref{ThmTruncation}-(1)  it is easy to see that 
\begin{equation}\label{>bolina>}
\begin{aligned}
c_j\int_{B_{\rho_{j}^{A}}(x_j)} \operatorname{e}^{\beta_{j}\tilde{\zeta}_{j} u_{j}^2}\mathrm{d}x & \ge \int_{B_{\rho_{j}^{A}}(x_j)}|u_{j}|\operatorname{e}^{\beta_{j}\tilde{\zeta}_{j} u_{j}^2} \mathrm{d}x \\
&\ge \frac{c_j}{A}\int_{B_{\rho_{j}^{A}}(x_j)}\operatorname{e}^{\beta_{j}\tilde{\zeta}_{j} u_{j}^2}\mathrm{d}x .
\end{aligned}
\end{equation}
Combining \eqref{Est1EqualityInBalls}, \eqref{Est2EqualityInBallsModulus}, \eqref{bolina>D} and \eqref{>bolina>} we get 
		\begin{align*}
			\frac{1}{A} + o_{j}(1) \leq \frac{\int_{B_{R}(x_{j})}u_{j}\operatorname{e}^{\beta_{j}\tilde{\zeta}_{j} u_{j}^2} \mathrm{d}x}{\int_{B_{R}(x_{j})}|u_{j}|\operatorname{e}^{\beta_{j}\tilde{\zeta}_{j} u_{j}^2} \mathrm{d}x}\leq 1.
		\end{align*}
	Letting  $j\to\infty$, $R \rightarrow \infty$ and $A\rightarrow 1$, we obtain the result.
	\end{proof}
	\begin{lemma}\label{LemmaTau}
		$\tau= 1$.
	\end{lemma}
	\begin{proof} 
Fix $R>0$.  From  Lemma~\ref{ThmTruncation}-(4),  for $j$ large enough,  we have 
\begin{equation}\label{T<1}
 \begin{aligned}
 \int_{B_{\rho_{j}^{A}}(x_j)}u^2_j\operatorname{e}^{\beta_{j}\tilde{\zeta}_{j} u_{j}^2} \mathrm{d}x & \ge  \int_{B_{Rr_j}(x_j)}u^2_j\operatorname{e}^{\beta_{j}\tilde{\zeta}_{j}u_{j}^{2}}\mathrm{d}x\\
 & =\lambda_j\int_{B_{R}(0)}w^2_j\operatorname{e}^{\beta_{j}\tilde{\zeta}_{j}z_j(w_j+1)}\mathrm{d}x.
 \end{aligned}
 \end{equation}
From \eqref{T<1}, we get
\begin{align*}
\lim_{j\to\infty}\frac{\xi_j}{c_j}& = \lim_{j\to\infty}\frac{1}{c_j}\frac{\lambda_j}{\int_{B_{\rho_{j}^{A}}(x_j)}|u_{j}|\operatorname{e}^{\beta_{j}\tilde{\zeta}_{j} u_{j}^2}\mathrm{d}x } \le  \lim_{j\to\infty}\frac{\lambda_j}{\int_{B_{\rho_{j}^{A}}(x_j)}u^2_{j}\operatorname{e}^{\beta_{j}\tilde{\zeta}_{j} u_{j}^2}\mathrm{d}x }\\
& \le  \lim_{j\to\infty}\frac{1}{\int_{B_{R}(0)}w^2_{j} \operatorname{e}^{\beta_{j}\tilde{\zeta}_{j} z_j(w_j+1)}\mathrm{d}x }=\frac{1}{\int_{B_{R}(0)} \operatorname{e}^{64\pi^{2} z}\mathrm{d}x}.
\end{align*}
Letting $R\to\infty$ and using Lemma~\ref{doz}, we obtain  $\lim_{j\to\infty}{\xi_j}/{c_j} \leq  1$, which yields $\tau\le 1$.  Now, analogous to \eqref{foraBR}  and \eqref{Est1EqualityInBalls} we can write 
	\begin{equation}\label{Ra>}
	\begin{aligned}
			\int_{\mathbb{R}^4}u^{2}_{j}\operatorname{e}^{\beta_{j}\tilde{\zeta}_{j} u_{j}^2}\mathrm{d}x & \ge \int_{B_{\rho_{j}^{A}}(x_j)}u^2_{j} \operatorname{e}^{\beta_{j}\tilde{\zeta}_{j} u_{j}^2} \mathrm{d}x + o_{j}(1)\\
			& \ge \frac{c^2_j}{A^2} \int_{B_{\rho_{j}^{A}}(x_j)}\operatorname{e}^{\beta_{j}\tilde{\zeta}_{j} u_{j}^2} \mathrm{d}x + o_{j}(1)\\
			& = \frac{c^2_j}{A^2}\Big[ \int_{B_{\rho_{j}^{A}}(x_j)}\operatorname{e}^{\beta_{j}\tilde{\zeta}_{j} u_{j}^2} \mathrm{d}x + o_{j}(1)\Big].
	\end{aligned}
		\end{equation}	
In addition, 
\begin{equation}\label{Ra<}
\int_{B_{\rho_{j}^{A}}(x_j)}|u_{j}|\operatorname{e}^{\beta_{j}\tilde{\zeta}_{j} u_{j}^2}\mathrm{d}x \le c_j\int_{B_{\rho_{j}^{A}}(x_j)}\operatorname{e}^{\beta_{j}\tilde{\zeta}_{j} u_{j}^2} \mathrm{d}x.
\end{equation}
By combining \eqref{Ra<} and \eqref{Ra>}, we can write 
\begin{align*}
\frac{\xi_j}{c_j}&=\frac{1}{c_j}\frac{\frac{\tilde{\zeta}_{j} }{1-\mu_j\|u_j\|^{2}_{2}} \int_{\mathbb{R}^4}u^{2}_{j}\operatorname{e}^{\beta_{j}\tilde{\zeta}_{j}u^2_{j}} \mathrm{d}x}{\int_{B_{\rho_{j}^{A}}(x_j)}|u_{j}|\operatorname{e}^{\beta_{j}\tilde{\zeta}_{j} u_{j}^2}\mathrm{d}x }\\
& \ge \frac{1}{A^2}\frac{\frac{\tilde{\zeta}_{j} }{1-\mu_j\|u_j\|^{2}_{2}} \Big[\int_{B_{\rho_{j}^{A}}(x_j)}\operatorname{e}^{\beta_{j}\tilde{\zeta}_{j} u_{j}^2} \mathrm{d}x + o_{j}(1)\Big]}{\int_{B_{\rho_{j}^{A}}(x_j)}\operatorname{e}^{\beta_{j}\tilde{\zeta}_{j} u_{j}^2} \mathrm{d}x}.
\end{align*}
Thus, 
$$
\tau=\lim_{j\to \infty}\frac{\xi_j}{c_j}\ge \frac{1}{A^2}.		
$$
Letting $A \rightarrow 1$ we get $\tau \ge 1$ which yields the desired result.
	\end{proof}
	\subsection{Assymptotic behavior of \texorpdfstring{$u_j$}{uj} away from the blow-up point}
	We recall the fundamental solution of the biharmonic operator $(\Delta^2 + \kappa^2)$, $\kappa>0$ in $\mathbb{R}^{4}$ whose properties below can be found on \cite{DengLi2007}. The fundamental solution $\Phi_{\kappa}(x, y)$ is the solution of the equation 
	$$
	(\Delta^2 + \kappa^2)\Phi_{\kappa}(x,y) = \delta_{x}(y), ~~ \text{in } \mathbb{R}^4.
	$$
We recall (cf. \cite[Theorem 2.4]{DengLi2007}) that for every solution $u \in H^{2}(\mathbb{R}^{4}) \cap C^{4}(\mathbb{R}^{4})$ of the equation $(\Delta^2 + \kappa^2)u = f$ we can write 
	\begin{equation}\label{SolutionsRep}
		u(x) = \int_{\mathbb{R}^{4}}\Phi_{\kappa}(x,y)f(y)\mathrm{d}y,\;\; x\in\mathbb{R}^4.
	\end{equation}
Let us now provide a few estimates for $\Phi_{\kappa}$, which will play a key role in what follows.
	\begin{flalign}
		&|\Phi_{\kappa}(x,y)| \leq c\ln\left(1+|x-y|^{-1}\right), \label{AbsEstimate1} \\
		&|\nabla^{i}\Phi_{\kappa}(x,y)| \leq c\left(|x-y|^{-1}\right), \text{ for } i\geq 1, ~~\forall x,y \in \mathbb{R}^4, x\neq y \text{ with } |x-y| \rightarrow 0,\label{AbsEstimate2}\\
		&|\nabla^{i}\Phi_{\kappa}(x,y)| = o\left(\exp\left(-\frac{\sqrt{\kappa}}{\sqrt{2}}|x-y|\right)\right), \text{ for } i= 0,1,2, ~~ \forall x,y \in \mathbb{R}^4, x\neq y \text{ with } |x-y| \rightarrow \infty. \label{AbsEstimate3}
	\end{flalign}
	\begin{lemma}\label{BoundnessLemma}
		For any $1<p<2$, we have $(c_j u_j)$ is bounded in $W^{2,p}(\mathbb{R}^4)$. 
	\end{lemma}
	\begin{proof}
		Let $v_{j}$ be the solution for the equation 
		\begin{equation}\label{E-auxiliar}
			\Delta ^2 v_{j} +\kappa^{2}_j v_{j}=\frac{\xi_j}{\lambda_j}u_{j}\tilde{\zeta}_{j}\operatorname{e}^{\beta_{j}\tilde{\zeta}_{j} u_{j}^2}, \;\;\mbox{in}\;\;  \mathbb{R}^4,
		\end{equation}
		where $\kappa_j=(1-\mu_j)^{1/2}\to (1-\alpha(1+\gamma))^{1/2}>0$ because we are assuming $\gamma<\frac{1}{\alpha}-1$. Hence, for $j$ large enough, the  representation formula  \eqref{SolutionsRep} yields
		$$
		v_{j}(x) = \frac{\xi_j}{\lambda_j}\int_{\mathbb{R}^{4}}\Phi_{\kappa_j}(x,y)u_{j}(y)\tilde{\zeta}_j\operatorname{e}^{\beta_{j}\tilde{\zeta}_{j} u_{j}^2(y)}\mathrm{d}y, \;\; x \in \mathbb{R}^{4}.
		$$
		Computing the $i$-th gradient 
		$$
		|\nabla^{i}v_{j}| = \left|\frac{\xi_j}{\lambda_{j}}\int_{\mathbb{R}^{4}}\nabla^{i}\Phi_{\kappa_j}(x,y) u_{j}(y)\tilde{\zeta}_{j}\operatorname{e}^{\beta_{j}\tilde{\zeta}_{j} u_{j}^2(y)}\mathrm{d}y\right|. 
		$$
		By the definition of $\xi_{j}/\lambda_j$, we have 
		\begin{align*}
			|\nabla^{i}v_{j}| = \left|\int_{\mathbb{R}^{4}}\nabla^{i}\Phi_{\kappa_j}(x,y)\frac{u_{j}(y)\operatorname{e}^{\beta_{j}\tilde{\zeta}_{j} u_{j}^2(y)}}{\int_{B_{\rho^{A}_j}(x_{j})}|u_j(z)|\operatorname{e}^{\beta_{j}\tilde{\zeta}_{j} u_{j}^2(z)}\mathrm{d}z}\mathrm{d}y\right|.		 
		\end{align*}
		Letting $R\to\infty$ in \eqref{Est2EqualityInBallsModulus}, we have 
		\begin{equation}\nonumber
			\int_{\mathbb{R}^4}|u_{j}(z)|\operatorname{e}^{\beta_{j}\tilde{\zeta}_{j} u_{j}^2(z)}\mathrm{d}z =\int_{B_{\rho_{j}^{A}}(x_j)}|u_{j}(z)|\operatorname{e}^{\beta_{j}\tilde{\zeta}_{j} u_{j}^2(z)}\mathrm{d}z  + o^{+}_{j}(1).
		\end{equation}
Then
\begin{align*}
			|\nabla^{i}v_{j}| \le\frac{1}{\int_{\mathbb{R}^4}g_j(z)\mathrm{d}z} \int_{\mathbb{R}^{4}}\left|\nabla^{i}\Phi_{\kappa_j}(x,y)\right|g_j(y)\mathrm{d}y,	 
		\end{align*} 
		where 
		$g_{j} =|u_{j}|\operatorname{e}^{\beta_{j}\tilde{\zeta}_{j} u_{j}^2}.$ 
  So, by H{\"o}lder's inequality for $1<p<2$, we get
  \begin{align*}
	 \int_{\mathbb{R}^{4}}\left|\nabla^{i}\Phi_{\kappa_j}(x,y)\right|g_j(y)\mathrm{d}y & = \int_{\mathbb{R}^{4}}\left|\nabla^{i}\Phi_{\kappa_j}(x,y)\right||g_j(y)|^{\frac{1}{p}}|g_j(y)|^{\frac{1}{p^{\prime}}}\mathrm{d}y\\
			& \le \left(\int_{\mathbb{R}^{4}}\left|\nabla^{i}\Phi_{\kappa_j}(x,y)\right|^{p}|g_j(y)|\mathrm{d}y\right)^{\frac{1}{p}}\left(\int_{\mathbb{R}^{4}}|g_j(y)|\mathrm{d}y\right)^{\frac{1}{p^{\prime}}}. 
		\end{align*} 
It follows that 
\begin{align*}
			|\nabla^{i}v_{j}|^{p} \le  \int_{\mathbb{R}^{4}}\left|\nabla^{i}\Phi_{\kappa_j}(x,y)\right|^{p}\frac{|g_j(y)|}{\int_{\mathbb{R}^{4}}|g_j(z)|\mathrm{d}z}\mathrm{d}y,
		\end{align*} 
for  $i = 0,1,2$. Applying Fubini’s theorem and using \eqref{AbsEstimate1},  \eqref{AbsEstimate2} and \eqref{AbsEstimate3}
		\begin{align*}
			\int_{\mathbb{R}^{4}}|\nabla^{i}v_j(x)|^{p}\mathrm{d}x & \leq \int_{\mathbb{R}^{4}} \left(\int_{\mathbb{R}^{4}}\left|\nabla^{i}\Phi_{\kappa_j}(x,y)\right|^{p}\frac{|g_j(y)|}{\int_{\mathbb{R}^{4}}|g_j(z)|\mathrm{d}z}\mathrm{d}y\right)\mathrm{d}x\\
			&\le c,\, \mbox{for}\;\; i=0,1,2.
		\end{align*}
		Therefore, $\|v_{j}\|_{W^{2, p}(\mathbb{R}^{4})}\le c$.
By noticing that  $v_{j} = \xi_{j}u_j$ satisfies \eqref{E-auxiliar}, we have 
\begin{equation}\label{SolEst}
			\|\xi_ju_j\|_{W^{2, p}(\mathbb{R}^{4})}\le c.
		\end{equation} 
	From Lemma~\ref{LemmaPhi}, we have $\xi_j/c_j\to 1$. Then the proof follows from \eqref{SolEst}.
	\end{proof}
	
The following result will be important to demonstrate the convergence of $c_{j}u_{j}$ to a Green function.
\begin{lemma}\label{TestFunctionPhi}
		Let $\phi \in C_{0}^{1}(\mathbb{R}^4)$, then we have
		\begin{equation}\label{LimitPhiFunction}			
			\lim\limits_{j\rightarrow \infty}\frac{c_j}{\lambda_j}\int_{\mathbb{R}^{4}}\phi u_j\operatorname{e}^{\beta_{j}\tilde{\zeta}_{j} u_{j}^2}\mathrm{d}x= \phi(0).
		\end{equation}
	\end{lemma}
	\begin{proof}
		Let $\phi \in C^{\infty}_{0}(\mathbb{R}^4)$ with $\supp( \phi) \subset B_{\rho}$, for some $\rho > 0$. We can separate the integral as follows
		\begin{align}
			\frac{c_j}{\lambda_j}\int_{\mathbb{R}^{4}}\phi u_j\operatorname{e}^{\beta_{j}\tilde{\zeta}_{j} u_{j}^2}\mathrm{d}x= I^{1}_{j}+I^{2}_{j} +I^{3}_{j}.
		\end{align}
		where
		\begin{align*}
			I^{1}_{j}&= \frac{c_j}{\lambda_j}\int_{B_{\rho_{j}^{A}}(x_j) \backslash B_{\hat{R}r_{j}}(x_j)}\phi u_j\operatorname{e}^{\beta_{j}\tilde{\zeta}_{j} u_{j}^2}\mathrm{d}x\\
			I^{2}_j &= \frac{c_j}{\lambda_j}\int_{B_{\hat{R}r_{j}}(x_j)}\phi u_j\operatorname{e}^{\beta_{j}\tilde{\zeta}_{j} u_{j}^2}\mathrm{d}x\\
			I^{3}_j &=\frac{c_j}{\lambda_j}\int_{B_{\rho }\backslash B_{\rho_{j}^{A}}(x_j)}\phi u_j\operatorname{e}^{\beta_{j}\tilde{\zeta}_{j} u_{j}^2}\mathrm{d}x
		\end{align*}
	We will show that $I^i_j\to 0$ for $i=1,3$ and $I^{2}_j\to \phi(0)$, as $j\to \infty$.
		From Lemma~\ref{lambdainf} and Lemma~\ref{ThmTruncation}-(1) 
		\begin{align*}
			|I^{1}_j|&\leq A\|\phi\|_{C^{0}}\frac{1}{\lambda_j}\int_{B_{\rho_{j}^{A}}(x_j) \backslash B_{\hat{R}r_{j}}}u^2_j\operatorname{e}^{\beta_{j}\tilde{\zeta}_{j} u_{j}^2}\mathrm{d}x \\
			& =A\|\phi\|_{C^{0}}\frac{1}{\lambda_j}\Big(\int_{\mathbb{R}^4}-\int_{B_{\hat{R}r_j}(x_j)}-\int_{\mathbb{R}^4\backslash B_{\rho^A_j}(x_j)}\Big)u^2_j\operatorname{e}^{\beta_{j}\tilde{\zeta}_{j} u_{j}^2}\mathrm{d}x\\
			& =A\|\phi\|_{C^{0}}\Big[\frac{1-\mu_j\|u_j\|^{2}_2}{\tilde{\zeta}_{j}}-\frac{1}{\lambda_j}\Big(\int_{B_{\hat{R}r_j}(x_j)}+\int_{\mathbb{R}^4\backslash B_{\rho^A_j}(x_j)}\Big)u^2_j\operatorname{e}^{\beta_{j}\tilde{\zeta}_{j} u_{j}^2}\mathrm{d}x\Big].
		\end{align*}
		Hence, 
		\begin{align}\label{I1-quase}
			|I^{1}_j|&\leq A\|\phi\|_{C^{0}}\Big[1-\frac{1}{\lambda_j}\int_{B_{\hat{R}r_j}(x_j)}u^2_j\operatorname{e}^{\beta_{j}\tilde{\zeta}_{j} u_{j}^2}\mathrm{d}x+o_j(1)\Big]
		\end{align}
As in \eqref{Eqrightside}, we have 
\begin{equation}\label{indo-1}
\begin{aligned}
	 \int_{B_{\hat{R}r_{j}}(x_j)}u_j^2\operatorname{e}^{\beta_{j}\tilde{\zeta}_{j}u_{j}^{2}} \mathrm{d}x    & = \lambda_j\int_{B_{\hat{R} }(0)}w^{2}_j\operatorname{e}^{\beta_{j}\tilde{\zeta}_{j}z_j(w_j+1)}\mathrm{d}x.		
		\end{aligned}
		\end{equation}
	Taking into account \eqref{indo-1}, letting 	$j \rightarrow \infty$ and $\hat{R} \rightarrow \infty$ in \eqref{I1-quase}, by Lemma~\ref{lemmaseqwk} and Lemma \ref{doz}, we conclude that $I^{1}_j\to 0$. In addition, 
		\begin{align*}
			I^{2}_j= \int_{B_{\hat{R}}(0)}\phi(x_{j}+r_{j}x)w_j\operatorname{e}^{\beta_{j}\tilde{\zeta}_{j}z_j(w_j+1)}\mathrm{d}x
		\end{align*}
Thus, by Lemmas~\ref{lemmaseqwk} and~\ref{doz}, letting $j \to \infty $ and $ \hat{R} \to \infty $, we obtain $ I_j^{2} \to \phi(0)$. Finally, note that $\exp(\beta_{j}\tilde{\zeta}_j|u_{j}^{A}|^2)$ is  bounded in $L^{q}(B_{\rho})$ for some $q>1$, and so  by choosing $q$ close to $1$ and applying  H{\"o}lder's inequality
		\begin{align*}
			|I^{3}_j| \leq \frac{c_j}{\lambda_j}\|\phi\|_{C^{0}}\left(\int_{B_{\rho}}|u_{j}|^{q^{\prime}}\mathrm{d}x\right)^{\frac{1}{q^{\prime}}}\left(\int_{B_{\rho}}\operatorname{e}^{\beta_{j}\tilde{\zeta}_jq|u^{A}_{j}|^2}\mathrm{d}x\right)^{\frac{1}{q}}.
		\end{align*}
From Lemma~\ref{LemmaDoubleLimit}, we have $c_j/\lambda_j\to 0$.	Then,  $I^{3}_j\to 0$ as $j\to\infty$. 
	\end{proof}
	
Before presenting the next result, we recall the following auxiliary proposition.
	\begin{proposition}\label{PropFracIntByParts}
		Let $\Omega \subset \mathbb{R}^4$ be a bounded open domain with Lipschitz boundary. Then for any $u \in H^{2}(\Omega)$, $\omega \in H^{4}(\Omega)$, we have
		$$
			\int_{\Omega}\Delta u  \Delta \omega \mathrm{d}x = \int_{\Omega} u \Delta^2 \omega \mathrm{d}x -  \int_{\partial \Omega} \nu \cdot u \nabla\Delta \omega \mathrm{d}S + \int_{\partial\Omega}\nu \cdot \nabla u \Delta \omega \mathrm{d}S
		$$
		where $\nu$ denotes the outer normal to $\partial\Omega$.
	\end{proposition}
	\begin{lemma}\label{GreenFunctionLemma}
		For any $1<p<2$, we have $c_{j}u_{j}  \rightharpoonup G \in C^{3}(\mathbb{R}^4\backslash\{0\})$ weakly in $W^{2,p}(\mathbb{R}^4)$,  where $G$ is a Green function satisfying   $\Delta^2 G+\kappa_0G = \delta_0$ in $\mathbb{R}^4$, where $\kappa_0=1-\alpha(\gamma+1)$.
		Moreover 
		$$
		G =	-\frac{1}{8\pi^2} \ln|x|+ K_{0} + h
		$$
		where $K_{0}$ is a constant depending on $0$, $h \in C^{3}(\mathbb{R}^4)$ with $h(0) = 0$. Further, 
		$$
		\lim\limits_{j \rightarrow \infty}\int_{\mathbb{R}^4 \backslash B_{\epsilon}}\big(|\Delta(c_{j}u_{j})|^2 +|c_{j}u_{j}|^2\big)\mathrm{d}x = -\frac{1}{8\pi^2} \ln\epsilon -\frac{1}{16\pi^2}+K_{0} +\alpha(\gamma+1)\|G\|^{2}_{2} + O(\epsilon), 
		$$
	as $\epsilon\to 0$,  where $B_{\epsilon}=B_{\epsilon}(0)$,	for  $\epsilon>0$.
	\end{lemma}
	\begin{proof}
By Lemma \ref{BoundnessLemma}, we obtain some $G \in W^{2,p}(\mathbb{R}^4)$ such that $c_{j}u_{j} \rightharpoonup G$ weakly in $W^{2,p}(\mathbb{R}^4)$, for any $1<p<2$. Since $\|\Delta u_{j}\|^2\mathrm{d}x \stackrel{\ast}{\rightharpoonup} \delta_0$ in the sense of measure, we have that  $\exp(\beta_{j}\tilde{\zeta_{j}}u_{j}^2)$ is bounded in $L^{p}(B_{R}\backslash B_{S})$, for any radius $0 < S < R$. Notice that $c_{j}u_{j}$ satisfies the Euler-Lagrange equation
		\begin{equation}\label{EulerLagrangeCkUk}
			\Delta^2(c_j u_j) + c_j u_j= \frac{c_j u_j}{\lambda_j}\tilde{\zeta}_j\operatorname{e}^{\beta_{j} \tilde{\zeta_{j}}u_{j}^{2}}+\mu_jc_ju_j \;\text{ in }\; \mathbb{R}^4.
		\end{equation}
		Therefore, by the standard regularity theory, we got $c_j u_j \rightarrow G$ in $C^{3}_{loc}(\mathbb{R}^4\backslash\{0\})$. By Lemma \ref{TestFunctionPhi}, for any $\phi\in C^{3}_{0}(\mathbb{R}^4)$, we can write
		$$
	 \lim\limits_{j\rightarrow \infty}\int_{\mathbb{R}^{4}}\phi\left(\frac{c_j u_j}{\lambda_j}\tilde{\zeta}_j\operatorname{e}^{\beta_{j} \tilde{\zeta_{j}}u_{j}^{2}}+(\mu_j -1)c_ju_j\right)\mathrm{d}x=\phi(0) - \kappa_0\int_{\mathbb{R}^{4}} \phi G\mathrm{d}x .
		$$
		This implies that
		$$
		\Delta^2 G+\kappa_0 G  = \delta_0~~ \text{in }\mathbb{R}^4.
		$$
		Fix $\rho>0$. Let $\psi \in C_{0}^{\infty}(B_{2\rho}(0))$ be a cutoff function  such that $\psi = 1$ in  $B_{\rho}(0)$ and
		$$
		g = G + \frac{1}{8\pi^2}\psi\ln|x|.
		$$
By computing directly the biharmonic of $g$ we have
		$$
		\Delta^2 g  = f ~~ \text{in }\mathbb{R}^4, 
		$$
		where the function $f$ is given by
		\begin{align*}
		f &= -\frac{1}{8\pi^2}\left(\Delta^2 \psi \ln|x| + 2\nabla \Delta \psi \cdot \nabla\ln|x| + 2 \Delta(\nabla \psi \cdot \nabla \ln|x|)+2\nabla \psi \cdot\nabla \Delta \ln|x|\right)\\
		&-\frac{2}{8\pi^2}\Delta\psi\Delta \ln |x| -\kappa_0 G
		\end{align*}
		where we have used that  $\psi\Delta^2\ln|x|=8\pi^2\delta_0$ in $\mathbb{R}^4$.
		Since $G \in W^{2,p}(\mathbb{R}^4)$ for any $1<p<2$, we have $f \in L^{q}_{loc}(\mathbb{R}^4)$ for any $q > 2$. Again, by the standard regularity theory, we get $g \in C^{3}_{loc}(\mathbb{R}^4)$. Let $K_{0} = g(0)$ and 
		$$
		h = g - g(0) + \frac{1}{8\pi^2}(1-\psi)\ln|x|.
		$$
		Then, we have
		\begin{align}
			G = -\frac{1}{8\pi^2}\ln|x| + K_{0} + h.
		\end{align}
	 Finally, set $ 	U_j = c_j u_j.$  From \eqref{EulerLagrangeCkUk}, we get 
	 \begin{equation}\label{eq.Uj}
			\Delta^2U_j +(1-\mu_j) U_j= \frac{\tilde{\zeta}_j}{\lambda_j}U_j\operatorname{e}^{\beta_{j} \tilde{\zeta_{j}}u_{j}^{2}} \;\text{ in }\; \mathbb{R}^4.
		\end{equation}
Let $0<\epsilon<R$. For any $\varphi\in C^{\infty}_{c}(B_{R})$, by applying Proposition~\ref{PropFracIntByParts} on $B_R\setminus B_{\epsilon}$, the equation \eqref{eq.Uj} yields
\begin{align*}
\int_{B_R\setminus B_{\epsilon}}\Delta U_j\Delta \varphi \mathrm{d}x &= \int_{B_R\setminus B_{\epsilon}} \varphi \Delta^2 U_j \mathrm{d}x -  \int_{\partial B_{\epsilon}} \eta\cdot \varphi \nabla\Delta U_j \mathrm{d}S + \int_{\partial B_{\epsilon}}\eta \cdot \nabla \varphi \Delta U_j \mathrm{d}S\\
&= \int_{B_R\setminus B_{\epsilon}} (\mu_j-1) U_j\varphi\mathrm{d}x+ \frac{\tilde{\zeta}_j}{\lambda_j}\int_{B_R\setminus B_{\epsilon}}U_j\operatorname{e}^{\beta_{j} \tilde{\zeta_{j}}u_{j}^{2}}\mathrm{d}x\\
&+\int_{\partial B_{\epsilon}} \nu \Big(\varphi \nabla\Delta U_j- \nabla \varphi \Delta U_j\Big) \mathrm{d}S
\end{align*}
where $\nu=-\eta$  is the outer normal vector of $\partial B_{\epsilon}$. By density, we can choose $\varphi=U_j$ to obtain
\begin{equation}\label{Eq.TR}
\begin{aligned}
\int_{B_R\setminus B_{\epsilon}}\big(|\Delta U_j|^2 +(1-\mu_j)|U_j|^2\big) \mathrm{d}x &= \frac{\tilde{\zeta}_jc^{2}_j}{\lambda_j}\int_{B_R\setminus B_{\epsilon}}u^2_j\operatorname{e}^{\beta_{j} \tilde{\zeta_{j}}u_{j}^{2}}\mathrm{d}x\\
&+\int_{\partial B_{\epsilon}} \nu \Big(U_j \nabla\Delta U_j- \nabla U_j \Delta U_j\Big) \mathrm{d}S.
\end{aligned}
\end{equation}
Letting $R\to \infty$, we have 
\begin{equation}\label{Eq.T}
\begin{aligned}
\int_{\mathbb{R}^{4}\setminus B_{\epsilon}}\big(|\Delta U_j|^2 +(1-\mu_j)|U_j|^2\big) \mathrm{d}x &= \frac{\tilde{\zeta}_jc^{2}_j}{\lambda_j}\int_{\mathbb{R}^{4}\setminus B_{\epsilon}}u^2_j\operatorname{e}^{\beta_{j} \tilde{\zeta_{j}}u_{j}^{2}}\mathrm{d}x\\
&+\int_{\partial B_{\epsilon}} \nu \Big(U_j \nabla\Delta U_j- \nabla U_j \Delta U_j\Big) \mathrm{d}S.
\end{aligned}
\end{equation}
 By taking into account the  Lemmas  \ref{radiallemma}, \ref{maxseqconc} and \ref{LemmaDoubleLimit}, we can write 
		\begin{align}\label{EqLimInt}
			\lim_{j \rightarrow \infty} \int_{\mathbb{R}^4\backslash B_{\epsilon}}\big(|\Delta U_j|^2 +(1-\mu_j)|U_j|^2  \big)\mathrm{d}x & = \int_{\partial B_{\epsilon}} \nu \big(G\nabla \Delta G- \nabla G \Delta G\big) \mathrm{d}S.
		\end{align}
	From \eqref{EqLimInt}, by employing  Fatou's Lemma we also can write
		\begin{align*}
		\int_{\mathbb{R}^4\backslash B_{\epsilon}}\big(|\Delta G|^2+\kappa_0 |G|^2  \big)\mathrm{d}x &\le \int_{\partial B_{\epsilon}} \nu (G\nabla \Delta G- \nabla G \Delta G) \mathrm{d}S.	
		\end{align*}
Since $\kappa_0>0$, we conclude that $G\in W^{2,2}(\mathbb{R}^4\setminus B_{\epsilon}(0))$ for any $\epsilon>0$. With this information in hand, returning to \eqref{Eq.TR} and letting  $ j \to \infty $, and then $R \to \infty$ we obtain 
\begin{align*}
		\int_{\mathbb{R}^4\backslash B_{\epsilon}}\big(|\Delta G|^2+\kappa_0 |G|^2  \big)\mathrm{d}x &=\int_{\partial B_{\epsilon}} \nu (G\nabla \Delta G- \nabla G \Delta G) \mathrm{d}S.	
		\end{align*}
In particular, 
\begin{align*}
		\lim_{\epsilon\to\infty}\int_{\partial B_{\epsilon}} \nu (G\nabla \Delta G- \nabla G \Delta G) \mathrm{d}S=0.	
		\end{align*}
So, from \eqref{EqLimInt}, 
\begin{align}\nonumber
			\lim_{\epsilon\to \infty}\lim_{j \rightarrow \infty}\int_{\mathbb{R}^4\backslash B_{\epsilon}}|U_j|^2\mathrm{d}x =0.
		\end{align}
		From this, since we already know that $U_j \rightarrow G$ in  $C^{3}_{\text{loc}}(\mathbb{R}^4 \setminus \{0\}),$
we obtain $ U_j \rightarrow G $ in  $L^2(\mathbb{R}^4).$ Now, from  \eqref{EqLimInt} we can write 
		\begin{align}\label{EqLimIntF}
			\lim_{j \rightarrow \infty} \int_{\mathbb{R}^4\backslash B_{\epsilon}}\big(|\Delta U_j|^2 +|U_j|^2  \big)\mathrm{d}x & = \int_{\partial B_{\epsilon}} \nu \big(G\nabla \Delta G- \nabla G \Delta G\big) \mathrm{d}S+\alpha(\gamma+1)\|G\|^{2}_{2}.
		\end{align}
It is well known that the fundamental solution of $\Delta^2 $ in $\mathbb{R}^4$ is $-\frac{1}{8\pi^2}\ln|x|$, and it satisfies
		\begin{align}\label{LaplacianIdentities}
			\nabla \ln|x| = \frac{x}{|x|^2}, ~~ \Delta(\ln|x|) = \frac{2}{|x|^2}, ~~ \nabla \Delta (\ln|x|) = -\frac{4x}{|x|^4}.
		\end{align}
So, we have
		\begin{equation}\label{ComputationsLaplacian}
		\begin{aligned}
			\nu G(\epsilon)(\nabla\Delta G)(\epsilon) & = \left(-\frac{1}{8\pi^2} \ln \epsilon+K_{0}+O(\epsilon)\right)\left(\frac{1}{2\pi^2}\frac{1}{\epsilon^3}+O(1)\right) \\
			&= \frac{1}{2\pi^2}\frac{1}{\epsilon^3}\left(-\frac{1}{8\pi^2} \ln\epsilon+K_{0}+O(\epsilon)\right)
		\end{aligned}
		\end{equation}
		and
		\begin{align}\label{ComputationsLaplacian2}
			-\nu (\nabla G)(\epsilon)(\Delta G)(\epsilon) = -\left(-\frac{1}{8\pi^2}\frac{1}{\epsilon}+O(1)\right)\left(-\frac{1}{8\pi^2}\frac{2}{\epsilon^2}+O(1)\right) = -\frac{1}{32\pi^4}\frac{1}{\epsilon^3}\left(1+O(\epsilon)\right).
		\end{align}
		By combining  \eqref{EqLimIntF},  \eqref{ComputationsLaplacian} and \eqref{ComputationsLaplacian2}, we obtain the desired  result.
	\end{proof}
	\subsection{The upper bound for the Adimurthi-Druet-Adams functional acting on concentrating sequences}
	To obtain an upper bound for the Adimurthi-Druet-Adams functional acting on concentrating sequences, we follow an approach similar to that in \cite{RufLi2008}. We start by establishing an upper bound for any blow up function sequences in $H_{0}^{2}(B_{R})$.
	\begin{lemma}\label{lemmaEstimatesFunctionalimsup}
		 Let $(u_{j})$ be a bounded sequence in $H_{0}^{2}(B_{R})$ such that $\|\Delta u_{j}\|_{2}^{2} = 1 $, where $B_{R} \subset \mathbb{R}^4$. If $u_{j} \rightharpoonup 0$ in $H_{0}^{2}(B_{R})$, then 
		$$
		\limsup\limits_{j \rightarrow \infty} \int_{B_{R}}\left(\operatorname{e}^{\beta_{j}\tilde{\zeta}_{j}u_{j}^{2}}-1\right)\mathrm{d}x \leq \frac{|B_{R}| \operatorname{e}^{-\frac{1}{3}}}{3}
		$$
	\end{lemma}
	\begin{proof}
		By \cite[Ineq. (5.23)]{LuYangEP} we have the estimate
		$$
		\limsup\limits_{j \rightarrow \infty} \int_{B_{R}}\left(\operatorname{e}^{\beta_{j}\tilde{\zeta}_{j}u_{j}^{2}}-1\right)\mathrm{d}x \leq \frac{\pi^{2}}{6}\operatorname{e}^{\frac{5}{3}+32\pi^{2}A_{0}}
		$$
		where $A_{0}$ is the value at $0$ of the trace of the part regular of the Green Function $G_{B_R}$ for the bi-Laplacian operator $\Delta^{2}$ on the ball $B_R$. It is well known that  $$G_{B_R} = \frac{1}{8\pi^{2}}\ln |x|+\frac{1}{16\pi^{2}}\frac{|x|^{2}}{R^{2}}+ \frac{1}{8\pi^{2}}\ln R-\frac{1}{16\pi^2}$$	and the term $\frac{1}{8\pi^{2}}\ln R-\frac{1}{16\pi^2}$ is the value at $0$ of the trace of the regular part of $G_{B_{R}}$. So, by simple computation, we obtain the desired estimate. 
	\end{proof}
	
	Having established this upper bound, we proceed to derive a corresponding upper bound for the Adimurthi-Druet-Adams inequality in the entire space 
 $\mathbb{R}^{4}$.
	\begin{lemma}\label{cAD<e}
		If $AD(4,2, 32\pi^2,  \alpha, \gamma)$ is not attained, then 
		$$AD(4,2, 32\pi^2,  \alpha, \gamma) \leq \ \frac{\pi^{2}}{6}\operatorname{e}^{\frac{5}{3}+32\pi^{2}K_{0}},$$ where $K_{0}$ is the value at $0$ of the regular part of the Green function $G$ for the operator $\Delta^{2}+\kappa_0$.
	\end{lemma}
	\begin{proof}
By assumption, we can assume $c_{j}\to \infty$ as $j\to\infty$.
Let
		$$
		\tilde{u}_{j}= \frac{u_{j}-u_{j}^{\epsilon}}{\|\Delta(u_{j}-u_{j}^{\epsilon})\|_{L^{2}(B_{\epsilon}(x_j))}},
		$$
		where $$u_{j}^{\epsilon}(r)=u_{j}(\epsilon)+\frac{u'_{j}(\epsilon)r^2}{2\epsilon}-\frac{u'_{j}(\epsilon)\epsilon}{2},\;\;\mbox{with}\;\; r=|x|.$$ 	
		Since $c_ju_j(\epsilon)\to G(\epsilon)$ and $c_ju^{\prime}_j(\epsilon)\to G^{\prime}(\epsilon)$, as $j\to\infty$,  we have 
		\begin{equation}\label{ueulimits}
		\left\{\begin{aligned}
		& \lim_{j\to\infty}\Big[c_ju_{j}(\epsilon)-\frac{c_ju'_{j}(\epsilon)\epsilon}{2}\Big]=G(\epsilon)-\frac{\epsilon G^{\prime}(\epsilon)}{2}=-\frac{1}{8\pi^2}\ln\epsilon + K_{0} +\frac{1}{16\pi^2} +o_{\epsilon}(1)\\
		&
		\lim_{j\to\infty}\frac{c_ju^{\prime}_j(\epsilon)}{2\epsilon}=\frac{G^{\prime}(\epsilon)}{2\epsilon}=-\frac{1}{16\pi^2\epsilon^2}+\frac{h^{\prime}(\epsilon)}{2\epsilon}.
		\end{aligned}\right.
		\end{equation}
	Let us compute $\|\Delta u_{j}-\Delta u_{j}^{\epsilon})\|_{L^{2}(B_{\epsilon})}^{2}$. We have 
	\begin{align}\nonumber
		\|\Delta(u_{j}-u_{j}^{\epsilon})\|_{L^{2}(B_{\epsilon})}^{2} &= \int_{B_{\epsilon}}|\Delta u_{j}|^{2}\mathrm{d}x -2\int_{B_{\epsilon}}\Delta u_{j}\Delta u_{j}^{\epsilon}\mathrm{d}x+\int_{B_{\epsilon}}|\Delta u_{j}^{\epsilon}|^{2}\mathrm{d}x = I_1-2I_2+I_3.
	\end{align}
	Firstly, 
	\begin{align*}
		I_{1} = \int_{B_{\epsilon}}|\Delta u_{j}|^{2}\mathrm{d}x & = 1 - \int_{\mathbb{R}^{4}\backslash B_{\epsilon}}\big(|\Delta u_{j}|^{2}+|u_{j}|^2\big)\mathrm{d}x - \int_{B_{\epsilon}}|u_{j}|^2\mathrm{d}x\\
		& = 1-\frac{1}{c^{2}_j}\Big(\int_{\mathbb{R}^{4}\backslash B_{\epsilon}}\big(|\Delta U_{j}|^{2}+|U_{j}|^2\big)\mathrm{d}x-\int_{B_{\epsilon}}|U_{j}|^2\mathrm{d}x\Big)
	\end{align*}	
	Then, from Lemma~\ref{GreenFunctionLemma} we get
	\begin{align}\label{int1}
		I_1 =1- \frac{1}{c_{j}^2}\Big[ -\frac{1}{8\pi^2} \ln\epsilon -\frac{1}{16\pi^2}+K_{0} +\alpha(\gamma+1)\|G\|^{2}_{2} + o_j(\epsilon)+o_{\epsilon}(1)\Big],
	\end{align}
	where $o_j(\epsilon)$ means that $\lim _{j\to\infty}o_j(\epsilon)=0$, if $\epsilon$ is fixed. 
	By Proposition~\ref{PropFracIntByParts}, we can decompose $I_{2}$ as follows
	\begin{align*}
	I_{2}&= \int_{B_{\epsilon}}u_{j}^{\epsilon}\Delta^{2}u_{j}\mathrm{d}x-\int_{\partial B_{\epsilon}}\nu\cdot u_{j}^{\epsilon}\nabla\Delta u_{j}\mathrm{d}S + \int_{\partial B_{\epsilon}}\nu\cdot \nabla u_{j}^{\epsilon}\Delta u_{j}\mathrm{d}S = I^{1}_2-I^{2}_2+I^{3}_{2}.
	\end{align*}
	By \eqref{EulerLagrangeCkUk},  we get
	\begin{equation}\label{esplitado}
	\begin{aligned}
		I^{1}_2&= \int_{B_{\epsilon}}u_{j}^{\epsilon}\Delta^{2}u_{j} \mathrm{d}x = \frac{1}{c_j}\int_{B_{\epsilon}}u_{j}^{\epsilon}\left(\frac{c_ju_{j}}{\lambda_{j}}\tilde{\zeta}_{j}\operatorname{e}^{\beta_{j}\tilde{\zeta}_{j} u_{j}^2} +(\mu_{j}-1)c_{j}u_j \right) \mathrm{d}x \\
		&=  \frac{1}{c^2_j}\Big[\Big(c_ju_{j}(\epsilon)-\frac{c_ju'_{j}(\epsilon)\epsilon}{2}\Big)\frac{c_j\tilde{\zeta}_{j}}{\lambda_{j}}\int_{B_{\epsilon}}u_j\operatorname{e}^{\beta_{j}\tilde{\zeta}_{j} u_{j}^2}\mathrm{d}x+\frac{c_ju^{\prime}_j(\epsilon)}{2\epsilon}\frac{c_j\tilde{\zeta}_{j}}{\lambda_{j}}\int_{B_{\epsilon}}|x|^2u_j\operatorname{e}^{\beta_{j}\tilde{\zeta}_{j} u_{j}^2}\mathrm{d}x\Big]\\
		&+\frac{1}{c^{2}_j}\Big[\Big(c_ju_{j}(\epsilon)-\frac{c_ju'_{j}(\epsilon)\epsilon}{2}\Big)(\mu_j-1)\int_{B_{\epsilon}} c_ju_j  \mathrm{d}x+\frac{c_ju^{\prime}_j(\epsilon)}{2\epsilon}(\mu_j-1)\int_{B_{\epsilon}}|x|^2c_ju_j  \mathrm{d}x\Big].
	\end{aligned}
	\end{equation}
By H\"{o}lder  inequality and using that $c_ju_j\to G$ in $L^{2}(\mathbb{R}^4)$, we have 	
\begin{equation}\label{uealimits2}
\left\{\begin{aligned}
& \Big|\int_{B_{\epsilon}} c_ju_j  \mathrm{d}x\Big|\le |B_{\epsilon}|^{\frac{1}{2}}\|c_ju_j\|_{L^2(\mathbb{R}^4)} =O(\epsilon^2)\\
& \Big|\int_{B_{\epsilon}}|x|^2 c_ju_j  \mathrm{d}x\Big|\le \epsilon^2|B_{\epsilon}|^{\frac{1}{2}}\|c_ju_j\|_{L^2(\mathbb{R}^4)}=O(\epsilon^{4})\\
& \lim_{j\to\infty}\Big[c_ju_{j}(\epsilon)-\frac{c_ju'_{j}(\epsilon)\epsilon}{2}\Big]= G(\epsilon)-\frac{\epsilon G^{\prime}(\epsilon)}{2}=O(\ln \epsilon)\\
&\lim_{j\to\infty}\frac{c_ju^{\prime}_j(\epsilon)}{2\epsilon}=\frac{G^{\prime}(\epsilon)}{2\epsilon}=O\Big(\frac{1}{\epsilon^2}\Big).
\end{aligned}\right.
\end{equation}
Taking account the	Lemma~\ref{TestFunctionPhi},  and using \eqref{ueulimits}, \eqref{esplitado} and \eqref{uealimits2}, we also can write 
\begin{align*}
		I^{1}_2 &= \frac{1}{c_{j}^{2}}\Big[-\frac{1}{8\pi^{2}}\ln \epsilon+K_0+\frac{1}{16\pi^2}+o_{j}(\epsilon)+o_{\epsilon}(1)\Big].
	\end{align*}
	Now, note that $u^{\epsilon}_j(x)=u_j(\epsilon)$ on $\partial B_{\epsilon}$, then 
	\begin{align*}
		I^{2}_{2} = \int_{\partial B_{\epsilon}}\nu\cdot u_{j}^{\epsilon}\nabla\Delta u_{j}\mathrm{d}S & = \frac{u_j(\epsilon)}{c_{j}}\int_{\partial B_{\epsilon}}\nu\cdot\nabla\Delta (c_ju_j) \mathrm{d}S\\
		&= \frac{u_j(\epsilon)}{c_{j}}\left[\int_{\partial B_{\epsilon}}\nu\cdot\nabla\Delta G\mathrm{d}S + o_{j}(\epsilon)\right]\\
		&= \frac{1}{c^2_{j}}\left[c_ju_j(\epsilon)\int_{\partial B_{\epsilon}}\nu\cdot \nabla\Delta G\mathrm{d}S + o_{j}(\epsilon)\right]
	\end{align*}
	and by  the identities in \eqref{LaplacianIdentities}
	\begin{align*}
		\int_{\partial B_{\epsilon}}\nu\cdot\nabla\Delta G\mathrm{d}S&=\frac{1}{2\pi^2}\int_{\partial B_{\epsilon}}\frac{1}{|x|^3}\mathrm{d}S+\int_{\partial B_{\epsilon}}\nu\cdot\nabla\Delta h\mathrm{d}S= 1+O(\epsilon^3).
	\end{align*}
	It follows that 
	\begin{align*}
		I^{2}_{2} = \frac{1}{c^2_{j}}\left[ G(\epsilon)+O(\epsilon^3G(\epsilon)) + o_{j}(\epsilon)\right]= \frac{1}{c^2_{j}}\left[ G(\epsilon) + o_{j}(\epsilon)+o_{\epsilon}(1)\right].
	\end{align*}
Analogous, from \eqref{LaplacianIdentities}  we can write 	
\begin{align*}
		\int_{\partial B_{\epsilon}}\Delta G\mathrm{d}S&=-\frac{1}{4\pi^2}\int_{\partial B_{\epsilon}}\frac{1}{|x|^2}\mathrm{d}S+\int_{\partial B_{\epsilon}}\Delta h\mathrm{d}S= -\frac{\epsilon}{2}+O(\epsilon^3)
	\end{align*}
and then we get
	\begin{align*}
		I^{3}_2 &=\int_{\partial B_{\epsilon}}\nu\cdot \nabla u_{j}^{\epsilon}\Delta u_{j}\mathrm{d}S= \frac{u^{\prime}_j(\epsilon)}{c_{j}}\left[\int_{\partial B_{\epsilon}}\Delta G\mathrm{d}S + o_{j}(\epsilon)\right] \\
		&= \frac{1}{c^{2}_{j}}\left[c_ju^{\prime}_j(\epsilon)\int_{\partial B_{\epsilon}}\Delta G\mathrm{d}S + o_{j}(\epsilon)\right] \\
			&= \frac{1}{c^{2}_{j}}\left[G^{\prime}(\epsilon)\int_{\partial B_{\epsilon}}\Delta G\mathrm{d}S + o_{j}(\epsilon)\right] \\
			&= \frac{1}{c^{2}_{j}}\left[-\frac{1}{8\pi^2\epsilon}\big(1+o_{\epsilon}(1)\big)\int_{\partial B_{\epsilon}}\Delta G\mathrm{d}S + o_{j}(\epsilon)\right] \\
		&= \frac{1}{c_{j}^{2}}\left[\frac{1}{16\pi^2}+o_{j}(\epsilon)+o_{\epsilon}(1)\right].
	\end{align*}
	Hence,
	\begin{equation}\label{IntLaplacianProd}
	\begin{aligned}
		I_2 &= I^{1}_2-I^{2}_2+I^{3}_{2}\\
		& =\frac{1}{c^2_j}\Big[-\frac{1}{8\pi^{2}}\ln \epsilon+K_0+\frac{1}{16\pi^2}+o_{j}(\epsilon)+o_{\epsilon}(1)\Big]\\
		&+ \frac{1}{c_{j}^{2}}\left[\frac{1}{16\pi^2}- G(\epsilon)+o_{j}(\epsilon)+o_{\epsilon}(1)\right]\\
		& = \frac{1}{c_{j}^{2}}\left[\frac{1}{8\pi^2}+o_{j}(\epsilon)+o_{\epsilon}(1)\right].
	\end{aligned}
	\end{equation}
Finally, by definition of $u_{j}^{\epsilon}$, 
\begin{equation}\label{I3}
\begin{aligned}
		I_3 &= \int_{B_{\epsilon}}|\Delta u_{j}^{\epsilon}|^{2} \mathrm{d}x = \int_{B_{\epsilon}}\left(\frac{4u_{j}'(\epsilon)}{\epsilon}\right)^{2} \mathrm{d}x \\
		& = 8\pi^2(u_{j}'(\epsilon))^2 \epsilon^2=\frac{1}{c^2_j}\Big[8\pi^2(\epsilon G^{\prime}(\epsilon))^2+o_j(\epsilon)\Big]\\
		&= \frac{1}{c^2_j}\Big[\frac{1}{8\pi ^2}+o_j(\epsilon)+o_{\epsilon}(1)\Big].
	\end{aligned}
\end{equation}
	By \eqref{int1}, \eqref{IntLaplacianProd} and \eqref{I3}, we have
	\begin{equation}\label{normDiferenceLaplacian}
	\begin{aligned}
		\|\Delta(u_{j}-u_{j}^{\epsilon})\|_{L^{2}(B_{\epsilon})}^{2}=1 - \frac{1}{c_{j}^2}\Big[-\frac{1}{8\pi^2}\ln\epsilon+\frac{1}{16\pi^{2}}+K_0+\alpha(\gamma+1)\|G\|^{2}_{2} +o_{j}(\epsilon)+o_{\epsilon}(1)\Big].
	\end{aligned}
	\end{equation}
	Thus, 
	\begin{align*} 
		\tilde{u}_j^{2}&=\frac{(u_{j}-u_{j}^{\epsilon})^{2}}{ 1- \frac{1}{c_{j}^2}\Big[-\frac{1}{8\pi^2}\ln\epsilon+\frac{1}{16\pi^{2}}+K_0+\alpha(\gamma+1)\|G\|^{2}_{2} +o_{j}(\epsilon)+o_{\epsilon}(1)\Big] } \\
		&=u_j^2\Big(1+ \frac{1}{c_{j}^2}\Big[-\frac{1}{8\pi^2}\ln\epsilon+\frac{1}{16\pi^{2}}+K_0+\alpha(\gamma+1)\|G\|^{2}_{2} +o_{j}(\epsilon)+o_{\epsilon}(1)\Big] \Big)\\
		&-\big(2u_ju^{\epsilon}_j+(u^{\epsilon}_j)^2\big)\big(1+o_j(\epsilon)\big)\\
		&=u_j^{2} -c \ln \epsilon^4+o_j(\epsilon).
	\end{align*}
	By Lemma \ref{LemmaDoubleLimit},
	$$
		\lim\limits_{\hat{R} \rightarrow \infty}\lim\limits_{j \rightarrow \infty}\int_{B_{\rho} \backslash B_{\hat{R}r_{j}}(x_{j})}\operatorname{e}^{\beta_{j} \tilde{\zeta_j}u_{j}^{2}}\mathrm{d}x = |B_{\rho}|, ~~ \text{ for any } \rho<\epsilon.
	$$
	Therefore, 
	\begin{align*}
		\lim\limits_{\hat{R} \rightarrow \infty}\lim\limits_{j\rightarrow \infty}\int_{B_{\rho} \backslash B_{\hat{R}r_{j}}(x_{j})}\operatorname{e}^{\beta_{j} \tilde{\zeta_j}\tilde{u}_{j}^{2}}\mathrm{d}x & \leq O\left(\frac{1}{\epsilon^{4c}}\right)\lim\limits_{\hat{R} \rightarrow \infty}\lim\limits_{j \rightarrow \infty}\int_{B_{\rho} \backslash B_{\hat{R}r_{j}}(x_{j})}\operatorname{e}^{\beta_{j} \tilde{\zeta_j}u_{j}^{2}}\mathrm{d}x\\
		 & = O\left(\frac{|B_{\rho}|}{\epsilon^{4c}}\right)\rightarrow 0 \text{ as } \rho \rightarrow 0.
	\end{align*}
		So, by Lemma \ref{maxseqconc}, we have
		$$
			\lim_{j\rightarrow \infty}\int_{B_{\epsilon} \backslash B_{\rho}}\left(\operatorname{e}^{\beta_{j}\tilde{\zeta_{j}}\tilde{u}_j^{2}(x)}-1\right) \mathrm{d}x = 0.
		$$
		Thus, with Lemma \ref{lemmaEstimatesFunctionalimsup}
		$$
		\lim\limits_{\hat{R}\rightarrow \infty}\lim_{j\rightarrow \infty}\int_{B_{\hat{R}r_{j}}}\left(\operatorname{e}^{\beta_{j}\tilde{\zeta_{j}}\tilde{u}_j^{2}(x)}-1\right) \mathrm{d}x	= 	\lim_{j\rightarrow \infty}\int_{B_{\epsilon}}\left(\operatorname{e}^{\beta_{j}\tilde{\zeta_{j}}\tilde{u}_j^{2}(x)}\right)\mathrm{d}x \leq \frac{1}{3}|B_{\epsilon}|\operatorname{e}^{-\frac{1}{3}}.
		$$
		Fix any $\hat{R} >0$ and any $x \in B_{\hat{R}r_{j}}(x_j)$,
		\begin{align*}
			\beta_{j}u_{j}^{2}(x) &= \beta_{j}\left(\frac{u_j(x)}{\|\Delta(u_{j}-u_{j}^{\epsilon})\|_{L^{2}(B_{\epsilon})}}\right)^2\int_{B_{\epsilon}}|\Delta(u_{j}-u_{j}^{\epsilon})|^2\mathrm{d}x \\
			&= \beta_{j}\left(\tilde{u}_{j}+\frac{u_j^{\epsilon}(x)}{\|\Delta(u_{j}-u_{j}^{\epsilon})\|_{L^{2}(B_{\epsilon})}}\right)^2\int_{B_{\epsilon}}|\Delta(u_{j}-u_{j}^{\epsilon})|^2\mathrm{d}x.
		\end{align*}
		By the estimate \eqref{normDiferenceLaplacian}, we have
		\begin{align*}
			\beta_{j}u_{j}^{2} &= \beta_{j}\left(\tilde{u}_{j}+u_{j}^{\epsilon}+O(c_{j}^{-2})\right)^2\\
			&\times \left(1 - \frac{1}{c_{j}^2}\Big[-\frac{1}{8\pi^2}\ln\epsilon+\frac{1}{16\pi^{2}}+K_0+\alpha(\gamma+1)\|G\|^{2}_{2} +o_{j}(\epsilon)+o_{\epsilon}(1)\Big]\right) \\
			&= \beta_{j}\tilde{u}_{j}^{2}\left(1 + \frac{u_{j}^{\epsilon}}{c_j}+O(c_{j}^{-3})\right)^{2}\\
			& \times \left(1 - \frac{1}{c_{j}^2}\Big[G(\epsilon)+\frac{1}{16\pi^{2}}+\alpha(\gamma+1)\|G\|^{2}_{2} +o_{j}(\epsilon)+o_{\epsilon}(1)\Big]\right). 
		\end{align*}
		Notice that 
		$$
			\lim\limits_{j \rightarrow \infty}\frac{\tilde{u}_{j}(x_j+r_j x)}{c_{j}} = 1,
		$$
		and since
		$$
			\tilde{u}_{j}(x_j+r_j x)u_{j}(\epsilon) \rightarrow G(\epsilon)
		$$
		we get
	\begin{align*}
			\beta_ju_j^2 & =\beta_j\tilde{u}_j^2\left(1+\frac{1}{c_j^2}\left(G(\epsilon)+\frac{1}{16 \pi^2}\right)+O\left(\frac{1}{c_j^3}\right)\right)^2 \\
			& \times \left(1 - \frac{1}{c_{j}^2}\Big[G(\epsilon)+\frac{1}{16\pi^{2}}+\alpha(\gamma+1)\|G\|^{2}_{2} +o_{j}(\epsilon)+o_{\epsilon}(1)\Big]\right) \\
			& =\beta_j \tilde{u}_j^2\left(1+\frac{2 G(\epsilon)}{c_j^2}+\frac{1}{8 \pi^2 c_j^2}-\frac{G(\epsilon)+\frac{1}{16 \pi^2}+O_j(\epsilon)}{c_j^2}\right) \\
			& =\beta_j \tilde{u}_j^2+\beta_j G(\epsilon)+\frac{\beta_j}{16 \pi^2}+O_j(\epsilon).
		\end{align*}
		Hence, 
	\begin{align*}
			& \lim _{\hat{R} \rightarrow \infty} \lim _{j \rightarrow \infty} \int_{B_{\hat{R} r_j}\left(x_j\right)}\left(\operatorname{e}^ {\beta_j \tilde{\zeta_{j}}u_j^2}-1\right) \mathrm{d} x \\
			& \leq \lim _{\hat{R} \rightarrow \infty} \lim _{j \rightarrow \infty} \int_{B_{\hat{R} r_j}\left(x_j\right)}\operatorname{e}^ {\beta_j \tilde{\zeta_{j}}u_j^2}\mathrm{d} x	\\
			& \leq \lim _{\hat{R} \rightarrow \infty} \lim _{j \rightarrow \infty} \operatorname{e}^{32 \pi^2 G(\epsilon)+2+o_\epsilon(1)} \int_{B_{\tilde{R} r_j}} \operatorname{e}^{\beta_j \tilde{\zeta_{j}}\tilde{u}_j^2} \mathrm{d} x \\
		& = \operatorname{e}^{32 \pi^2 G(\epsilon)+2+o_\epsilon(1)} \left|B_\epsilon\right| \frac{1}{3}\operatorname{e}^{-\frac{1}{3}}\\
		& =\operatorname{e}^{-4\ln \epsilon+32 \pi^2 K_0+32\pi^2h(\epsilon)+2+o_\epsilon(1)} \left|B_\epsilon\right| \frac{1}{3}\operatorname{e}^{-\frac{1}{3}}\\
		& =\frac{\pi^2}{6} \operatorname{e}^{\frac{5}{3}+32 \pi^2 K_0}+o_{\epsilon}(1).
	\end{align*}
	Letting $\epsilon \rightarrow 0$, we have
	$$
	AD(4,2, 32\pi^2,  \alpha, \gamma) \leq \frac{\pi^2}{6} \operatorname{e}^{\frac{5}{3}+32 \pi^2 K_0}.
	$$
\end{proof}	
	\section{The test function computation}
	In this section, we complete the proof of Theorem~\ref{thm-attainn=4} by showing that 
	\begin{equation}\label{AD>e}
	AD(4,2, 32\pi^2,  \alpha, \gamma) >\frac{\pi^2}{6} \operatorname{e}^{\frac{5}{3}+32 \pi^2 K_0},
	\end{equation}
for $\alpha>0 $ small enough and $\gamma\le \gamma_0$, which together with  Lemma~\ref{cAD<e} completes the proof. In order to get \eqref{AD>e}, following some ideas in  \cite{luluzhu20,NguyenVanHoang2019}, we use some test function computations.  Let us define
	$$
	\phi_{\varepsilon}=
        \begin{cases}
    		C+\frac{1}{C}\left[a-\frac{1}{16 \pi^2}\psi\Big(\frac{|x|}{\epsilon}\Big)+K_{0}+h(x)+b|x|^2\right],\;         &\text { if }|x| \leq L \epsilon, \\
    		\frac{G(x)}{C}, \quad &\text { if }|x| \geq L \epsilon
	\end{cases}
	$$	
where $\psi(s)=\ln(1+\frac{\pi}{\sqrt{6}}s^2)$ and 	 $L, C, a, b$ are functions of $\varepsilon$ such that
 \begin{enumerate}[label=(\roman*)]
     \item  $\epsilon=\exp (-L), \frac{1}{C^2}=O\left(\frac{1}{L}\right)$ as $\epsilon \rightarrow 0$ 
     \item $a=-\frac{1}{8 \pi^2} \ln (L \epsilon)-C^2+\frac{1}{16 \pi^2}\psi(L)-b  L^2 \epsilon^2$; 
     \item $b=-\frac{1}{16 \pi^2 L^2 \epsilon^2\left(1+\frac{\pi}{\sqrt{6}} L^2\right)}$.
 \end{enumerate}
 Note that, by $(i)$
  \begin{equation}\label{psi-4ln}
  \begin{aligned}
   \psi(L)-2\ln (L \epsilon) &=\ln \left(\frac{\frac{\pi}{\sqrt{6}}}{\epsilon^2}\right)+ \ln \left(1+\frac{\sqrt{6}}{\pi} \frac{1}{L^2}\right)\\
   &= \ln\frac{\pi}{\sqrt{6}\epsilon^2}+O\left(\frac{1}{\ln^2\epsilon}\right).
  \end{aligned}
  \end{equation}
 Hence, from $(i)$-$(ii)$
 \begin{equation}\label{CA}
 a=-C^2+\frac{1}{16\pi^2} \ln\frac{\pi}{\sqrt{6}\epsilon^2}+O\left(\frac{1}{\ln^2\epsilon}\right).
 \end{equation}
It was shown in  \cite{luluzhu20}  that  $\phi_{\epsilon} \in H^2\left(\mathbb{R}^4\right)$.  By Lemma \ref{GreenFunctionLemma}, we have	
	\begin{equation}\label{phi-fullnorm}
	\begin{aligned}
		&\int_{\mathbb{R}^4 \backslash B_{L \epsilon}(0)}\left(\left|\Delta \phi_{\epsilon}\right|^2 +\left|\phi_{\epsilon}\right|^2\right) \mathrm{d} x  =\frac{1}{C^2} \int_{\mathbb{R}^4 \backslash B_{L \epsilon}(0)}\left(|\Delta G|^2+|G|^2\right) \mathrm{d}  x \\
		&=\frac{1}{C^2}\left[-\frac{1}{8\pi^2} \ln L\epsilon -\frac{1}{16\pi^2}+K_{0} +\alpha(\gamma+1)\|G\|^{2}_{2} + O(L\epsilon)\right]
	\end{aligned}
	\end{equation}
	and by \cite{luluzhu20,LuYangEP}, it follows
		$$
	\int_{B_{L \epsilon}(0)}\left|\Delta \phi_{\epsilon}\right|^2  \mathrm{d}  x=\frac{1}{96 \pi^2 C^2}\left[6\psi(L)+1+O\left(\frac{1}{\ln ^2 \epsilon}\right)\right].
	$$
	Notice that
	\begin{equation}
	\label{normL2phi}
	 \begin{aligned}
   & 	\int_{B_{L \epsilon}(0)}\left|\phi_{\epsilon}\right|^2  \mathrm{d} x=\frac{1}{C^2}O\left((L \epsilon)^4 C^4\right)\\
   &\int_{\mathbb{R}^{4}\setminus B_{L \epsilon}(0)}\left|\phi_{\epsilon}\right|^2  \mathrm{d} x =\frac{1}{C^2}\int_{\mathbb{R}^{4}\setminus B_{L \epsilon}(0)}|G|^2\mathrm{d}x=\frac{1}{C^2}\left[\|G\|^2_{2}+O\Big((-L\epsilon\ln(L\epsilon))^{4}\Big)\right]
        \end{aligned}
	\end{equation}
	where we have used in the last identity the expression of $G$ given in Lemma~\ref{GreenFunctionLemma}.
 Therefore        
	\begin{align*}
		& \int_{\mathbb{R}^4}\left(\left|\Delta \phi_{\epsilon}\right|^2+\left|\phi_{\epsilon}\right|^2\right) \mathrm{d} x 
		 =\frac{1}{32 \pi^2 C^2}\left[2\psi(L)-4 \ln (L \epsilon)+O\left(\frac{1}{\ln ^2 \epsilon}\right)\right] \\
		&+ \frac{1}{32 \pi^2 C^2}\left[-\frac{5}{3}+32 \pi^2 K_{0}+32\pi^2\alpha(\gamma+1)\|G\|^{2}_{2}\right] \\
		& =\frac{1}{32 \pi^2 C^2}\left[ 2\ln\frac{\pi}{\sqrt{6}\epsilon^2}-\frac{5}{3}+32 \pi^2 K_{0}+32\pi^2\alpha(\gamma+1)\|G\|^{2}_{2}+O\left(\frac{1}{\ln^2\epsilon}\right)\right].
	\end{align*}
By setting  $\int_{\mathbb{R}^4}\left(\left|\Delta \phi_{\varepsilon}\right|^2+\left|\phi_{\varepsilon}\right|^2\right)\mathrm{d} x=1$ we get
	\begin{align}\label{C2=}
	32 \pi^2 C^2=2 \ln\frac{\pi}{\sqrt{6}\epsilon^2}-\frac{5}{3}+32 \pi^2 K_{0}+32\pi^2\alpha(\gamma+1)\|G\|^{2}_{2}+O\left(\frac{1}{\ln^2\epsilon}\right)
	\end{align}
 and so
	\begin{equation}\label{C2aproxim}
	      \begin{aligned}
	 C^2 =\alpha(\gamma+1)\|G\|^{2}_{2}+\frac{1}{16\pi^2}\ln\frac{\pi}{\sqrt{6}\epsilon^2}-\frac{5}{96\pi^2}+ K_{0}+O\left(\frac{1}{\ln ^2 \epsilon}\right)= O(-\ln \epsilon).
 \end{aligned}
	\end{equation}
In addition, since $\frac{1+\alpha t^2}{1-\alpha\gamma t^2}=1+\alpha(\gamma+1)t^2+O(t^4)$, as $t\to 0$ and, from \eqref{normL2phi} we have $\|\phi_{\epsilon}\|^{2}_{2}=O\Big(\frac{1}{C^2}\Big)\to 0$ as $\epsilon\to 0$. So, by using \eqref{normL2phi} again we can write 
\begin{equation}\label{Druet}
\begin{aligned}
\frac{1+\alpha\|\phi_{\epsilon}\|^{2}_{2}}{1-\alpha\gamma\|\phi_{\epsilon}\|^{2}_2}&=1+\alpha(\gamma+1)\|\phi_{\epsilon}\|^{2}_{2}+O(\|\phi_{\epsilon}\|^{4}_{2})\\
&= 1+\frac{\alpha(\gamma+1)\|G\|^2_{2}}{C^2}+O\Big(C^{-2}(-L\epsilon\ln(L\epsilon))^{4}\Big)+O\Big(C^{-4}\Big).
\end{aligned}
\end{equation}
 Firstly, since $\operatorname{e}^{t}-1\ge t$, for $t\ge 0$, from \eqref{normL2phi} and \eqref{Druet}, we can write 
 \begin{equation}\label{fora-BL}
 \begin{aligned}
  \int_{\mathbb{R}^{4}\setminus B_{L\epsilon}}\left(\operatorname{e}^{32\pi^2\Big(\frac{1+\alpha\|\phi_{\epsilon}\|^{2}_{2}}{1-\alpha\gamma\|\phi_{\epsilon}\|^{2}_2}\big)\phi^{2}_{\epsilon}}-1\right)\mathrm{d}x & \ge 32\pi^2\left(\frac{1+\alpha\|\phi_{\epsilon}\|^{2}_{2}}{1-\alpha\gamma\|\phi_{\epsilon}\|^{2}_2}\right) \int_{\mathbb{R}^{4}\setminus B_{L\epsilon}}\phi^{2}_{\epsilon}\mathrm{d}x\\
 &=\frac{32\pi^2}{C^2}\left(\frac{1+\alpha\|\phi_{\epsilon}\|^{2}_{2}}{1-\alpha\gamma\|\phi_{\epsilon}\|^{2}_2}\right)\left[\|G\|^2_{2}+O\Big((-L\epsilon\ln(L\epsilon))^{4}\Big)\right]\\
 & =\frac{32\pi^2\|G\|^2_{2}}{C^2}+O\Big(C^{-4}\Big).
 \end{aligned}
 \end{equation}
Now,  from \eqref{Druet}, we also  have  
\begin{equation}\label{druet-phi}
 \begin{aligned}
 \left(\frac{1+\alpha\|\phi_{\epsilon}\|^{2}_{2}}{1-\alpha\gamma\|\phi_{\epsilon}\|^{2}_2}\right)\phi^{2}_{\epsilon}& = \Big(1+\frac{\alpha(\gamma+1)\|G\|^2_{2}}{C^2}\Big)\phi_{\epsilon}^2+\left[O\Big((-L\epsilon\ln(L\epsilon))^{4}\Big)+O\Big(C^{-2}\Big)\right]\frac{\phi_{\epsilon}^2}{C^2}.
 \end{aligned}
 \end{equation}
   Next, we shall estimate each term on the right hand side of \eqref{druet-phi} on $B_{L \varepsilon}$.  Firstly, 
by $(ii)$, \eqref{CA} and \eqref{C2aproxim}, for any $x \in B_{L \varepsilon}$, 
\begin{equation}\label{phi>}
\begin{aligned}
		& \phi_{\epsilon}^2  \geq C^2+2\left(a-\frac{1}{16 \pi^2}\psi\Big(\frac{|x|}{\epsilon}\Big) +K_{0}+h(x)+b|x|^2\right)\\
		&= C^2+2a -\frac{1}{8\pi^2}\psi\Big(\frac{|x|}{\epsilon}\Big)+2(K_0+h(x)+b|x|^2)\\
		& =-\alpha(\gamma+1)\|G\|^{2}_{2}+\frac{1}{16\pi^2}\ln\frac{\pi}{\sqrt{6}\epsilon^2}-\frac{5}{96\pi^2}+K_{0}+O\left(\frac{1}{\ln ^2 \varepsilon}\right)\\
		&+ \frac{10}{96\pi^2}-\frac{1}{8\pi^2}\psi\Big(\frac{|x|}{\epsilon}\Big)+2(h(x)+b|x|^2)\\
		&= C^2-2\alpha(\gamma+1)\|G\|^{2}_{2}+\frac{10}{96\pi^2}+2z\Big(\frac{x}{\epsilon}\Big)+O\left(\frac{1}{\ln ^2 \epsilon}\right).
  \end{aligned} 
  \end{equation}
  where $z(x)$ is given by \eqref{zform} and  we have used in the last identity \eqref{phi>}  the fact that $h\in C^3$ with $h(0)=0$, which implies that $2(h(x)+b|x|^2)=O(L\epsilon)=O\Big(\frac{1}{\ln^2\epsilon}\Big)$ for 
  $x \in B_{L \varepsilon}$.  From \eqref{phi>}, we obtain 
 \begin{equation}\nonumber
 \begin{aligned}
  \Big(1+\frac{\alpha(\gamma+1)\|G\|^2_{2}}{C^2}\Big)\phi_{\epsilon}^2 &\ge \Big(1+\frac{\alpha(\gamma+1)\|G\|^2_{2}}{C^2}\Big)2z\Big(\frac{x}{\epsilon}\Big)+C^2-\alpha(\gamma+1)\|G\|^{2}_{2}+\frac{10}{96\pi^2}\\
 &-\frac{2\alpha^{2}(\gamma+1)^2\|G\|^4_{2}+\frac{10}{96\pi^2}}{C^2}+O\Big(\frac{1}{\ln ^2 \epsilon}\Big)
 \end{aligned}
 \end{equation}
 and from \eqref{C2aproxim}
 \begin{equation}\label{druet-term1}
 \begin{aligned}
    \Big(1+\frac{\alpha(\gamma+1)\|G\|^2_{2}}{C^2}\Big)\phi_{\epsilon}^2&  \ge 
  \Big(1+\frac{\alpha(\gamma+1)\|G\|^2_{2}}{C^2}\Big)2z\Big(\frac{x}{\epsilon}\Big)\\
  &+\frac{5}{96\pi^2}+K_0+\frac{1}{16\pi^2}\ln\frac{\pi}{\sqrt{6}\epsilon^2}+O\Big(\frac{1}{\ln^2 \epsilon}\Big)
 \end{aligned}
 \end{equation}
 Now, for $x\in B_{L\epsilon}$
 \begin{align*}
 \Big|\psi\Big(\frac{|x|}{\epsilon}\Big)\Big|\le \ln\Big(1+\frac{\pi}{\sqrt{6}}L^2\Big)=2\ln L+ \ln\Big(\frac{\pi}{\sqrt{6}}+\frac{1}{L^2}\Big)=O\Big(\ln (-\ln\epsilon)\Big).
 \end{align*}
 So, by definition of $\phi_{\epsilon}$, $(iii)$,  \eqref{CA} and \eqref{C2aproxim},  on $B_{L\epsilon}$ we get
 \begin{equation}\nonumber
 \begin{aligned}
\Big|\frac{\phi_{\epsilon}}{C}\Big|&=\Big|1+\frac{1}{C^2}\left[a-\frac{1}{16 \pi^2}\psi\Big(\frac{|x|}{\epsilon}\Big)+K_{0}+h(x)+b|x|^2\right]\Big|\\
& \le O\Big(C^{-2}\ln (-\ln\epsilon)\Big)+O\Big(C^{-2}\Big)\to 0,\;\;\mbox{as}\;\; \epsilon\to 0.
 \end{aligned}
 \end{equation}
It follows that 
 $$
 \left[O\Big((-L\epsilon\ln(L\epsilon))^{4}\Big)+O\Big(C^{-2}\Big)\right]\frac{\phi_{\varepsilon}^2}{C^2}=O\Big(\frac{1}{\ln ^2\epsilon}\Big).
 $$
 So, by combining \eqref{druet-phi} with \eqref{druet-term1}, for $x\in B_{L\epsilon}$ we can write
 \begin{equation}\label{druet-final}
 \begin{aligned}
    32\pi^2\left(\frac{1+\alpha\|\phi_{\epsilon}\|^{2}_{2}}{1-\alpha\gamma\|\phi_{\epsilon}\|^{2}_2}\right)\phi^{2}_{\epsilon} & \ge \frac{5}{3}+32\pi^2K_0+\ln\frac{\pi^2}{6}+ \ln\frac{1}{\epsilon^4}\\
   &+\Big(1+\frac{\alpha(\gamma+1)\|G\|^2_{2}}{C^2}\Big)64\pi^2z\Big(\frac{x}{\epsilon}\Big)+ O\Big(\frac{1}{\ln^2 \epsilon}\Big).
 \end{aligned}
 \end{equation}
 Hence, 
  \begin{equation}\label{quase-dentro-BL}
 \begin{aligned}
 &  \int_{B_{L\epsilon}}\left(\operatorname{e}^{32\pi^2\big(\frac{1+\alpha\|\phi_{\epsilon}\|^{2}_{2}}{1-\alpha\gamma\|\phi_{\epsilon}\|^{2}_2}\big)\phi^{2}_{\epsilon}}-1\right)\mathrm{d}x 
  = \int_{B_{L\epsilon}}\operatorname{e}^{32\pi^2\big(\frac{1+\alpha\|\phi_{\epsilon}\|^{2}_{2}}{1-\alpha\gamma\|\phi_{\epsilon}\|^{2}_2}\big)\phi^{2}_{\epsilon}}\mathrm{d}x +O((\epsilon L)^4)\\
 & \ge \frac{\pi^2}{6\epsilon^4}\operatorname{e}^{\frac{5}{3}+32\pi^2K_0} \int_{B_{L\epsilon}}\Big(1+\frac{\pi}{\sqrt{6}}\frac{|x|^2}{\epsilon^2}\Big)^{-4\big(1+\frac{\alpha(\gamma+1)\|G\|^2_{2}}{C^2}\big)}\mathrm{d}x+ O\Big(\frac{1}{\ln^{2}\epsilon}\Big)\\
 &=\frac{\pi^2}{6}\operatorname{e}^{\frac{5}{3}+32\pi^2K_0} \int_{B_{L}}\Big(1+\frac{\pi}{\sqrt{6}}|x|^2\Big)^{-4\big(1+\frac{\alpha(\gamma+1)\|G\|^2_{2}}{C^2}\big)}\mathrm{d}x+ O\Big(\frac{1}{\ln^{2}\epsilon}\Big).
 \end{aligned}
 \end{equation}
 By Taylor expansion (cf. \cite[Lemma 19]{JJ2023}) we have 
 \begin{equation}
 \frac{6\Gamma(2) \Gamma(2+y)}{\Gamma(4+y)}=1-\frac{5y}{6}+O(y^2), \;\; \mbox{as}\;\; y\to 0.
 \end{equation}
 Thus,  by setting  $y=y_{\epsilon}:=\frac{4\alpha(\gamma+1)\|G\|^2_{2}}{C^2}$ we obtain
 \begin{equation}
 \begin{aligned}
  \int_{B_{L}}\Big(1+\frac{\pi}{\sqrt{6}}|x|^2\Big)^{-4\big(1+\frac{\alpha(\gamma+1)\|G\|^2_{2}}{C^2}\big)}\mathrm{d}x& = \int_{B_{L}}\frac{1}{(1+\frac{\pi}{\sqrt{6}}|x|^2)^{4+y_{\epsilon}}}\mathrm{d}x\\
  &=6\int_{0}^{\frac{\pi}{\sqrt{6}}L^2}\frac{s}{(1+s)^{4+y_{\epsilon}}}\mathrm{d}s\\
  &=6\int_{0}^{\infty}\frac{s}{(1+s)^{4+y_{\epsilon}}}\mathrm{d}s+O(L^{-2})\\
  &=\frac{6\Gamma(2) \Gamma(2+y_{\epsilon})}{\Gamma(4+y_{\epsilon})}+O(L^{-2})\\
  & = 1-\frac{10}{3}\frac{\alpha(\gamma+1)\|G\|^2_{2}}{C^2}+O(L^{-2})+O(C^{-4})
 \end{aligned}
 \end{equation}
This together with \eqref{quase-dentro-BL} yields
  \begin{equation}\label{dentro-BL}
 \begin{aligned}
 &  \int_{B_{L\epsilon}}\left(\operatorname{e}^{32\pi^2\big(\frac{1+\alpha\|\phi_{\epsilon}\|^{2}_{2}}{1-\alpha\gamma\|\phi_{\epsilon}\|^{2}_2}\big)\phi^{2}_{\epsilon}}-1\right)\mathrm{d}x 
 \ge \frac{\pi^2}{6}\operatorname{e}^{\frac{5}{3}+32\pi^2K_0}\\
 & +\frac{\pi^2}{6}\operatorname{e}^{\frac{5}{3}+32\pi^2K_0}\frac{1}{C^2}\left[-\frac{10}{3}\alpha(\gamma+1)\|G\|^2_{2}+O(C^2L^{-2})+O(C^{-2})+O\Big(\frac{C^2}{\ln^{2}\epsilon}\Big)\right].
 \end{aligned}
 \end{equation}
 By combining \eqref{dentro-BL} with \eqref{fora-BL}, it follows that 
 \begin{equation}\label{fullR4Teste}
 \begin{aligned}
 &  \int_{\mathbb{R}^4}\left(\operatorname{e}^{32\pi^2\big(\frac{1+\alpha\|\phi_{\epsilon}\|^{2}_{2}}{1-\alpha\gamma\|\phi_{\epsilon}\|^{2}_2}\big)\phi^{2}_{\epsilon}}-1\right)\mathrm{d}x 
 \ge \frac{\pi^2}{6}\operatorname{e}^{\frac{5}{3}+32\pi^2K_0}\\
 & +\frac{\pi^2}{6}\operatorname{e}^{\frac{5}{3}+32\pi^2K_0}\frac{1}{C^2}\left[-\frac{10}{3}\alpha(\gamma+1)\|G\|^2_{2}+O(C^2L^{-2})+O(C^{-2})+O\Big(\frac{C^2}{\ln^{2}\epsilon}\Big)\right]\\
 &+\frac{32\pi^2\|G\|^2_{2}}{C^2}+O\Big(C^{-4}\Big)\\
 &= \frac{\pi^2}{6}\operatorname{e}^{\frac{5}{3}+32\pi^2K_0}\\
 & +\frac{\pi^2}{6C^2}\operatorname{e}^{\frac{5}{3}+32\pi^2K_0}\left[\Big(192\operatorname{e}^{-\frac{5}{3}-32\pi^2K_0}-\frac{10}{3}\alpha(\gamma+1)\Big)\|G\|^2_{2}+O(C^2L^{-2})+O(C^{-2})+O\Big(\frac{C^2}{\ln^{2}\epsilon}\Big)\right]\\
 \end{aligned}
 \end{equation}
 Since $$O(C^2L^{-2})+O(C^{-2})+O\Big(\frac{C^2}{\ln^{2}\epsilon}\Big)\to 0 \;\;\mbox{as} \;\;\epsilon\to 0$$
and $\gamma\le \gamma_0$ is bounded,  by choosing $\alpha_0>0$ small such that  $$192\operatorname{e}^{-\frac{5}{3}-32\pi^2K_0}-\frac{10}{3}\alpha_0(\gamma+1)>0$$ we conclude that \eqref{AD>e} holds for all $0\le \alpha<\alpha_0$.
\bibliographystyle{siam}
 \bibliography{Ref}

\begin{thebibliography}{10}

\bibitem{adachitanaka2000}
{\sc S.~Adachi and K.~Tanaka}, {\em Trudinger type inequalities in
  {$\mathbf{R}^N$} and their best exponents}, Proc. Amer. Math. Soc., 128
  (2000), pp.~2051--2057.

\bibitem{Adams1988}
{\sc D.~R. Adams}, {\em A sharp inequality of {J}. {M}oser for higher order
  derivatives}, Ann. of Math. (2), 128 (1988), pp.~385--398.

\bibitem{AdimurthiDruet2004}
{\sc Adimurthi and O.~Druet}, {\em Blow-up analysis in dimension 2 and a sharp
  form of {T}rudinger-{M}oser inequality}, Comm. Partial Differential
  Equations, 29 (2004), pp.~295--322.

\bibitem{Baird1}
{\sc P.~Baird, A.~Fardoun, and R.~Regbaoui}, {\em {$Q$}-curvature flow on
  4-manifolds}, Calc. Var. Partial Differential Equations, 27 (2006),
  pp.~75--104.

\bibitem{Baird2}
\leavevmode\vrule height 2pt depth -1.6pt width 23pt, {\em Prescribed
  {$Q$}-curvature on manifolds of even dimension}, J. Geom. Phys., 59 (2009),
  pp.~221--233.

\bibitem{BN}
{\sc H.~Br\'{e}zis and L.~Nirenberg}, {\em Positive solutions of nonlinear
  elliptic equations involving critical {S}obolev exponents}, Comm. Pure Appl.
  Math., 36 (1983), pp.~437--477.

\bibitem{Cao}
{\sc D.~M. Cao}, {\em Nontrivial solution of semilinear elliptic equations with
  critical exponent in r}, Communications in Partial Differential Equations, 17
  (1992), pp.~407--435.

\bibitem{Carleson-Chang}
{\sc L.~Carleson and S.-Y.~A. Chang}, {\em On the existence of an extremal
  function for an inequality of {J}. {M}oser}, Bull. Sci. Math. (2), 110
  (1986), pp.~113--127.

\bibitem{CST}
{\sc D.~Cassani, F.~Sani, and C.~Tarsi}, {\em Equivalent {M}oser type
  inequalities in {$\mathbb{R}^2$} and the zero mass case}, J. Funct. Anal.,
  267 (2014), pp.~4236--4263.

\bibitem{luluzhu20}
{\sc L.~Chen, G.~Lu, and M.~Zhu}, {\em Existence and nonexistence of extremals
  for critical {A}dams inequalities in {$\mathbb{R}^4$} and {T}rudinger-{M}oser
  inequalities in {$\mathbb{R}^2$}}, Adv. Math., 368 (2020), pp.~107143, 61.

\bibitem{JJ2015}
{\sc J.~F. de~Oliveira and J.~a.~M. do~\'O}, {\em Concentration-compactness
  principle and extremal functions for a sharp {T}rudinger-{M}oser inequality},
  Calc. Var. Partial Differential Equations, 52 (2015), pp.~125--163.

\bibitem{JJ2023}
{\sc J.~F. de~Oliveira and J.~M. do~{\'O}}, {\em On a sharp inequality of
  {Adimurthi}-{Druet} type and extremal functions}, Calc. Var. Partial Differ.
  Equ., 62 (2023), p.~40.
\newblock Id/No 162.

\bibitem{deOliveiraMacedo2022}
{\sc J.~F.~A. de~Oliveira and A.~C. Macedo}, {\em Estimate for concentration
  level of the {A}dams functional and extremals for {A}dams-type inequality},
  J. Funct. Anal., 283 (2022), pp.~Paper No. 109633, 32.

\bibitem{DoOMana1}
{\sc M.~de~Souza and J.~a.~M. do~\'O}, {\em A sharp {T}rudinger-{M}oser type
  inequality in {$\mathbb{R}^2$}}, Trans. Amer. Math. Soc., 366 (2014),
  pp.~4513--4549.

\bibitem{DelaTorre}
{\sc A.~DelaTorre and G.~Mancini}, {\em Improved {A}dams-type inequalities and
  their extremals in dimension $2m$}, Commun. Contemp. Math.,  (2020).

\bibitem{DengLi2007}
{\sc Y.~Deng and Y.~Li}, {\em Exponential decay of the solutions for nonlinear
  biharmonic equations}, Commun. Contemp. Math., 9 (2007), pp.~753--768.

\bibitem{DoOMana2}
{\sc J.~a.~M. do~\'O and M.~de~Souza}, {\em A sharp inequality of
  {T}rudinger-{M}oser type and extremal functions in
  {$H^{1,n}(\mathbb{R}^n)$}}, J. Differential Equations, 258 (2015),
  pp.~4062--4101.

\bibitem{DoO}
{\sc J.~M.~B. do~{\'O}}, {\em {$N$-Laplacian equations in $\mathbb{R}^N$ with
  critical growth}}, Abstract and Applied Analysis, 2 (1997), pp.~301 -- 315.

\bibitem{Flucher1992}
{\sc M.~Flucher}, {\em Extremal functions for the {T}rudinger-{M}oser
  inequality in {$2$} dimensions}, Comment. Math. Helv., 67 (1992),
  pp.~471--497.

\bibitem{FontanaLuigi2015SAaM}
{\sc L.~Fontana and C.~Morpurgo}, {\em Sharp {A}dams and {M}oser-{T}rudinger
  inequalities on {$\mathbb{R}^n$} and other spaces of infinite measure},
  arXiv.org,  (2015).

\bibitem{FontanaMorpurgo2018}
\leavevmode\vrule height 2pt depth -1.6pt width 23pt, {\em Sharp exponential
  integrability for critical {R}iesz potentials and fractional {L}aplacians on
  {$\mathbb{R}^n$}}, Nonlinear Anal., 167 (2018), pp.~85--122.

\bibitem{MR2667016}
{\sc F.~Gazzola, H.-C. Grunau, and G.~Sweers}, {\em Polyharmonic boundary value
  problems}, vol.~1991 of Lecture Notes in Mathematics, Springer-Verlag,
  Berlin, 2010.
\newblock Positivity preserving and nonlinear higher order elliptic equations
  in bounded domains.

\bibitem{HMT}
{\sc J.~A. Hempel, G.~R. Morris, and N.~S. Trudinger}, {\em On the sharpness of
  a limiting case of the {S}obolev imbedding theorem}, Bull. Austral. Math.
  Soc., 3 (1970), pp.~369--373.

\bibitem{Ibra}
{\sc S.~Ibrahim, N.~Masmoudi, and K.~Nakanishi}, {\em Trudinger-{M}oser
  inequality on the whole plane with the exact growth condition}, J. Eur. Math.
  Soc. (JEMS), 17 (2015), pp.~819--835.

\bibitem{Ishiwata2011}
{\sc M.~Ishiwata}, {\em Existence and nonexistence of maximizers for
  variational problems associated with {T}rudinger-{M}oser type inequalities in
  {$\mathbb{R}^N$}}, Math. Ann., 351 (2011), pp.~781--804.

\bibitem{YUDO61}
{\sc V.~I. Judovi\v{c}}, {\em Some estimates connected with integral operators
  and with solutions of elliptic equations}, Dokl. Akad. Nauk SSSR, 138 (1961),
  pp.~805--808.

\bibitem{kavian}
{\sc O.~Kavian}, {\em Introduction \`a la th\'{e}orie des points critiques et
  applications aux probl\`emes elliptiques}, vol.~13 of Math\'{e}matiques \&
  Applications (Berlin) [Mathematics \& Applications], Springer-Verlag, Paris,
  1993.

\bibitem{lamlu2012}
{\sc N.~Lam and G.~Lu}, {\em Sharp {A}dams type inequalities in {S}obolev
  spaces {$W^{m,\frac{n}{m}} (\mathbb{R}^n)$} for arbitrary integer {$m$}}, J.
  Differential Equations, 253 (2012), pp.~1143--1171.

\bibitem{LamLuMAA2012}
\leavevmode\vrule height 2pt depth -1.6pt width 23pt, {\em Sharp singular
  {A}dams inequalities in high order {S}obolev spaces}, Methods Appl. Anal., 19
  (2012), pp.~243--266.

\bibitem{lamlufree}
\leavevmode\vrule height 2pt depth -1.6pt width 23pt, {\em A new approach to
  sharp {M}oser-{T}rudinger and {A}dams type inequalities: a rearrangement-free
  argument}, J. Differential Equations, 255 (2013), pp.~298--325.

\bibitem{LamLuChapter}
\leavevmode\vrule height 2pt depth -1.6pt width 23pt, {\em Sharp singular
  {T}rudinger-{M}oser-{A}dams type inequalities with exact growth}, in
  Geometric methods in {PDE}'s, vol.~13 of Springer INdAM Ser., Springer, Cham,
  2015, pp.~43--80.

\bibitem{LLZ}
{\sc N.~Lam, G.~Lu, and L.~Zhang}, {\em Equivalence of critical and subcritical
  sharp {T}rudinger-{M}oser-{A}dams inequalities}, Rev. Mat. Iberoam., 33
  (2017), pp.~1219--1246.

\bibitem{LLZ2}
{\sc N.~Lam, G.~Lu, and L.~Zhang}, {\em Existence and nonexistence of extremal
  functions for sharp trudinger-moser inequalities}, Advances in Mathematics,
  352 (2019), pp.~1253--1298.

\bibitem{Lenzmann}
{\sc E.~Lenzmann and J.~Sok}, {\em {A Sharp Rearrangement Principle in Fourier
  Space and Symmetry Results for PDEs with Arbitrary Order}}, International
  Mathematics Research Notices, 2021 (2020), pp.~15040--15081.

\bibitem{LiLuZhu2018}
{\sc J.~Li, G.~Lu, and M.~Zhu}, {\em Concentration-compactness principle for
  {T}rudinger-{M}oser inequalities on {H}eisenberg groups and existence of
  ground state solutions}, Calc. Var. Partial Differential Equations, 57
  (2018), pp.~Paper No. 84, 26.

\bibitem{RufLi2008}
{\sc Y.~Li and B.~Ruf}, {\em A sharp {T}rudinger-{M}oser type inequality for
  unbounded domains in {$\mathbb{R}^n$}}, Indiana Univ. Math. J., 57 (2008),
  pp.~451--480.

\bibitem{Lin1998}
{\sc C.-S. Lin}, {\em A classification of solutions of a conformally invariant
  fourth order equation in {${\bf R}^n$}}, Comment. Math. Helv., 73 (1998),
  pp.~206--231.

\bibitem{Lin1996}
{\sc K.-C. Lin}, {\em Extremal functions for {M}oser's inequality}, Trans.
  Amer. Math. Soc., 348 (1996), pp.~2663--2671.

\bibitem{Lions1984I}
{\sc P.-L. Lions}, {\em The concentration-compactness principle in the calculus
  of variations. {T}he limit case. {I}}, Rev. Mat. Iberoamericana, 1 (1985),
  pp.~145--201.

\bibitem{Lions1984II}
\leavevmode\vrule height 2pt depth -1.6pt width 23pt, {\em The
  concentration-compactness principle in the calculus of variations. {T}he
  limit case. {II}}, Rev. Mat. Iberoamericana, 1 (1985), pp.~45--121.

\bibitem{LuTang}
{\sc G.~Lu and H.~Tang}, {\em Sharp {M}oser-{T}rudinger inequalities on
  hyperbolic spaces with exact growth condition}, J. Geom. Anal., 26 (2016),
  pp.~837--857.

\bibitem{LuYangEP}
{\sc G.~Lu and Y.~Yang}, {\em Adams' inequalities for bi-{L}aplacian and
  extremal functions in dimension four}, Adv. Math., 220 (2009),
  pp.~1135--1170.

\bibitem{LuYang}
\leavevmode\vrule height 2pt depth -1.6pt width 23pt, {\em Sharp constant and
  extremal function for the improved {M}oser-{T}rudinger inequality involving
  {$L^p$} norm in two dimension}, Discrete Contin. Dyn. Syst., 25 (2009),
  pp.~963--979.

\bibitem{Martinazzi2009}
{\sc L.~Martinazzi}, {\em A threshold phenomenon for embeddings of {$H^m_0$}
  into {O}rlicz spaces}, Calc. Var. Partial Differential Equations, 36 (2009),
  pp.~493--506.

\bibitem{MasmoudiSani2014}
{\sc N.~Masmoudi and F.~Sani}, {\em Adams' inequality with the exact growth
  condition in {$\mathbb{R}^4$}}, Comm. Pure Appl. Math., 67 (2014),
  pp.~1307--1335.

\bibitem{MasmoudiSani2015}
\leavevmode\vrule height 2pt depth -1.6pt width 23pt, {\em Trudinger-{M}oser
  inequalities with the exact growth condition in {$\mathbb{R}^N$} and
  applications}, Comm. Partial Differential Equations, 40 (2015),
  pp.~1408--1440.

\bibitem{Moser1970/71}
{\sc J.~Moser}, {\em A sharp form of an inequality by {N}. {T}rudinger},
  Indiana Univ. Math. J., 20 (1970/71), pp.~1077--1092.

\bibitem{NguyenVanHoang2019}
{\sc V.~H. Nguyen}, {\em Extremal functions for the {M}oser-{T}rudinger
  inequality of {A}dimurthi-{D}ruet type in {$W^{1,N}(\mathbb{R}^N)$}}, Commun.
  Contemp. Math., 21 (2019), pp.~1850023, 37.

\bibitem{OneilConvolution1963}
{\sc R.~O'Neil}, {\em Convolution operators and {$L(p,\,q)$} spaces}, Duke
  Math. J., 30 (1963), pp.~129--142.

\bibitem{Ozawa1995}
{\sc T.~Ozawa}, {\em On critical cases of sobolev's inequalities}, FREEDOM, 127
  (1995), p.~259.

\bibitem{Panda}
{\sc R.~Panda}, {\em On semilinear {N}eumann problems with critical growth for
  the {$n$}-{L}aplacian}, Nonlinear Anal., 26 (1996), pp.~1347--1366.

\bibitem{Peetre}
{\sc J.~Peetre}, {\em Espaces d'interpolation et th\'eor\`eme de {S}oboleff},
  Ann. Inst. Fourier (Grenoble), 16 (1966), pp.~279--317.

\bibitem{Pizzetti}
{\sc P.~Pizzetti}, {\em Sulla media dei valori che una funzione dei punti dello
  spazio assume sulla superficie della sfera}, Rend. Mat. Acc. Lincei, 18
  (1909), pp.~182--185.

\bibitem{P0H065}
{\sc S.~I. Pohozaev}, {\em The sobolev embedding in the case $pl = n$},
  Proceedings of the Technical Scientific Conference on Advances of Scientific
  Research, 138 (1965), pp.~158--170.

\bibitem{Ruf2005}
{\sc B.~Ruf}, {\em A sharp {T}rudinger-{M}oser type inequality for unbounded
  domains in {$\mathbb{R}^2$}}, J. Funct. Anal., 219 (2005), pp.~340--367.

\bibitem{RufSani2013}
{\sc B.~Ruf and F.~Sani}, {\em Sharp {A}dams-type inequalities in
  {$\mathbb{R}^n$}}, Trans. Amer. Math. Soc., 365 (2013), pp.~645--670.

\bibitem{Struwe1988}
{\sc M.~Struwe}, {\em Critical points of embeddings of {$H^{1,n}_0$} into
  {O}rlicz spaces}, Ann. Inst. H. Poincar\'{e} Anal. Non Lin\'{e}aire, 5
  (1988), pp.~425--464.

\bibitem{Tarsi2012}
{\sc C.~Tarsi}, {\em Adams' inequality and limiting {S}obolev embeddings into
  {Z}ygmund spaces}, Potential Anal., 37 (2012), pp.~353--385.

\bibitem{Tinta}
{\sc C.~Tintarev}, {\em Trudinger-{M}oser inequality with remainder terms}, J.
  Funct. Anal., 266 (2014), pp.~55--66.

\bibitem{Trudinger67}
{\sc N.~S. Trudinger}, {\em On imbeddings into {O}rlicz spaces and some
  applications}, J. Math. Mech., 17 (1967), pp.~473--483.

\bibitem{CernyCianchiHencl13}
{\sc R.~\v{C}ern\'{y}, A.~Cianchi, and S.~Hencl}, {\em
  Concentration-compactness principles for {M}oser-{T}rudinger inequalities:
  new results and proofs}, Ann. Mat. Pura Appl. (4), 192 (2013), pp.~225--243.

\bibitem{Yang2006}
{\sc Y.~Yang}, {\em Extremal functions for a sharp {M}oser-{T}rudinger
  inequality}, Internat. J. Math., 17 (2006), pp.~331--338.

\bibitem{Yang}
\leavevmode\vrule height 2pt depth -1.6pt width 23pt, {\em A sharp form of
  {M}oser-{T}rudinger inequality in high dimension}, J. Funct. Anal., 239
  (2006), pp.~100--126.

\bibitem{YangS}
\leavevmode\vrule height 2pt depth -1.6pt width 23pt, {\em Extremal functions
  for {T}rudinger-{M}oser inequalities of {A}dimurthi-{D}ruet type in dimension
  two}, J. Differential Equations, 258 (2015), pp.~3161--3193.

\bibitem{YangZhu}
{\sc Y.~Yang and X.~Zhu}, {\em Blow-up analysis concerning singular
  {T}rudinger-{M}oser inequalities in dimension two}, J. Funct. Anal., 272
  (2017), pp.~3347--3374.

\end{thebibliography}

\end{document}